\documentclass[reqno]{amsart}
\usepackage{graphicx}
\usepackage{amscd}
\usepackage{amsmath}
\usepackage{amsfonts}
\usepackage{amssymb}
\usepackage{cite}
\usepackage{xcolor}
\usepackage{longtable}
\usepackage{caption}
\usepackage{float}
\usepackage{multicol}        
\restylefloat{table}
\newtheorem{theorem}{Theorem}[section]
\theoremstyle{plain}

\newtheorem{corollary}{Corollary}[section]

\newtheorem{example}{Example}[section]

\newtheorem{lemma}{Lemma}[section]

\newtheorem{proposition}{Proposition}[section]
\newtheorem{remark}{Remark}[section]

\numberwithin{equation}{section}
\def\R {\mathbb{R}}
\def\T {\mathbb{T}}
\def\F {\mathbb{F}}
\def\N {\mathbb{N}}
\def\s {\mathcal{S}}
\def\C{{\mathbb C}}
\def\P{{\mathbb{P}}}

\def\K{{\boldsymbol{K}}}
\def\d{\mathfrak{d}}
\def\b{\mathfrak{b}}
\def\h{\mathfrak{h}}
\def\z{\mathfrak{z}}
\def\y{\mathfrak{y}}
\def\ss {\mathfrak{S}}

\begin{document}

\title[Functional Difference Equations]
{On a class of  functional difference equations: explicit\\
solutions, asymptotic behavior and  applications}

\author[ N. Vasylyeva]
{ Nataliya Vasylyeva}

\address{Institute of Applied Mathematics and Mechanics of NAS of Ukraine
\newline\indent
G.Batyuka str. 19, 84100 Sloviansk, Ukraine; and
\newline\indent
Institute of Hydromechanics of NAS of Ukraine
\newline\indent
Zhelyabova str. 8/4, 03057 Kyiv, Ukraine}
\email[N.Vasylyeva]{nataliy\underline{\ }v@yahoo.com}

\subjclass[2000]{Primary 39A06, 39B32; Secondary
35C15,35R11,35R35} \keywords{functional difference equations,
explicit solution, asymptotic, subdiffusion equation}

\begin{abstract}
For $\nu\in[0,1]$ and a complex parameter $\sigma,$  $Re\, \sigma>0,$ we discuss a linear inhomogeneous functional difference equation with variable coefficients on a complex plane $z\in\C$:
\[
(a_{1}\sigma+a_{2}\sigma^{\nu})\mathcal{Y}(z+\beta,\sigma)-\Omega(z)\mathcal{Y}(z,\sigma)=\F(z,\sigma), \quad\beta\in\R,\, \beta\neq 0,
\]
where $\Omega(z)$ and $\F(z)$ are given complex functions, while $a_{1}$ and $a_{2}$ are given real non-negative numbers. Under suitable conditions on the given functions and parameters, we construct explicit solutions of the equation and describe their asymptotic behavior as $|z|\to +\infty$. Some applications to the theory of   functional difference equations and to the theory of boundary value problems governed by subdiffusion in nonsmooth domains are then discussed.
\end{abstract}

\maketitle

\section{Introduction}
\label{s1}

\noindent Functional difference equations (\textbf{FDEs}) have
shown an incredible power to solve many problems arising in
mathematics, physics, biology and economics. In particular, they
give rise to many examples of applications including  population
models, molecular evolution models, price and quantity levels and
indices, utility theory, production functions, Fisher equation of
exchange, theory of multi-sectoral growth,  the famous problem of
the 'gambler's ruin', combinatorial problems encompassing the
recurrences generating many famous combinatorial numbers (e.g.
Fine, central Delannoy, Schr\"{o}der, counting directed animals)
(see \cite{Ag,BKS,CC,EGLLR,KP,TI,TF,Sc} and also references
therein). Still another way in which difference equations appear
is as recurrence relations in connection  with  special functions
or as numerical approximations to differential equations and to
initial-boundary value problems (especially if these problems are
stated in the domains with singular boundaries)
\cite{Ba,BV1,BF,DV,La,LK,Ka,HFI,SF,Va1}. Finding explicit
solutions and closure forms of  solutions for linear \textbf{FDEs}
using various approaches, including Mallin, Laplace
transformations, series and infinite products, special functions
(Barnes double Gamma-functions, Maliuzhnets' function,
Alexeiewsky's $G$ function),  is  one of the main research topics
related with asymptotic theory of analytic FDEs and symbolic
computations. Indeed, construction of solutions to  difference
equations with different  types of  coefficients (e.g. constant,
not transcendental, trigonometric, rational) has been extensively
studied  by many authors (see
\cite{Ba,BF,BV4,C,CH,FGH,HE,Ka,LK,L,La,SF,Va1,Va2} and the
references therein).

In this paper, for  fixed real  numbers  $\nu\in(0,1)$,  $a_{1}$, $a_{2}$, $\beta$,
we focus on a linear inhomogeneous functional difference equation with a variable coefficient in the unknown complex-valued function $\mathcal{Y}=\mathcal{Y}(z,\sigma),$
\begin{equation}\label{i.1}
(a_{1}\sigma+a_{2}\sigma^{\nu})\mathcal{Y}(z+\beta,\sigma)-\Omega(z)\mathcal{Y}(z,\sigma)=\F(z,\sigma),\, z\in\C,
\end{equation}
where $\sigma$ is a complex parameter, $Re\, \sigma>0$, and $\F(z,\sigma)$ is some given function whose properties will be specified later.
Coming to the coefficient involved, the function $\Omega(z)$ reads as a finite or  an infinite product
\begin{equation}\label{i.2}
\Omega(z)=\d_{0}\exp\{Az^{2}+Bz\}\frac{\prod_{i=1}^{M_{1}}(\d_{i}^{1}-z)\prod_{i=1}^{M_{2}}(\d_{i}^{2}+z)}
{\prod_{i=1}^{M_{3}}(\d_{i}^{3}-z)\prod_{i=1}^{M_{4}}(\d_{i}^{4}+z)}\Omega_{1}(z),
\end{equation}
where
\[
\Omega_{1}(z)=1\quad\text{if} \quad \Omega(z)\,\,\text{ is a finite product,}
\]
otherwise this function reads as
\[
\Omega_{1}(z)=\prod_{n=1}^{+\infty}\frac{\prod_{i=1}^{M_{5}}(\mathfrak{h}_{i,n}-z)\prod_{i=1}^{M_{6}}(\gamma_{i,n}+z)}
{\prod_{i=1}^{M_{7}}(\zeta_{i,n}-z)\prod_{i=1}^{M_{8}}(\eta_{i,n}+z)}
\frac{\prod_{i=1}^{M_{7}}\zeta_{i,n}\prod_{i=1}^{M_{8}}\eta_{i,n}}{\prod_{i=1}^{M_{5}}\mathfrak{h}_{i,n}\prod_{i=1}^{M_{6}}\gamma_{i,n}}
.
\]
Here  $M_{j},$ $j\in\{1,2,...,8\},$ are non-negative integer numbers, and $\mathfrak{h}_{i,n}$, $\gamma_{i,n},$ $\eta_{i,n},$ $\zeta_{i,n},$  $\d_{i}^{k},$ $k\in\{1,2,3,$ $4\},$  are real sequences, $\d_{0}\neq 0,$ $A$ and $B$ are complex numbers.

Actually, as we will see later, equation \eqref{i.1} turns out to be more power tools in solving linear \textbf{FDEs} with variable  coefficients containing meromorphic functions  and in finding solutions of linear initial-boundary value problems governed by  subdiffusion in domains with singular boundaries.

The first order homogenous equation \eqref{i.1} (i.e. $\beta=1$) with coefficients  being independent of the parameter $\sigma$ was discussed in \cite{Ba} in the case of  $\Omega(z)=\prod_{n=1}^{+\infty}(\h_{1,n}-z)(\eta_{1,n}+z)^{-1}$. Assuming convergence of the series
$ \sum_{n=1}^{+\infty}[|\h_{1,n}|^{-m_1}+|\eta_{1,n}|^{-m_2}]$ with some real positive $m_1$and $m_2$, the author constructed the explicit solution  in the terms of the Barns double Gamma-functions. We also refer to work \cite{Va1}, in which exploiting the technique from \cite{Ba}, the author obtained an explicit solution to the similar homogenous equation with $\Omega(z)=\prod_{n=1}^{+\infty}\prod_{i=1}^{2}\frac{(\h_{i,n}-z)}{(\zeta_{i,n}-z)}$, where $\h_{2,n}-\h_{1,n}=1$ and $\zeta_{2,n}-\zeta_{1,n}=1$.

Concerning the first order inhomogeneous equation like \eqref{i.1}, we point out the papers
 \cite{BF,BV4,SF}, where several particular  solutions were built in the case of  a bounded analytic function $\F(z,\sigma)$ in $\C$ for each $\sigma,$ $Re\sigma>0$,  $a_1=1,a_2=0$ and $$\Omega(z)=\d_0\prod_{n=1}^{+\infty}\frac{(\h_{1,n}-z)(\gamma_{1,n}+z)\zeta_{1,n}\eta_{1,n}}{(\zeta_{1,n}-z)(\eta_{1,n}+z)\h_{1,n}\gamma_{1,n}}$$  with real sequences satisfying  relations:
\begin{equation}\label{i.3}
\h_{1,n},\gamma_{1,n},\eta_{1,n},\zeta_{1,n}=O(n)\quad\text{as}\quad n\to+\infty,\quad\text{and}\quad
\h_{1,n}-\gamma_{1,n}\equiv C,\, \eta_{1,n}-\zeta_{1,n}\equiv C_{0},
\end{equation}
with some non-negative constants $C$ and $C_{0}$. We also recall that the first order equation \eqref{i.1} with $a_2=0$ was analyzed in \cite{BV1,BV5}, where the explicit general  solutions  were constructed in the case of the coefficient $\Omega(z)=\d_0\prod_{n=1}^{+\infty}\frac{(\h_{1,n}-z)(\gamma_{1,n}+z)\prod_{i=1}^{M_7}(\zeta_{i,n}\eta_{i,n})}{\prod_{i=1}^{M_7}(\zeta_{i,n}-z)(\eta_{1,n}+z)\h_{i,n}\gamma_{1,n}}$, where $M_7>1$ and the sequences $\h_{1,n},\gamma_{1,n},\eta_{i,n},\zeta_{i,n}$ satisfied \eqref{i.3} with positive constants $C, C_0$.
It is worth noting that all studies presented in \cite{BF,BV1,BV4,BV5,SF,Va1} concern  equation \eqref{i.1} with the coefficient $\Omega(z)$  containing only real sequences.

In this paper, under weaker assumptions (compared to previous works) on the right-hand side of \eqref{i.1} and on the coefficient $\Omega_1(z)$ (see \eqref{i.2}), we construct explicit solutions of equation \eqref{i.1} with $\beta\neq 0$ as linear combination
\[
\mathcal{Y}(z,\sigma)=\mathcal{Y}_{h}(z,\sigma;\P)+\mathcal{Y}_{ih}(z,\sigma),
\]
 where $\mathcal{Y}_{h}(z,\sigma,\P)$ is an  general solution of homogenous equation \eqref{i.1} depending on an arbitrary periodic function $\P(z/\beta)$ while  $\mathcal{Y}_{ih}(z,\sigma)$ being a particular solution of inhomogeneous \eqref{i.1} is selected
in the integral form
\begin{equation}\label{i.5}
\mathcal{Y}_{ih}(z,\sigma)=\frac{\mathcal{Y}_{h}(z,\sigma)}{2i}\int_{\ell_{d_0}}\frac{\F(z+\beta\xi,\sigma)\mathcal{K}(\xi)}{(a_1\sigma+a_2\sigma^{\nu})\mathcal{Y}_{h}(z+\beta\xi+\beta,\sigma)}d\xi
\end{equation}
 with the appropriate chosen kernel $\mathcal{K}(\xi)$ and the contour $\ell_{d_0}$ in the complex plane.

 \noindent Concerning the function $\Omega_1(z)$, we just require that the sequences included in $\Omega_1(z)$ are strictly monotonic increasing and satisfy the estimate:
\[
    \sum_{n=1}^{+\infty}\Big[\sum_{i=1}^{M_{5}}\mathfrak{h}^{-2}_{i,n}+\sum_{i=1}^{M_{6}}\gamma^{-2}_{i,n}
    +\sum_{i=1}^{M_{7}}\zeta^{-2}_{i,n}+\sum_{i=1}^{M_{8}}\eta^{-2}_{i,n}
    \Big]+\Big|\sum_{i=1}^{M_{5}}\mathfrak{h}^{-1}_{i,n}-\sum_{i=1}^{M_{6}}\gamma^{-1}_{i,n}
    -\sum_{i=1}^{M_{7}}\zeta^{-1}_{i,n}+\sum_{i=1}^{M_{8}}\eta^{-1}_{i,n}
    \Big|<+\infty.
\]
Besides, in our work we analyze equation \eqref{i.1} in the case of complex sequences included in the coefficient $\Omega(z)$ and accordingly in  $\Omega_1(z)$.

As for  the properties of the function $\F(z,\sigma)$, we relax requirements on the growth of this function and on the region of analyticity of $\F$. Namely,
we consider both bounded and unbounded functions $\F$ whose growth for large $|z|$ is controlled by a specially selected analytic periodic rapidly decreasing function $\mathcal{K}$. As for the domain of  analyticity for  $\F$, we narrow this region  from  $\C$ to a specified strip
 $z_0-1<Re(z/\beta)<z_0+d_0$ with some fixed $d_0\in[0,1]$ and $z_0\in\R$.

In order to find $\mathcal{Y}_{h}(z,\sigma;\P)$ and $\mathcal{Y}_{ih}(z,\sigma)$ we employ the following strategy.
First, performing the special change of variables and introducing new unknown function, we convert equation \eqref{i.1} with $\beta\neq 1$ into the first order difference equation. Then exploiting the well-known properties of the Gamma-function, e.g. $\Gamma(z+1)=z\Gamma(z)$, and solutions of the  simplest  first order difference equations, we build general solutions $\mathcal{Y}_{h}(z,\sigma)$ of homogenous equation \eqref{i.1} in the form of the finite product (if $\Omega_1=1$) or the infinite  product  (otherwise). After that, using the well-known asymptotic  of Gamma- and Digamma-functions and technique of summation of the series from \cite[Section 2]{Ba}, we prove convergence of the corresponding products and obtain asymptotic behavior of $\mathcal{Y}_{h}(z,\sigma;\P)$ for $|Im z|\to+\infty$. Besides, we describe the region of analyticity of $\mathcal{Y}_{h}(z,\sigma;\P)$. The properties of $\mathcal{Y}_{h}(z,\sigma;\P)$ and $\F(z,\sigma)$ allow us to construct a particular solution of inhomogeneous equation \eqref{i.1} in the form \eqref{i.5} with a suitable function $\mathcal{K}$.

Moreover, we discuss application of  this technique to find explicit solutions of equation \eqref{i.1} in the multidimensional case. Namely, we analyze the linear \textbf{FDE} in the unknown function $\boldsymbol{Y}:=\boldsymbol{Y}(\mathbf{z},\sigma)$:
$$
(a_{1}\sigma+a_{2}\sigma^{\nu})\boldsymbol{Y}(\mathbf{z}+\mathbf{b},\sigma)-\mathbf{S}(\mathbf{z})\boldsymbol{Y}(\mathbf{z},\sigma)=\F(\mathbf{z},\sigma),\quad \mathbf{z}=\{z_{1},z_{2},...,z_{k}\}\in \C^{k}, \, k>1,
$$
with
$$\mathbf{b}=\{\beta_{1},\beta_{2},...,\beta_{n}\},\quad \mathbf{S}(\mathbf{z})=\d_{0}\prod_{j=1}^{k}
\exp\{A_{j}z_{j}^{2}+B_{j}z_{j}\}\frac{\prod_{i=1}^{M_{1,j}}(\d_{i,j}^{1}-z_j)\prod_{i=1}^{M_{2,j}}
(\d_{i,j}^{2}+z_j)}
{\prod_{i=1}^{M_{3,j}}(\d_{i,j}^{3}-z_j)\prod_{i=1}^{M_{4,j}}(\d_{i,j}^{4}+z_j)}\Omega_{1}^{j}(z_{j}).
$$
It is worth noting that, the one of the particular case of this equation  in two-dimensional case arises in studying of a boundary value problem for the superdiffusion equation in the right angle \cite{DV}.

In connection with application of \eqref{i.1}, taking into account the proposed approach to search explicit solutions of \eqref{i.1}, we are able to obtain solutions of \textbf{FDEs} like \eqref{i.1} with the coefficient $\Omega(z):=\mathcal{S}(z)$ being either a finite product of entire functions or their quotient. Appealing to properties of entire functions, we describe sufficient conditions on the function $\mathcal{S}(z)$, which allow us to factorize $\mathcal{S}(z)$ as an infinite product similar to \eqref{i.2}, and then apply aforementioned technique. Besides, we demonstrate how to use our technique in the case of $\mathcal{S}(z)$ being trigonometrical functions and either their linear combinations or their quotients.

Another important applications of \eqref{i.1}, as anticipated above, concerns with construction of solutions to transmission problem for Laplace operators with fractional dynamic boundary conditions governed by subdiffusion. It is worth noting that fractional differential equations play a key role in the description of the so-called anomalous phenomena in nature and in the theory of complex systems. In particular, these equations ensure a more faithful  representation of the long-memory and nonlocal dependence of many anomalous processes. The feature of these anomalies in diffusion/transport processes is that the mean square displacement of the diffusing species $\langle(\Delta\mathbf{x})^{2}\rangle $ scales as a nonlinear power law in time, i.e.  $\langle(\Delta\mathbf{x})^{2}\rangle\sim t^{\nu},$ $\nu>0$ \cite{MSSPCM}. For a subdiffusive process, the value of $\nu$ is such that $\nu\in(0,1)$, while for normal diffusion $\nu=1$, and for superdiffusive process, we have $\nu>1$.

 Let $(r,\varphi)$ be a polar coordinate system in $\R^{2}$, and $G_1$,  $G_2\subset \R^{2}$ be plane corners. For a given $\omega_0\in[0,\frac{\pi}{2})$ , we set
\[
G_1=\{(r,\varphi):\, r>0,\, \varphi\in(-\frac{\pi}{2},\omega_0)\},\,
G_2=\{(r,\varphi):\, r>0,\, \varphi\in(\omega_0,\frac{\pi}{2})\},\,
g=\{(r,\varphi):\, r\geq0,\, \varphi=\omega_0\}.
\]
For an arbitrarily given time $T>0$, we denote
\[
G_{1,T}=G_1\times(0,T),\quad G_{2,T}=G_2\times(0,T),\quad g_T=g\times[0,T].
\]
For fixed $\nu\in(0,1)$, we analyze the Laplace equations in the unknown functions $u_1$ and $u_2$, $u_1:=u_1(r,\varphi,t):G_{1,T}\to\R,$
$u_2:=u_2(r,\varphi,t):G_{2,T}\to\R,$
\begin{equation}\label{6.1}
\Delta u_1=f_{0,1}\quad\text{in}\quad G_{1,T},\quad \Delta u_2=f_{0,2}\quad\text{in}\quad G_{2,T},
\end{equation}
supplemented with the initial conditions
\begin{equation}\label{6.2}
u_1(r,\varphi,0)=0\quad\text{in}\quad \bar{G}_{1},\quad u_2(r,\varphi,0)=0\quad\text{in}\quad \bar{G}_{2},
\end{equation}
and subject to the transmission conditions on $g_T$:
\begin{equation}\label{6.3}
r^{\mathfrak{s}_0}\Big[a_1\frac{\partial}{\partial t}(u_1-u_2)+a_2\mathbb{D}_t^{\nu}(u_1-u_2)\Big]+\Big[\frac{\partial u_1}{\partial \mathbf{n}}-\frac{\partial u_2}{\partial \mathbf{n}}\Big]+a_3\frac{\partial}{\partial r}(u_1-u_2)=f(r,t)\quad\text{on}\quad g_T,
\end{equation}
\begin{equation}\label{6.4}
\frac{\partial u_1}{\partial \mathbf{n}}-\mathfrak{K}\frac{\partial u_2}{\partial \mathbf{n}}-\mathfrak{K}a_4\frac{\partial}{\partial r}[u_1-u_2]=f_1\quad \text{on}\quad g_T,
\end{equation}
and  to the Dirichlet boundary condition (\textbf{DBC}) on  $\partial G_1\backslash g$ and $\partial G_2\backslash g$:
\begin{equation}\label{6.5}
u_1(r,-\pi/2,t)=f_2,\quad u_2(r,\pi/2,t)=f_3,\quad r\geq 0,\quad t\in[0,T],
\end{equation}
where the functions $f_{0,1}$, $f_{0,2},$ $f$, $f_{k},$ $k=1,2,3,$ are prescribed, $\mathbf{n}$ is the unit normal to $g$ directed to the domain $G_1$.

\noindent
Here the symbol $\mathbb{D}_{t}^{\nu}$ stands for the Caputo fractional derivative of the order $\nu\in(0,1)$  with respect to time $t$, defined as
\[
\mathbb{D}_{t}^{\nu}u(\cdot,t)=\frac{1}{\Gamma(1-\nu)}\frac{\partial}{\partial t}\int\limits_{0}^{t}\frac{u(\cdot,\tau)-u(\cdot,0)}{(t-\tau)^{\nu}}d\tau.
\]
It is worth noting that in the limit cases $\nu=0$ and $\nu=1$, the Caputo fractional derivatives of $u(\cdot,t)$ boil down to $[u(\cdot,t)-u(\cdot,0)],$ and $\frac{\partial u}{\partial t}(\cdot,t)$, respectively.

In virtue of presence of time derivatives in transmission boundary condition \eqref{6.3}, this condition is called either dynamic boundary condition if $a_2=0$ or fractional dynamic boundary condition if $a_2\neq 0$.    If $\mathfrak{s}_{0}>0$ the dynamic/ fractional dynamic boundary condition is called degenerate, while the case of negative $\mathfrak{s}_{0}$ corresponds to the singular  condition. It is worth noting that, transmission problem  \eqref{6.1}-\eqref{6.5} is a key point in the study of the contact Muskat problem in the case of either normal diffusion ($a_{2}=0$) or subdiffusion ($a_{1}=0$). We recall that the Muskat problem (or two-phase Hele-Shaw problem) describes the evolution of the interface between the two immiscible incompressible fluids    subjected to either the classical Darcy law or its fractional version (see for  details \cite{BV1,VV}).

In this work, performing special change of variables and then Laplace and Fourier transforms, we reduce problem \eqref{6.1}-\eqref{6.5} to the functional difference equation like \eqref{i.1} with $\beta=\mathfrak{s}_0+1$. Then, incorporated the proposed technique leading to finding solution of \eqref{i.1}, we construct the integral representation of the solution to \eqref{6.1}-\eqref{6.5} which will be play a key role in the further study  of solvability and  waiting time phenomena of the corresponding contact Muskat problem.

The paper is organized as follows. In the next section, we introduce notation, main assumptions on the data in \eqref{i.1} and state our main results related with explicit solutions of \eqref{i.1} in one-dimensional and multidimensional cases  (Theorems \ref{t3.1}-\ref{t3.2}, Remark \ref{r3.5}). In Section \ref{s3}, we prove Theorem \ref{t3.1} concerning to explicit solutions of homogenous equation \eqref{i.1} ($\F\equiv 0$) and their asymptotic behaviors. The proof of Theorem \ref{t3.2} is given in Section \ref{s4}. Besides, in this section, we discuss finding solutions of \eqref{i.1} in the case of complex sequences included in  $\Omega_1(z)$ (Theorem \ref{t4.1}). Section \ref{s5} is related with searching solution of \eqref{i.1} with the coefficients being entire functions. The main results of this section is stated in Theorems \ref{t5.1} and \ref{t5.2}. Moreover, we demonstrate our technique with some explicit examples. In Section \ref{s6}, we discuss application \eqref{i.1} to solve transmission boundary problem \eqref{6.1}-\eqref{6.5}. Finally, in Appendix, we prove auxilliary results (Proposition \ref{p3.1})  playing a  key role in the course of the investigation of
 of asymptotic behavior for  homogenous solutions in Subsection \ref{s3.2}.


\section{Main Results}
\label{s2}
\noindent The first achievement of the paper is related to a solvability of homogenous equation \eqref{i.1} (i.e. $\F\equiv 0$). Throughout the work, the symbol $C$ will denote a generic positive constant, depending only on the structural quantities of the model.
Besides, our  important convention is if $n_{1}>n_{2}$, then by the definition
\[
\sum_{i=n_{1}}^{n_{2}}\mathfrak{a}(i)=0\quad\text{and}\quad\prod_{i=n_1}^{n_2}\mathfrak{a}(i)=1.
\]
Accordingly, we have that
$
\sum_{i=1}^{0}\mathfrak{a}(i)=0$ and $\prod_{i=1}^{0}\mathfrak{a}(i)=1.
$

First, we state our general hypothesis on the structural terms of the equation.

\noindent \textbf{H1 (Condition on the parameters):} Let $a_1$ and $a_2$, $\beta$, $\nu$  be real numbers, and $A$ and $B$ be complex, such that
$$a_{1}\geq 0,\quad a_{2}\geq 0,\quad a_{1}+a_{2}>0,\quad \nu\in(0,1),\quad \beta\neq 0.
$$
Besides, we assume that $\sigma$ and $\d_{0}$ are complex numbers satisfying inequalities: $Re\, \sigma\geq 0$ and $ \d_{0}\neq 0$.

\noindent \textbf{H2 (Condition on the infinite product $\Omega_{1}(z)$):} For each fixed $i$, infinite sequences $\mathfrak{h}_{i,n},$ $\gamma_{i,n},$ $\zeta_{i,n}$ and $\eta_{i,n}$ satisfy inequalities:
$$
    0<\mathfrak{h}_{i,n}<\mathfrak{h}_{i,n+1},\quad 0<\gamma_{i,n}<\gamma_{i,n+1},
    \quad 0<\zeta_{i,n}<\zeta_{i,n+1},\quad 0<\eta_{i,n}<\eta_{i,n+1};
    $$
        \begin{equation}\label{2.2}
    \sum_{n=1}^{+\infty}\Big[\sum_{i=1}^{M_{5}}\mathfrak{h}^{-2}_{i,n}+\sum_{i=1}^{M_{6}}\gamma^{-2}_{i,n}
    +\sum_{i=1}^{M_{7}}\zeta^{-2}_{i,n}+\sum_{i=1}^{M_{8}}\eta^{-2}_{i,n}
    \Big]<C,
    \,
    \sum_{n=1}^{+\infty}\Big|\sum_{i=1}^{M_{5}}\mathfrak{h}^{-1}_{i,n}-\sum_{i=1}^{M_{6}}\gamma^{-1}_{i,n}
    -\sum_{i=1}^{M_{7}}\zeta^{-1}_{i,n}+\sum_{i=1}^{M_{8}}\eta^{-1}_{i,n}
    \Big|<C.
    \end{equation}

\noindent \textbf{H3 (Condition on the coefficient $\Omega(z)$):} For non-negative integer numbers $M_1$, $M_{2},$ $M_3$ and $M_4$, we assume that sequences $\{\d_{i}^{1}\}_{i=1}^{M_1}$, $\{\d_{i}^{2}\}_{i=1}^{M_2}$, $\{\d_{i}^{3}\}_{i=1}^{M_3}$, $\{\d_{i}^{4}\}_{i=1}^{M_4}$ are real and besides, $\d_{i}^{1}\neq 0$ for each $i\in\{1,2,...,M_1\},$ and $\d_{i}^{3}\neq 0$ for each $i\in\{1,2,...,M_{3}\}$ .
\begin{example}\label{e3.1}
The following are examples of sequences satisfying condition \textbf{H2}: for $M_{8},M_{6}\geq 1,$
$
M_{5}=M_{7}=M\geq 1,$  and for $0<A_{1}<A_{2}<...<A_{M}<A_{0}$ there holds
\[
\gamma_{i,n}=n^{2}-\frac{i}{M_{6}},\, i\in\{1,2,...,M_{6}\};\quad
 \eta_{i,n}=n^{3/2}+i,\, i\in\{1,2,...,M_8\};\]
\[
\mathfrak{h}_{i,n}=2A_{0}n-A_{i},\quad \zeta_{i,n}=2A_{0} n+A_{i},\, i\in\{1,2,...,M\}.
\]
\end{example}
\begin{example}\label{e3.2}
It is apparent that the function $\tan z$ can be presented as an infinite product $\Omega(z)$ with $\d_{0}=1,$ $M_{1}=M_3=M_4=0,$ $M_2=1,$ $M_5=M_6=M_7=M_8=1,$ $A=B=0$, $\d_{1}^{2}=0$, $\h_{1,n}=\gamma_{1,n}=\pi n$, $\zeta_{1,n}=\eta_{1,n}=\frac{(2n-1)\pi}{2},$ i.e.
\[
\tan z=z\prod_{n=1}^{+\infty}\frac{(n\pi-z)(n\pi+z)}{(\frac{(2n-1)\pi}{2}-z)(\frac{(2n-1)\pi}{2}+z)}
\frac{\Big(\frac{(2n-1)\pi}{2}\Big)^{2}}{(n\pi)^{2}}
\]
\end{example}
More examples of functions which can be presented in form \eqref{i.2} are discussed in Section \ref{s5.1}.

Denoting
\begin{align*}
\mathcal{R}(n)&=\text{sgn}\beta\Big\{-\frac{\ln 2\pi}{2}[M_{6}-M_{5}+M_{7}-M_{8}]-\sum_{i=1}^{M_{6}}\frac{\gamma_{i,n}}{|\beta|}\big[\ln \frac{\gamma_{i,n}}{|\beta|}-1\big]
-\sum_{i=1}^{M_{7}}\frac{\zeta_{i,n}}{|\beta|}\big[\ln \frac{\zeta_{i,n}}{|\beta|}-1\big]
\\&+
\sum_{i=1}^{M_{5}}\frac{\h_{i,n}}{|\beta|}\big[\ln \frac{\h_{i,n}}{|\beta|}-1\big]
+\sum_{i=1}^{M_{8}}\frac{\eta_{i,n}}{|\beta|}\big[\ln \frac{\eta_{i,n}}{|\beta|}-1\big]\Big\};\end{align*}
\[
\mathbb{L}_{1}(z)
=
\begin{cases}
\prod\limits_{n=1}^{+\infty}
\frac{\prod\limits_{i=1}^{M_6}\Gamma\Big(\frac{\gamma_{i,n}+z}{\beta}\Big) \prod\limits_{i=1}^{M_7}\Gamma\Big(\frac{\zeta_{i,n}-z}{\beta}+1\Big)}
{\prod\limits_{i=1}^{M_8}\Gamma\Big(\frac{\eta_{i,n}+z}{\beta}\Big) \prod\limits_{i=1}^{M_5}\Gamma\Big(\frac{\h_{i,n}-z}{\beta}+1\Big)}
\Bigg(\frac{\prod\limits_{i=1}^{M_7}\frac{\zeta_{i,n}}{\beta}
\prod\limits_{i=1}^{M_8}\frac{\eta_{i,n}}{\beta}
}
{\prod\limits_{i=1}^{M_5}\frac{\h_{i,n}}{\beta}
\prod\limits_{i=1}^{M_6}\frac{\gamma_{i,n}}{\beta}}
\Bigg)^{\frac{z}{\beta}-\frac{1}{2}} e^{\mathcal{R}(n)},
\quad \text{if}\quad \beta>0,
\\
\prod\limits_{n=1}^{+\infty}
\frac{\prod\limits_{i=1}^{M_5}\Gamma\Big(\frac{\h_{i,n}-z}{|\beta|}\Big) \prod\limits_{i=1}^{M_8}\Gamma\Big(\frac{\eta_{i,n}+z}{|\beta|}+1\Big)}
{\prod\limits_{i=1}^{M_7}\Gamma\Big(\frac{\zeta_{i,n}-z}{|\beta|}\Big) \prod\limits_{i=1}^{M_6}\Gamma\Big(\frac{\gamma_{i,n}+z}{|\beta|}+1\Big)}
\Bigg(\frac{\prod\limits_{i=1}^{M_7}\frac{\zeta_{i,n}}{|\beta|}
\prod\limits_{i=1}^{M_8}\frac{\eta_{i,n}}{|\beta|}
}
{\prod\limits_{i=1}^{M_5}\frac{\h_{i,n}}{|\beta|}
\prod\limits_{i=1}^{M_6}\frac{\gamma_{i,n}}{\beta}}
\Bigg)^{\frac{z}{\beta}-\frac{1}{2}} e^{\mathcal{R}(n)},
\quad \text{if}\quad \beta<0,
\\
\end{cases}
\]
\[
\mathbb{L}(z)=\frac{\prod\limits_{i=1}^{M_{2}}\Gamma\Big(\frac{\d_{i}^{2}+z}{\beta}\Big)
\prod\limits_{i=1}^{M_{3}}\Gamma\Big(\frac{\d_{i}^{3}-z}{\beta}+1
\Big)}
{\prod\limits_{i=1}^{M_{4}}\Gamma\Big(\frac{\d_{i}^{4}+z}{\beta}\Big)
\prod\limits_{i=1}^{M_{1}}\Gamma\Big(\frac{\d_{i}^{1}-z}{\beta}+1
\Big)}
\begin{cases}
1,\quad\quad\quad\text{if}\quad\Omega_{1}=1,\\
\mathbb{L}_{1}(z),\quad\text{if}\quad \Omega_{1}\, \text{is an infinite product},
\end{cases}
\]
we state our first result.
\begin{theorem}\label{t3.1}
Let $\F\equiv 0$ and let for $m\in\mathbb{N}\cup\{0\}$ and $n\in\mathbb{N},$ the inequalities hold
\begin{equation}\label{2.3}
Re\, \Big(\frac{z+\d_{i}^{2}}{\beta}\Big)\neq-m,\, i\in\{1,2,...,M_{2}\},\,\text{and}\,
Re\, \Big(\frac{z-\d_{j}^{3}}{\beta}\Big)\neq 1+m,\, j\in\{1,2,...,M_{3}\}.
\end{equation}
In addition, in the case of $\Omega_{1}(z)$ being the infinite product, we also require
\begin{equation}\label{2.4}
\begin{cases}
Re\, \Big(\frac{z+\gamma_{i,n}}{\beta}\Big)\neq-m,\, i\in\{1,2,...,M_{6}\},\quad
Re\, \Big(\frac{z-\zeta_{i,n}}{\beta}\Big)\neq 1+m,\, i\in\{1,2,...,M_{7}\},\,\text{ if } \beta>0,
\\
Re\, \Big(\frac{z-\h_{i,n}}{|\beta|}\Big)\neq m,\, i\in\{1,2,...,M_{5}\},\quad
Re\, \Big(\frac{z+\eta_{i,n}}{|\beta|}\Big)\neq -1-m,\, i\in\{1,2,...,M_{8}\},\,\text{ if } \beta<0.
\end{cases}
\end{equation}
Then, under assumptions \textbf{H1-H3}, a general solution of homogenous equation \eqref{i.1} is given by
\begin{equation}\label{2.4*}
\mathcal{Y}_{h}(z,\sigma;\P)=\exp\Big\{\frac{Az^{3}}{3\beta}+\frac{B-A\beta}{2\beta}z^{2}+\frac{A\beta-3B}{6}z\Big\}
\Big(\frac{\d_{0} \beta^{M_1+M_2-M_3-M_4}}{(a_{1}\sigma+a_{2}\sigma^{\nu})}\Big)^{\frac{z}{\beta}-\frac{1}{2}}\P\Big(\frac{z}{\beta}\Big)
\mathbb{L}(z),
\end{equation}
where $\P\Big(\frac{z}{\beta}\Big)$ is an arbitrary analytic periodic function,
$
\P\Big(\frac{z+\beta}{\beta}\Big)=\P\Big(\frac{z}{\beta}\Big).
$

Moreover, the following statements hold:

\noindent \textbf{(s.1)} for each fixed $z$ satisfying \eqref{2.4}, the infinite product $\mathbb{L}_{1}(z)$ converges;

\noindent\textbf{(s.2)} for each fixed $Re \, z$ satisfying \eqref{2.3} and \eqref{2.4}, and for $|Im\, z|\to +\infty$, there is the relation
\begin{equation}\label{2.5}
\mathbb{L}(z)=[\Omega(z)\d_{0}^{-1}\exp\{-Az^{2}-Bz\}]^{\frac{z}{\beta}-\frac{1}{2}}\exp\{\Phi(z)\}
\quad
\text{where}\quad
\frac{|\Phi(z)|}{|z^{2}|\ln |z|}\leq C.
\end{equation}
Besides, $\Phi(z)=\exp\{C_{1}\ln z+C_{2} z+O(1)\}$  if one of  the following conditions hold for each $n\in\mathbb{N}$:

\noindent (i)
$ M_{5}=M_{7},\,  M_{6}=M_{8},\,\text{and }\,
\h_{i,n}-\zeta_{i,n}=C_{3},\, i\in\{1,2,...,M_{5}\},\quad
\eta_{i,n}-\gamma_{i,n}=C_{4},\, i\in\{1,2,...,M_{6}\}.
$

\noindent (ii) $M_{7}=M_{8},\, M_{5}=M_{6},\,\text{and }\,
\h_{i,n}-\gamma_{i,n}=C_{5},\, i\in\{1,2,...,M_{5}\},\quad
\eta_{i,n}-\zeta_{i,n}=C_{6},\, i\in\{1,2,...,M_{7}\},
$
where $C_{1}$ and $C_{2}$ are some constants, and the real quantities $C_3$, $C_4$, $C_5$ and $C_6$  depend only on $i$ and are independent of $n$.
\end{theorem}

The following assessment is a simple consequence of assumptions \eqref{2.3}, \eqref{2.4} and the explicit form of solution \eqref{2.4*}.
\begin{corollary}\label{c3.1}
Let for $m\in\mathbb{N}\cup\{0\}$ and $n\in\mathbb{N}$ the inequalities be fulfilled
\[
Re\, \Big(\frac{z+\d_{i}^{4}}{\beta}\Big)\neq -m,\quad i\in\{1,2,...,M_{4}\},\quad
Re\, \Big(\frac{z-\d_{j}^{1}}{\beta}\Big)\neq 1+m,\quad j\in\{1,2,...,M_{1}\}.
\]
If  $\Omega_1$ is an infinite product,  we additionally assume  that
\[
\begin{cases}
Re\, \Big(\frac{z+\eta_{i,n}}{\beta}\Big)\neq -m,\quad i\in\{1,2,...,M_{8}\},\quad
Re\, \Big(\frac {z-\h_{i,n}}{\beta}\Big)\neq 1+m,\quad i\in\{1,2,...,M_{5}\},\quad\text{if } \beta>0,
\\
Re\, \Big(\frac{z+\gamma_{i,n}}{|\beta|}\Big)\neq -m-1,\quad i\in\{1,2,...,M_{6}\},\quad
Re\, \Big(\frac {z-\zeta_{i,n}}{|\beta|}\Big)\neq m,\quad i\in\{1,2,...,M_{7}\},\quad\text{if } \beta<0,
\end{cases}
\]
Then, under assumptions \eqref{2.3} and \eqref{2.4}, the solution $\mathcal{Y}_{h}(z,\sigma;\P)$ given by \eqref{2.4*} has no poles, and the function $\frac{\mathcal{Y}_{h}(z,\sigma;\P)}{\P(z/\beta)}$ does not vanish.

If, in addition, for any fixed $d_0\in[0,1]$ the relations hold:
\begin{equation}\label{2.6}
Re\, \Big(\frac{z+\d_{i}^{4}}{\beta}\Big)\neq -m-1+d_0,\quad i\in\{1,2,...,M_{4}\},\quad
Re\, \Big(\frac{z-\d_{j}^{1}}{\beta}\Big)\neq d_0+m,\quad j\in\{1,2,...,M_{1}\},
\end{equation}
and, in the case of $\Omega_1$ being an infinite product,
\begin{equation}\label{2.7}
\begin{cases}
Re\, \Big(\frac{z+\eta_{i,n}}{\beta}\Big)\neq -m-1+d_0,\, i\in\{1,2,...,M_{8}\},\,
Re\, \Big(\frac{z-\h_{i,n}}{\beta}\Big)\neq d_0+m,\, i\in\{1,2,...,M_{5}\},\, \text{if }\beta>0,
\\
Re\, \Big(\frac{z+\gamma_{i,n}}{|\beta|}\Big)\neq -m-d_0,\, i\in\{1,2,...,M_{6}\},\,
Re\, \Big(\frac{z-\zeta_{i,n}}{|\beta|}\Big)\neq 1- d_0+m,\, i\in\{1,2,...,M_{7}\},\, \text{if }\beta<0,
\end{cases}
\end{equation}
then the function $\frac{\P(1-d_0+z/\beta)}{\mathcal{Y}_{h}(z+\beta-\beta d_0,\sigma;\P)}$
does not have any poles.
\end{corollary}
\begin{remark}\label{r3.1}
Coming a behavior of the function $\P(z/\beta)$ as $|Im\, \frac{z}{\beta}|\to+\infty$, the results in \cite[p.331-334]{Ma} tell us that $\P(z/\beta)$ (if $\P\neq const.$) grows no slower than $\exp\{2\pi|Im\frac{z}{\beta}|\}$ if either $Im \frac{z}{\beta}\to+\infty$ or $Im \frac{z}{\beta}\to-\infty$. This fact will be a key point under construction of a general and particular solutions to inhomogeneous equation \eqref{i.1}.
\end{remark}
\begin{remark}\label{r3.2}
It is apparent that inequalities \eqref{2.3} and \eqref{2.4} describe the region of analyticity  of the general solution $\mathcal{Y}_{h}(z,\sigma;\P)$. Nevertheless, this domain can be extended (i.e. conditions \eqref{2.3} and \eqref{2.4} can be relaxed) if the function $\P(z/\beta)$ fulfills a stronger requirements.
Indeed, assumptions \eqref{2.3} can be removed in Theorem \ref{t3.1} if we require that the function $\P$ has zeros in the points
\begin{equation}\label{2.3*}
Re\, \Big(\frac{z+\d^{2}_{i}}{\beta}\Big)=-m,\quad i\in\{1,...,M_2\},\quad
Re\, \Big(\frac{z-\d^{3}_{j}}{\beta}\Big)=1+m,\quad j\in\{1,...,M_3\}.
\end{equation}
Clearly, an example of the function $\P(z/\beta)$ satisfying assumptions of Theorem \ref{t3.1} and \eqref{2.3*} is
\[
\P\Big(\frac{z}{\beta}\Big)=\prod_{i=1}^{M_2}\exp\Big\{\frac{i\pi(z+\d_{i}^{2})}{\beta}\Big\}\sin\, \pi \Big(\frac{z+\d^{2}_{i}}{\beta}\Big)
\prod_{i=1}^{M_3}\exp\Big\{i\pi\Big(\frac{z-\d_{i}^{3}}{\beta}+1\Big)\Big\}\sin\, \pi \Big(\frac{z+\d^{3}_{i}}{\beta}+1\Big).
\]
\end{remark}

We are now in the position to construct a solution of inhomogeneous equation \eqref{i.1} (i.e. $\F\neq 0$). As a consequence of Theorem \ref{t3.1}, we are left with the problem of finding a particular solution to the given inhomogeneous  equation.

Till the end of the paper, we assume that $d_0\in[0,1]$ is an arbitrary fixed quantity, and we define the contour $\ell_{d_{0}}$ in the complex plane $\xi\in\C$ (see Figure \ref{fig:1} for geometric setting) having the following properties:

\noindent$\bullet$ $\ell_{d_{0}}=\{Re\, \xi=-d_{0},\, Im\, \xi\in\R\}$ if $d_0\in(0,1)$;

\noindent$\bullet$ if $d_0=1$, then $\ell_{d_{0}}=\ell_{1}$ consists of three parts: the half-circle $\{|\xi+1|=d_1,\, Re\, \xi>-1\}$ with a small positive number $d_1$, $0<d_1<d_0/8,$ and  the intervals: $\{Re\, \xi=-1,\, Im\, \xi\in(-\infty,-d_1)\}$ and $\{Re\, \xi=-1,\, Im\, \xi\in(d_1,+\infty)\}$;

\noindent$\bullet$ the contour $\ell_{0}$ (i.e. $d_0=0$) is obtained from $\ell_{1}$ after its shifting to the right-hand side on $Re\, \xi=1$.
\begin{figure}[t]
\includegraphics[scale=0.37]{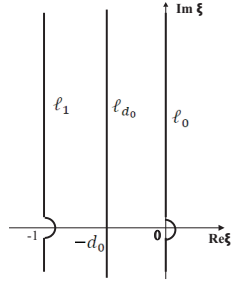}

 \caption{Typical configurations of the contour $\ell_{d_{0}}$, $d_{0}\in[0,1]$.}
 \label{fig:1}
\end{figure}
Next we introduce additional assumptions.

\noindent\textbf{H4 (Assumption on the kernel):} Let
$
\mathcal{K}(\xi)=\frac{1}{2i}[\cot\pi\xi+i]\mathcal{K}_{1}(\xi),\, \xi\in\C,
$
where $\mathcal{K}_{1}(\xi)$ obeys  conditions:

\noindent$\bullet$ $\mathcal{K}_{1}(\xi)$ is periodic with the period of $1$, i.e. $\mathcal{K}_1(\xi+1)=\mathcal{K}_1(\xi)$;

\noindent$\bullet$ $\mathcal{K}_{1}(\xi)$ does not have any poles in the strip $\{Re\, \xi\in[-1,0],\, Im\, \xi\in\R\}$;

\noindent$\bullet$ $\mathcal{K}_{1}(0)\neq 0$, and
 for each fixed $Re\, \xi\in[-1,0]$, $\underset{|Im\xi|\to+\infty}{\lim}\mathcal{K}_{1}(\xi)=0$ .

\noindent\textbf{H5 (Assumption on the right-hand side in \eqref{i.1}):} We assume that  $\F(z,\sigma)$ is an analytic function for $z\in\C$ and each $\sigma$ satisfying \textbf{H1},
such that for all $z$ satisfying \eqref{2.3}, \eqref{2.4}, \eqref{2.6} and \eqref{2.7}, and each $Re\, \xi\in[-1,0],$
the relation holds
\[
\underset{|Im\, \xi|\to+\infty}{\lim}\Big|\frac{\mathcal{K}(\xi)\F(z+\beta\xi,\sigma)}{\mathcal{Y}_{h}(z+\beta\xi+\beta,\sigma;\P_{1})}\Big|=0
\]
with appropriately selected
 functions $\mathcal{K}_{1}(\xi)$ and $\P_1:=\P(z/\beta)$ satisfying \textbf{H4} and  requirements of Theorem \ref{t3.1}, respectively.
\begin{example}\label{e3.3}
The following are examples of the kernel $\mathcal{K}_{1}$ and the function $\F$ satisfying assumptions \textbf{H4-H5} in the case of \eqref{i.1} with $\beta=1$ and $\Omega(z)$ given in Example \ref{e3.1}:
\[
\F(z,\sigma)=\sigma\sin\pi z,\quad\P_{1}(z)=e^{i\pi z}\sin\pi z,\quad \mathcal{K}_{1}(\xi)=\frac{\sin^{6}\pi d}{\sin^{6}\pi(\xi+d)}\quad\text{with} \quad 0<d-d_{0}<1.
\]
In Section \ref{s5.2} we will discuss another examples of the  function $\mathcal{K}_{1}(\xi)$ which correspond to various $\P_1(z/\beta)$ and $\Omega(z)$.
\end{example}
Now, we are ready to state the main results related with solutions of inhomogeneous equation  \eqref{i.1}.
\begin{theorem}\label{t3.2}
Let assumptions of Theorem \ref{t3.1} and hypotheses \eqref{2.6}, \eqref{2.7} and \textbf{H4-H5} hold. Then a general solution of \eqref{i.1} is a function
\begin{equation}\label{2.11*}
\mathcal{Y}(z,\sigma)=\mathcal{Y}_{h}(z,\sigma;\P)+\mathcal{Y}_{ih}(z,\sigma),
\end{equation}
where $\mathcal{Y}_{h}(z,\sigma;\P)$ is given by \eqref{2.4*} and $\mathcal{Y}_{ih}$ is a particular solution defined as
\begin{equation}\label{2.11}
\mathcal{Y}_{ih}(z,\sigma)=\frac{\mathcal{Y}_{h}(z,\sigma;\P_{1})}{\mathcal{K}_{1}(0)}\int_{\ell_{d_{0}}}\frac{\F(z+\beta\xi,\sigma)\mathcal{K}(\xi)}{(a_{1}\sigma+a_{2}\sigma^{\nu})\mathcal{Y}_{h}(z+\beta\xi+\beta,\sigma;\P_{1})}d\xi.
\end{equation}
\end{theorem}
It is worth noting that the assumption on analyticity of $\F(z,\sigma)$ in $\C$ can be relaxed provided a stronger requirements on $Re\, z$.
\begin{corollary}\label{c3.2}
Let assumptions \textbf{H1-H5}, \eqref{2.3}, \eqref{2.4}, \eqref{2.6} and \eqref{2.7} hold. We assume that for each fixed $d_{0}\in[0,1]$, the function $\frac{\F(z,\sigma)}{\mathcal{Y}_{h}(z+\beta,\sigma;\P_{1})}$ is analytic in the strip
$
z_{0}-1<Re\, (z/\beta)<z_{0}+d_{0}
$
with some $z_{0}\in\R$ satisfying \eqref{2.6}, \eqref{2.3} and, additionally, \eqref{2.4} and  \eqref{2.7}, if $\Omega_{1}$ is an infinite product. Then for each $z$ from this strip, general and particular solutions of \eqref{i.1} are defined with \eqref{2.11*} and \eqref{2.11}, correspondingly. Besides, $\mathcal{Y}_{ih}(z,\sigma)$ is analytic in the strip
\[
z_{0}<Re\, (z/\beta)<z_{0}+1+d_0.
\]
\end{corollary}
\begin{remark}\label{r3.3}
A choice of the function $\P$ in $\mathcal{Y}_{h}(z,\sigma;\P)$ (see \eqref{2.11*}) depends on the property of the desired solution and corresponding boundary problems. Indeed, if we look for an analytic solution of \eqref{i.1} which is bounded for $Im\, z\to\pm\infty$, then (see Theorem \ref{t3.1} and Remark \ref{r3.1}) $\mathcal{Y}_{h}(z,\sigma;\P)$ will grow as  $Im\, z\to\pm\infty$ unless $\P\equiv 0$. This means the desired solution of \eqref{i.1} is given by
\[
\mathcal{Y}(z,\sigma)=\frac{\mathcal{Y}_{h}(z,\sigma;\P_{1})}{\mathcal{K}_{1}(0)}\int_{\ell_{d_{0}}}\frac{\F(z+\beta\xi,\sigma)\mathcal{K}(\xi)}{(a_{1}\sigma+a_{2}\sigma^{\nu})\mathcal{Y}_{h}(z+\beta\xi+\beta,\sigma;\P_{1})}d\xi.
\]
\end{remark}
\begin{remark}\label{r3.4}
It is apparent that assumptions on the function $\mathcal{K}_{1}(\xi)$ can be relaxed if the function $\frac{\F(z+\beta\xi,\sigma)}{\mathcal{Y}_{h}(z+\beta\xi+\beta,\sigma;\P_{1})}$ fulfills  stronger requirements.
Indeed, if  $\F(z,\sigma)$ is a rapidly decreasing function for $|Im\, z|\to+\infty$ and satisfies the equality
\begin{equation*}
\underset{|Im\,\xi|\to+\infty}{\lim}\Big|\frac{\F(z+\beta\xi,\sigma)}{\mathcal{Y}_{h}(z+\beta\xi+\beta,\sigma;\P_{1})}\Big|=0,
\end{equation*}
then assumptions \textbf{H4-H5} fulfill automatically with $\mathcal{K}_{1}(\xi)\equiv 1$.
\end{remark}
\begin{remark}\label{r3.5}
Actually, with nonessential modifications in the proofs, the very same results hold for the equation like \eqref{i.1} in the multidimensional case:
\begin{equation}\label{2.14}
(a_{1}\sigma+a_{2}\sigma^{\nu})\boldsymbol{Y}(\mathbf{z}+\mathbf{b},\sigma)-\mathbf{S}(\mathbf{z})\boldsymbol{Y}(\mathbf{z},\sigma)=\F(\mathbf{z},\sigma),\quad \mathbf{z}=\{z_{1},z_{2},...,z_{k}\}\in \C^{k}, \, k>1.
\end{equation}
Here we put
$$\mathbf{b}=\{\beta_{1},\beta_{2},...,\beta_{n}\},\quad \mathbf{S}(\mathbf{z})=\d_{0}\prod_{j=1}^{k}
\exp\{A_{j}z_{j}^{2}+B_{j}z_{j}\}\frac{\prod_{i=1}^{M_{1,j}}(\d_{i,j}^{1}-z_j)\prod_{i=1}^{M_{2,j}}
(\d_{i,j}^{2}+z_j)}
{\prod_{i=1}^{M_{3,j}}(\d_{i,j}^{3}-z_j)\prod_{i=1}^{M_{4,j}}(\d_{i,j}^{4}+z_j)}\Omega_{1}^{j}(z_{j}),
$$
with
$
\Omega_{1}^{j}(z_{j})=1$ in the case of $\Omega^{j}(z_{j})$ being a finite product,
and
\[
\Omega_{1}^{j}(z_{j})=\prod_{n=1}^{+\infty}\frac{\prod_{i=1}^{M_{5,j}}(\mathfrak{h}^{j}_{i,n}-z_{j})\prod_{i=1}^{M_{6,j}}(\gamma^{j}_{i,n}+z_{j})}
{\prod_{i=1}^{M_{7,j}}(\zeta^{j}_{i,n}-z_{j})\prod_{i=1}^{M_{8,j}}(\eta^{j}_{i,n}+z_{j})}
\frac{\prod_{i=1}^{M_{7,j}}\zeta^{j}_{i,n}\prod_{i=1}^{M_{8,j}}\eta^{j}_{i,n}}{\prod_{i=1}^{M_{5,j}}\mathfrak{h}^{j}_{i,n}\prod_{i=1}^{M_{6,j}}\gamma^{j}_{i,n}}
\quad
\text{otherwise.
}
\]
Here  $M_{i,j},$ $i\in\{1,2,...,8\},$ $j\in\{1,2,...,k\}$ are non-negative integer numbers, and sequences  $\mathfrak{h}^{j}_{i,n}$, $\gamma^{j}_{i,n},$ $\eta^{j}_{i,n},$ $\zeta^{j}_{i,n},$  $\d_{i,j}^{l},$ $l\in\{1,2,3,4\},$  satisfy assumptions \textbf{H1-H3};  $A_{j}$ and $B_{j}$ are complex numbers. Under assumptions of Theorems \ref{t3.1} and \ref{t3.2}, equation \eqref{2.14} in the case of $\beta_{j}>0$, $j\in\{1,2,...,k\}$, has a solution defined by
\[
\boldsymbol{Y}(\mathbf{z},\sigma)=\boldsymbol{Y}_{h}(\mathbf{z},\sigma;\P)+
\frac{\boldsymbol{Y}_{h}(\mathbf{z},\sigma;\P_{1})}{\mathcal{K}_{1}(0)}\int_{\ell_{d_{0}}}\frac{\F(\mathbf{z}+\mathbf{b}\xi,\sigma)\mathcal{K}(\xi)}{(a_{1}\sigma+a_{2}\sigma^{\nu})\boldsymbol{Y}_{h}(\mathbf{z}+\mathbf{b}\xi+\mathbf{b},\sigma;\P_{1})}d\xi,
\]
with
\begin{align*}
\boldsymbol{Y}_{h}(\mathbf{z},\sigma;\P)&=\Big(\frac{\d_{0}\beta_{1}^{M_{1,1}+M_{2,1}-M_{3,1}-M_{4,1}}}{a_{1}\sigma+a_{2}\sigma^{\nu}}\Big)^{\frac{z_{1}}{\beta_{1}}-\frac{1}{2}}\prod_{j=1}^{k}\mathcal{Y}_{h}^{j}(z_{j},\sigma;\P)
\\
\mathcal{Y}_{h}^{j}(z_{j},\sigma;\P)&=
\exp\Big\{\frac{A_{j}z_{j}^{3}}{3\beta_{j}}+\frac{B_{j}-A_{j}\beta_{j}}{2\beta_{j}}z_{j}^{2}+\frac{A_{j}\beta_{j}-3B_{j}}{6}z_{j}\Big\}
\P\Big(\frac{z_{j}}{\beta_{j}}\Big)\\
&
\times
\frac{\prod\limits_{i=1}^{M_{2,j}}\Gamma\Big(\frac{\d_{i,j}^{2}+z_{j}}{\beta_j}\Big)
\prod\limits_{i=1}^{M_{3,j}}\Gamma\Big(\frac{\d_{i,j}^{3}-z_j}{\beta_j}+1
\Big)}
{\prod\limits_{i=1}^{M_{4,j}}\Gamma\Big(\frac{\d_{i,j}^{4}+z_j}{\beta_j}\Big)
\prod\limits_{i=1}^{M_{1,j}}\Gamma\Big(\frac{\d_{i,j}^{1}-z_j}{\beta_j}+1
\Big)}
\begin{cases}
1,\quad\quad\quad\text{if}\quad\Omega_{1}^{j}(z_j)=1,\\
\mathbb{L}_{1}^{j}(z_j),\quad\text{otherwise},
\end{cases}
\end{align*}
\begin{align*}
\mathbb{L}_{1}^{j}(z_j)&=\prod_{n=1}^{+\infty}
\frac{\prod\limits_{i=1}^{M_{6,j}}\Gamma\Big(\frac{\gamma^{j}_{i,n}+z_j}{\beta_j}\Big) \prod\limits_{i=1}^{M_{7,j}}
\Gamma\Big(\frac{\zeta^{j}_{i,n}-z_j}{\beta_j}+1\Big)}
{\prod\limits_{i=1}^{M_{8,j}}\Gamma\Big(\frac{\eta^{j}_{i,n}+z_j}{\beta_j}\Big) \prod\limits_{i=1}^{M_{5,j}}\Gamma\Big(\frac{\h^{j}_{i,n}-z_j}{\beta_j}+1\Big)}
\Bigg(\frac{\prod\limits_{i=1}^{M_{7,j}}\frac{\zeta^{j}_{i,n}}{\beta_j}
\prod\limits_{i=1}^{M_{8,j}}\frac{\eta^{j}_{i,n}}{\beta_j}
}
{\prod\limits_{i=1}^{M_{5,j}}\frac{\h^{j}_{i,n}}{\beta_j}
\prod\limits_{i=1}^{M_{6,j}}\frac{\gamma^{j}_{i,n}}{\beta_j}}
\Bigg)^{\frac{z_j}{\beta_j}-\frac{1}{2}} e^{\mathcal{R}_{j}(n)},\\
\mathcal{R}_{j}(n)&=-\frac{\ln 2\pi}{2}[M_{6,j}-M_{5,j}+M_{7,j}-M_{8,j}]-\sum_{i=1}^{M_{6,j}}\frac{\gamma^{j}_{i,n}}{\beta_j}\big[\ln \frac{\gamma^{j}_{i,n}}{\beta_j}-1\big]
-\sum_{i=1}^{M_{7,j}}\frac{\zeta^{j}_{i,n}}{\beta_j}\big[\ln \frac{\zeta^{j}_{i,n}}{\beta_j}-1\big]
\\&+
\sum_{i=1}^{M_{5,j}}\frac{\h^{j}_{i,n}}{\beta_j}\big[\ln \frac{\h^{j}_{i,n}}{\beta_{j}}-1\big]
+\sum_{i=1}^{M_{8,j}}\frac{\eta^{j}_{i,n}}{\beta_j}\big[\ln \frac{\eta^{j}_{i,n}}{\beta_j}-1\big].
\end{align*}
The similar representation to the solution of \eqref{2.14} holds in the case of either all or some $\beta_{j}<0,$ $j\in\{1,2,..., k\}$.
\end{remark}


\section{Proof of Theorem \ref{t3.1}}
\label{s3}
\noindent We first dwell on the special case where $\beta=1$, namely homogenous equation \eqref{i.1} is replaced by the simpler equation
\begin{equation}\label{3.1}
(a_{1}\sigma+a_2\sigma^{\nu})\mathcal{Y}(z+1,\sigma)-\Omega(z)\mathcal{Y}(z,\sigma)=0.
\end{equation}
In a further step, we will show how to reduce the general cas to this special one.

Here we will follow the strategy consisting in two main steps. In this first one, we construct the general solution $\mathcal{Y}_{h}(z,\sigma;\P)$ of homogenous equation \eqref{3.1} and prove the convergence of the infinite product $\mathbb{L}_{1}(z)$. Then, based on the asymptotic representation of the Euler-Gamma function, we  obtain asymptotic \eqref{2.5} including the special cases (i) and (ii) in Theorem \ref{t3.1}.

\subsection{The General Solution of \eqref{3.1}}
\label{s3.1}
\noindent In order to verify that $\mathcal{Y}_{h}(z,\sigma;\P)$ given by \eqref{2.4*} solves equation \eqref{3.1}, we appeal to well-known properties of the Euler-Gamma function $\Gamma(z)$:
$
    (z+\mathfrak{c})\Gamma(z+\mathfrak{c})=\Gamma(z+\mathfrak{c}+1)$ \textit{for any fixed} $ \mathfrak{c}\in\C;$ and
     $\Gamma(z)$ \textit{has a simple poles in the points} $Re\, z=-m,$ $m\in\mathbb{N}\cup\{0\}.$

In particular, these relations with straightforward calculations enable to conclude that  the functions: $Y_{0}(z)=\Gamma(z+\mathfrak{c}),$ $Y_{1}(z)=\frac{1}{\Gamma(1+\mathfrak{c}-z)}$, $Y_{2}(z)=\frac{1}{\Gamma(\mathfrak{c}+z)}$, $Y_{3}(z)=\Gamma(1+\mathfrak{c}-z)$,
$Y_{4}(z)=\mathfrak{c}^z$, $Y_{5}(z)=\exp\{\mathfrak{c} z^{3}+\frac{\mathfrak{c}_{1}-\mathfrak{c}}{2}z^{2}+
\frac{\mathfrak{c}-3\mathfrak{c}_{1}}{6}z\}$,
solve the simplest homogenous difference equations:
\begin{align*}
Y_0(z+1)&-(z+\mathfrak{c})Y_0(z)=0,\quad Y_1(z+1)-(\mathfrak{c}-z)Y_1(z)=0,\quad Y_2(z+1)-(z+\mathfrak{c})^{-1}Y_2(z)=0,
\\
Y_3(z+1)&-(\mathfrak{c}-z)^{-1}Y_3(z)=0,\quad
Y_4(z+1)-\mathfrak{c}Y_4(z)=0,\quad Y_5(z+1)-\exp\{\mathfrak{c}z^2+\mathfrak{c}_1z\}Y_5(z)=0.
\end{align*}
Thus, the straightforward calculations tell us that \eqref{2.4*} is a general solution of \eqref{3.1}. Besides, under  assumptions \eqref{2.3} and \eqref{2.4}, this solution does not have any poles.

At this point, we verify the convergence of the infinite  product $\mathbb{L}_{1}(z)$. It is apparent that its convergence is equivalent to convergence of the series
\begin{align}\label{3.3}\notag
\ln\mathbb{L}_{1}(z)&=\sum_{n=1}^{+\infty}\Big[(z-1/2)\ln\frac{\prod\limits_{i=1}^{M_7}\zeta_{i,n}\prod\limits_{i=1}^{M_8}\eta_{i,n}}
{\prod\limits_{i=1}^{M_5}\h_{i,n}\prod\limits_{i=1}^{M_6}\gamma_{i,n}}+\sum_{i=1}^{M_{6}}\ln\Gamma(\gamma_{i,n}+z)
+\sum_{i=1}^{M_{7}}\ln\Gamma(1+\zeta_{i,n}-z)\\
&
-
\sum_{i=1}^{M_{5}}\ln\Gamma(1+\h_{i,n}-z)-\sum_{i=1}^{M_{8}}\ln\Gamma(\eta_{i,n}+z)+\mathcal{R}(n)
\Big]
\equiv
\sum_{n=1}^{+\infty}R_{n}(z).
\end{align}
Here, for any fixed $z$ satisfying \eqref{2.4}, we show that the reminder of \eqref{3.3} is bounded. To this end, we chose $n=\mathfrak{M}$ such that the inequality holds
\[
|z|\leq\underset{n\geq\mathfrak{M}}{\min}\{\gamma_{1,n},...,\gamma_{M_6,n},\zeta_{1,n},...,\zeta_{M_{7},n},\eta_{1,n},...,\eta_{M_{8},n},\h_{1,n},...,\h_{M_{5},n}\}.
\]
After that, denoting
\[
\psi^{(0)}(z):=\psi(z)=\frac{d}{dz}\ln \Gamma(z),\quad \psi^{(j-1)}=\frac{d^{j}}{dz^{j}}\psi(z), j\in\mathbb{N},
\]
and using Taylor expansion for the function $\ln\Gamma(z+\mathfrak{c})$ near the point $\mathfrak{c},$ ($\mathfrak{c}>0$):
\[
\ln\Gamma(z+\mathfrak{c})=\ln\Gamma(\mathfrak{c})+\sum_{j=1}^{+\infty}\psi^{(j-1)}(\mathfrak{c})\frac{z^{j}}{j!},
\]
we rewrite the reminder of series \eqref{3.3} in more convenient form
\begin{equation}\label{3.4}
\sum_{n=\mathfrak{M}}^{+\infty}R_{n}(z)=\sum_{n=\mathfrak{M}}^{+\infty}Q_{1}(n)+z\sum_{n=\mathfrak{M}}^{+\infty}Q_{2}(n)
+\frac{z^{2}}{2}\sum_{n=\mathfrak{M}}^{+\infty}Q_{3}(n)+\sum_{n=\mathfrak{M}}^{+\infty}Q_{4}(z,n).
\end{equation}
Here we set
\begin{align*}
Q_{1}(n)&=-\frac{1}{2}\ln\frac{\prod\limits_{i=1}^{M_7}\zeta_{i,n}\prod\limits_{i=1}^{M_8}\eta_{i,n}}
{\prod\limits_{i=1}^{M_5}\h_{i,n}\prod\limits_{i=1}^{M_6}\gamma_{i,n}}+\sum_{i=1}^{M_{6}}\ln\Gamma(\gamma_{i,n})
+\sum_{i=1}^{M_{7}}\ln\Gamma(1+\zeta_{i,n})-\sum_{i=1}^{M_{5}}\ln\Gamma(1+\h_{i,n})\\
&-\sum_{i=1}^{M_{8}}\ln\Gamma(\eta_{i,n})+\mathcal{R}(n),
\end{align*}
\[
Q_{2}(n)=\ln\frac{\prod\limits_{i=1}^{M_7}\zeta_{i,n}\prod\limits_{i=1}^{M_8}\eta_{i,n}}
{\prod\limits_{i=1}^{M_5}\h_{i,n}\prod\limits_{i=1}^{M_6}\gamma_{i,n}}+\sum_{i=1}^{M_{6}}\psi(\gamma_{i,n})-
\sum_{i=1}^{M_{7}}\psi(1+\zeta_{i,n})+\sum_{i=1}^{M_{5}}\psi(1+\h_{i,n})-\sum_{i=1}^{M_{8}}\psi(\eta_{i,n}),\]
\[
Q_3(n)=\sum_{i=1}^{M_{6}}\psi^{(1)}(\gamma_{i,n})+
\sum_{i=1}^{M_{7}}\psi^{(1)}(1+\zeta_{i,n})-\sum_{i=1}^{M_{5}}\psi^{(1)}(1+\h_{i,n})-\sum_{i=1}^{M_{8}}\psi^{(1)}(\eta_{i,n}),\]
\begin{align*}
Q_4(z,n)&=\sum_{j=3}^{+\infty}\Big[
\sum_{i=1}^{M_{6}}\psi^{(j-1)}(\gamma_{i,n})\frac{z^{j}}{j!}-
\sum_{i=1}^{M_{7}}\psi^{(j-1)}(1+\zeta_{i,n})\frac{(-z)^{j}}{j!}
+\sum_{i=1}^{M_{5}}\psi^{(j-1)}(1+\h_{i,n})\frac{(-z)^{j}}{j!}\\
&
-\sum_{i=1}^{M_{8}}\psi^{(j-1)}(\eta_{i,n})\frac{z^{j}}{j!}
\Big].
\end{align*}
At this point, being within assumptions \textbf{H2} and \eqref{2.4},  we check (separately) that each series in the right-hand side of \eqref{3.4} is absolutely convergent. To this end, we essentially use asymptotic representations of $\psi(z)$ and $\ln\Gamma(z)$ as $|z|\to+\infty$ and $|\arg z|<\pi$ and the recurrence relations to $\psi(z)$ (see e.g. \cite[Section 6]{AS}):
\begin{align}\label{3.5}\notag
\psi(z+1)&=\psi(z)+z^{-1},\, \psi^{(1)}(z+1)=\psi^{(1)}(z)-z^{-2};\,
\psi(z)=\ln z-\frac{1}{2z}+O(z^{-2}),\, \psi^{(1)}(z)=\frac{1}{z}+O(z^{-2}),\\
\ln\Gamma(z)&=(z-1/2)\ln z-z+\frac{\ln(2\pi)}{2}+\frac{1}{12 z}+O(z^{-3}).
\end{align}

\noindent$\bullet$ Concerning the term $Q_1(n)$, the last asymptotic in \eqref{3.5}
arrives at the relation
\begin{align*}
Q_1(n)&=\mathcal{R}(n)+\frac{\ln(2\pi)}{2}[M_{6}+M_{7}-M_{5}-M_{8}]+\frac{1}{12}\Big[\sum_{i=1}^{M_6}\frac{1}{\gamma_{i,n}}+\sum_{i=1}^{M_7}\frac{1}{\zeta_{i,n}}-\sum_{i=1}^{M_5}\frac{1}{\h_{i,n}}-\sum_{i=1}^{M_8}\frac{1}{\eta_{i,n}}\Big]\\
&
+\sum_{i=1}^{M_6}\gamma_{i,n}[\ln\gamma_{i,n}-1]+\sum_{i=1}^{M_7}\zeta_{i,n}[\ln\zeta_{i,n}-1]-\sum_{i=1}^{M_5}\h_{i,n}[\ln\h_{i,n}-1]-
\sum_{i=1}^{M_8}\eta_{i,n}[\ln\eta_{i,n}-1]\\
&
+\bigg[\sum_{i=1}^{M_6}O(\gamma_{i,n}^{-3})+\sum_{i=1}^{M_7}O(\zeta_{i,n}^{-3})
+\sum_{i=1}^{M_5}O(\h_{i,n}^{-3})+\sum_{i=1}^{M_8}O(\eta_{i,n}^{-3})
\bigg].
\end{align*}
Then, collecting this equality with definition of $\mathcal{R}(n)$ and taking into account assumption \textbf{H2}, we end up with the desired convergence of the series
$
\sum_{n=\mathfrak{M}}^{+\infty}\,|Q_{1}(n)|$.

\noindent$\bullet$ Asymptotic of $\psi(z)$ in \eqref{3.5} allows us to conclude
\begin{align*}
Q_2(n)&=\frac{1}{2}\bigg[\sum_{i=1}^{M_5}\frac{1}{\h_{i,n}}-\sum_{i=1}^{M_7}\frac{1}{\zeta_{i,n}}+
\sum_{i=1}^{M_8}\frac{1}{\eta_{i,n}}-\sum_{i=1}^{M_6}\frac{1}{\gamma_{i,n}}\bigg]\\
&
+\bigg[
\sum_{i=1}^{M_5}O(\h_{i,n}^{-2})+\sum_{i=1}^{M_6}O(\gamma_{i,n}^{-2})+\sum_{i=1}^{M_7}O(\zeta_{i,n}^{-2})+
\sum_{i=1}^{M_8}O(\eta_{i,n}^{-2})
\bigg].
\end{align*}
After that, we appeal to assumptions  \eqref{2.2} and deduce that the series $\sum\limits_{n=\mathfrak{M}}^{+\infty}|Q_{2}(n)|$ is convergent.

\noindent$\bullet$ Coming to the third term in \eqref{3.4}, the asymptotic representation to the derivative $\psi^{(1)}(z)$ in \eqref{3.5} entails
\begin{align*}
Q_{3}(n)&=\sum_{i=1}^{M_{6}}\frac{1}{\gamma_{i,n}}+\sum_{i=1}^{M_{7}}\frac{1}{\zeta_{i,n}}-\sum_{i=1}^{M_{5}}\frac{1}{\h_{i,n}}
-\sum_{i=1}^{M_{8}}\frac{1}{\eta_{i,n}}\\
&
+
\sum_{i=1}^{M_5}O(\h^{-2}_{i,n})+\sum_{i=1}^{M_6}O(\gamma^{-2}_{i,n})+\sum_{i=1}^{M_7}O(\zeta^{-2}_{i,n})+
\sum_{i=1}^{M_8}O(\eta^{-2}_{i,n}).
\end{align*}
Then, assumption \textbf{H2} provides the convergence of $\sum\limits_{n=\mathfrak{M}}^{+\infty}|Q_{3}(n)|$.

\noindent$\bullet$ Finally, operating arguments from \cite[p.445]{Ba} and appealing to condition \textbf{H2}, we end up with the bound
\[
\sum_{n=\mathfrak{M}}^{+\infty}|Q_{4}(n,z)|<C.
\]
In summary, we conclude the absolutely convergence of series in the right-hand side of  \eqref{3.4}.
Thus, we arrive at the convergence of $\mathbb{L}_{1}(z)$ for each fixed $z$ which in turn finishes the proof of the first part of Theorem \ref{t3.1} in the case of $\beta=1$. \qed


\subsection{ Asymptotic \eqref{2.5}}
\label{s3.2}
\noindent The key point in our arguments here is the following technical assertion which will be proved in Appendix.
\begin{proposition}\label{p3.1}
Let $z\in\C,$ $z=z_{1}+iz_{2}$, and let $\{\b(n)\}_{n=1}^{+\infty}$ be a  real strictly monotonic increasing sequence obeying the properties:

\noindent $\bullet$ $\b(n)>0$ for all
$n\in\N,$ and
 $\sum_{n=1}^{+\infty}\frac{1}{\b^{\,2}(n)}<+\infty,$

\noindent $\bullet$ for each fixed $z_{1}$ and each real number $C^{\star}$, there are  integer  numbers   $\mathfrak{N}_{0}, \mathfrak{N}_{1}\geq 1$, such that inequalities hold
$$
\b(n)\pm z_{1}>0\, \text{ for } n\geq \mathfrak{N}_{0}+1,\quad \text{and}\quad
\b(n)+C^{\star}\pm z_{1}>0\, \text{ for }  n\geq\mathfrak{N}_{1}+1.
$$
Besides, there exists  $\mathfrak{N}\geq \max\{\mathfrak{N}_0,\mathfrak{N}_1\}$ such that bounds are fulfilled:
\[
\b(n)>4|z_{1}|,\quad \b(n)+2C^{\star}>4|z_{1}|\quad\text{for each } n\geq\mathfrak{N}.
\]
Then, for $|z_2|\to+\infty$ the relations hold:
\begin{description}
    \item[i]\[
    \sum_{n=\mathfrak{N}}^{+\infty}\Big|\frac{z}{\b(n)(\b(n)+z)}\Big|+
        \sum_{n=\mathfrak{N}}^{+\infty}\Big|\frac{z}{\b(n)(\b(n)-z)}\Big|\leq C[1+\ln|z|]+O(|z|^{-1}),
    \]
    \item[ii]\begin{align*}
    &\sum_{n=\mathfrak{N}}^{+\infty}\Big|\frac{z}{(\b(n)+z+C^{\star})(\b(n)+z)}\Big|+
        \sum_{n=\mathfrak{N}}^{+\infty}\Big|\frac{z}{(\b(n)-z+C^{\star})(\b(n)-z)}\Big|+
        \sum_{n=\mathfrak{N}}^{+\infty}\Big|\frac{z}{(\b(n)-z+C^{\star})(\b(n)+z)}\Big|
        \\
        &\leq C+O(|z|^{-1}),
    \end{align*}
        \item[iii]\begin{align*}
    \sum_{n=\mathfrak{N}}^{+\infty}\Big|\b(n)\bigg(\frac{z}{\b(n)+z}\bigg)^{3}\sum_{j=1}^{+\infty}\frac{1}{j+3}\Big(\frac{z}{\b(n)+z}\Big)^{j}\Big|&+
        \sum_{n=\mathfrak{N}}^{+\infty}\Big|\b(n)\bigg(\frac{z}{\b(n)-z}\bigg)^{3}\sum_{j=1}^{+\infty}\frac{(-1)^{j-2}}{j+3}\Big(\frac{z}{\b(n)-z}\Big)^{j}\Big|
                \\
                &\leq C[1+|z^2|+|z^2|\ln|z|].
    \end{align*}
    \item[iv]
    \begin{align*}
    &\Big|\sum_{n=\mathfrak{N}}^{+\infty}\b(n)\bigg(\frac{z}{\b(n)+z}\bigg)^{3}\sum_{j=1}^{+\infty}\frac{1}{j+3}\Big(\frac{z}{\b(n)+z}\Big)^{j}
    \\&-
        \sum_{n=\mathfrak{N}}^{+\infty}[\b(n)+C^{\star}]\bigg(\frac{z}{\b(n)+C^{\star}+z}\bigg)^{3}\sum_{j=1}^{+\infty}\frac{1}{j+3}\Big(\frac{z}{\b(n)+z+C^{\star}}\Big)^{j}\Big|
                \\
    &+\Big|\sum_{n=\mathfrak{N}}^{+\infty}\b(n)\bigg(\frac{z}{\b(n)-z}\bigg)^{3}\sum_{j=1}^{+\infty}\frac{(-1)^{j-2}}{j+3}\Big(\frac{z}{\b(n)-z}\Big)^{j}\\&+
        \sum_{n=\mathfrak{N}}^{+\infty}[\b(n)+C^{\star}]\bigg(\frac{z}{\b(n)+C^{\star}-z}\bigg)^{3}\sum_{j=1}^{+\infty}\frac{(-1)^{j-2}}{j+3}\Big(\frac{z}{\b(n)-z+C^{\star}}\Big)^{j}\Big|
                \\
                &\leq C|z|+O(1).
    \end{align*}
\end{description}
\end{proposition}
Coming to asymptotic \eqref{2.5} and being within assumptions \eqref{2.3} and \eqref{2.4} to $Re\, z$, we first remark that the desired asymptotic
is a simple consequence of the following relations
\begin{align}\label{3.6}\notag
\ln\frac{\mathbb{L}(z)}{\mathbb{L}_{1}(z)}-(z-1/2)\ln\Omega_{0}(z)&=C_{1}^{0}\ln z+C_{2}^{0}z+O(1),\\
\ln\mathbb{L}_{1}(z)-(z-1/2)\ln\Omega_{1}(z)&=C_{1}^{1}\ln z+C_{2}^{1}z+\Phi^{\star}(z),\,\text{ as } |Im\, z|\to+\infty,
\end{align}
where $\Phi^{\star}(z)$ obeys the properties of $\Phi(z)$, $C_{1}^{0},$ $C_{2}^{0}$, $C_{1}^{1}$ and $C_{2}^{1}$ are some constants, and
$
\Omega_{0}(z)=\frac{\Omega(z)}{\d_0\Omega_{1}(z)}\exp\{-Az^{2}-Bz\}.
$
Thus, we are left to prove \eqref{3.6}.
To this end, we first fixate $Re\, z$ and being within assumptions \textbf{H2} and \eqref{2.4}, we choose a positive integer number $N$ (which may depend on $Re\,z $) such that the estimate holds
\begin{equation}\label{3.6*}
4|Re\, z|<\max\{\gamma_{1,N},...,\gamma_{M_6,N},\eta_{1,N},...,\eta_{M_8,N},\zeta_{1,N},...,\zeta_{M_7,N},\h_{1,N},...,\h_{M_5,N}\}.
\end{equation}
Then putting
\begin{align*}
S_{1}^{0}(z)&=\sum_{i=1}^{M_2}\d_i^{2}\ln(\d_i^{2}+z)+\sum_{i=1}^{M_3}\d_i^{3}\ln(\d_{i}^{3}-z)-\sum_{i=1}^{M_4}\d_i^{4}\ln(\d_{i}^{4}+z)
-\sum_{i=1}^{M_1}\d_{i}^{1}\ln(d_{i}^{1}-z),
\\
S_2^{0}(z)&=z[M_3+M_4-M_2-M_1],\\
S_{1,N}(z)&=\sum_{n=1}^{N}\Big[
\sum_{i=1}^{M_{5}}O((\h_{i,n}+z)^{-3})+\sum_{i=1}^{M_{6}}O((z+\gamma_{i,n})^{-3})+\sum_{i=1}^{M_{7}}O((\zeta_{i,n}-z)^{-3})\\
&+
\sum_{i=1}^{M_{8}}O((z+\eta_{i,n})^{-3})
\Big],\\
S_{1,\infty}(z)&=\sum_{n=N+1}^{+\infty}\Big[
\sum_{i=1}^{M_{5}}O((\h_{i,n}-z)^{-3})+\sum_{i=1}^{M_{6}}O((z+\gamma_{i,n})^{-3})+\sum_{i=1}^{M_{7}}O((\zeta_{i,n}-z)^{-3})\\
&+
\sum_{i=1}^{M_{8}}O((z+\eta_{i,n})^{-3})
\Big],\\
S_{2,N}(z)&=\frac{1}{12}\sum_{n=1}^{N}\Big[\sum_{i=1}^{M_6}(\gamma_{i,n}+z)^{-1}-
\sum_{i=1}^{M_5}(\h_{i,n}-z)^{-1}+\sum_{i=1}^{M_7}(\zeta_{i,n}-z)^{-1}
-\sum_{i=1}^{M_8}(\eta_{i,n}+z)^{-1}
\Big],
\\
S_{2,\infty}(z)&=\frac{1}{12}\sum_{n=N+1}^{+\infty}\Big[\sum_{i=1}^{M_6}(\gamma_{i,n}+z)^{-1}-
\sum_{i=1}^{M_5}(\h_{i,n}-z)^{-1}+\sum_{i=1}^{M_7}(\zeta_{i,n}-z)^{-1}
-\sum_{i=1}^{M_8}(\eta_{i,n}+z)^{-1}
\Big],\\
S_{3,N}(z)&=\sum_{n=1}^{N}\Big[\sum_{i=1}^{ M_6}(\gamma_{i,n}\ln\frac{\gamma_{i,n}+z}{\gamma_{i,n}}-z)
+
\sum_{i=1}^{    M_7}(\zeta_{i,n}\ln\frac{\zeta_{i,n}-z}{\zeta_{i,n}}+z)+
\sum_{i=1}^{    M_8}(z- \eta_{i,n}\ln\frac{\eta_{i,n}+z}{\eta_{i,n}})\\
&
-\sum_{i=1}^{   M_5}(\h_{i,n}\ln\frac{\h_{i,n}-z}{\h_{i,n}}+z)
\Big]\\
S_{3,\infty}(z)&=\sum_{n=N+1}^{+\infty}\Big[\sum_{i=1}^{    M_6}(\gamma_{i,n}\ln\frac{\gamma_{i,n}+z}{\gamma_{i,n}}-z)
+
\sum_{i=1}^{    M_7}(\zeta_{i,n}\ln\frac{\zeta_{i,n}-z}{\zeta_{i,n}}+z)+
\sum_{i=1}^{    M_8}(z- \eta_{i,n}\ln\frac{\eta_{i,n}+z}{\eta_{i,n}})\\
&
-\sum_{i=1}^{   M_5}(\h_{i,n}\ln\frac{\h_{i,n}-z}{\h_{i,n}}+z)
\Big],
\end{align*}
and
 appealing to asymptotic  \eqref{3.5} and assumptions \eqref{2.3} and \eqref{2.4}, we rewrite the left-hand sides in \eqref{3.6} in the form
\begin{align}\label{3.7}\notag
\ln\frac{\mathbb{L}(z)}{\mathbb{L}_{1}(z)}-(z-1/2)\ln\Omega_{0}(z)&=S_{1}^{0}(z)+S_{2}^{0}(z)+O(1),\\
\ln\mathbb{L}_{1}(z)-(z-1/2)\ln\Omega_{1}(z)&=\sum_{j=1}^{3}[S_{j,N}(z)+S_{j,\infty}(z)],\quad |Im z|\to+\infty.
\end{align}
Performing technical  calculations, we deduce the equality
\begin{align*}\label{3.8}\notag
S_{1}^{0}(z)&=\Big[\sum_{i=1}^{M_2}\d_{i}^{2}+\sum_{i=1}^{M_3}\d_{i}^{3}-\sum_{i=1}^{M_4}\d_{i}^{4}-\sum_{i=1}^{M_1}\d_{i}^{1}\Big]\ln z
\\
&+
\frac{1}{z}\Big[\sum_{i=1}^{M_2}(\d_{i}^{2})^{2}-
\sum_{i=1}^{M_3}(\d_{i}^{3})^{2}-\sum_{i=1}^{M_4}(\d_{i}^{4})^{2}
+\sum_{i=1}^{M_1}(\d_{i}^{1})^{2}
\Big]+O(1),
\end{align*}
which in turn together with the representation of  $S_{2}^{0}(z)$ arrive at the first relation in \eqref{3.6} with
\[
C_{1}^{0}=M_3+M_4-M_2-M_1,\quad C_{2}^{0}=\sum_{i=1}^{M_2}\d_{i}^{2}+\sum_{i=1}^{M_3}\d_{i}^{3}-\sum_{i=1}^{M_4}\d_{i}^{4}-\sum_{i=1}^{M_1}\d_{i}^{1}.
\]
In order to obtain the second relation in \eqref{3.6}, we are left to evaluate, separately, each term $S_{j,\infty}(z)$ and $S_{j,N}(z)$ at the right-hand side of the second relation in \eqref{3.7}.

\noindent $\bullet$ Recasting the arguments leading to representation of $S_{1}^{0}(z)$ and taking into account that $|Im\, z|\to+\infty$, we immediately conclude
\begin{align}\label{3.9}\notag
S_{1,N}(z)+S_{2,N}(z)+S_{3,N}(z)&=
\sum_{n=1}^{N}\Big[\sum_{i=1}^{M_6}\gamma_{i,n}-\sum_{i=1}^{M_5}\h_{i,n}+\sum_{i=1}^{M_7}\zeta_{i,n}-\sum_{i=1}^{M_8}\eta_{i,n}\Big]\ln z
\\&+z N[M_7+M_8-M_6-M_5]+O(1).
\end{align}

\noindent $\bullet$ The first relation in assumption \eqref{2.2} and condition \eqref{3.6*} lead to the bound
\[
|S_{1,\infty}(z)|\leq\frac{C}{|Im\, z|}\sum_{n=N+1}^{+\infty}\Big[
\sum_{i=1}^{M_6}\gamma_{i,n}^{-2}+\sum_{i=1}^{M_8}\eta_{i,n}^{-2}+
\sum_{i=1}^{M_7}\zeta_{i,n}^{-2}+\sum_{i=1}^{M_5}\h_{i,n}^{-2}
\Big],
\]
which provides the desired equality
\begin{equation*}\label{3.10}
S_{1,\infty}(z)=O(1) \quad\text{as}\quad |Im\, z|\to +\infty.
\end{equation*}

\noindent $\bullet$ Introducing the functions:
\begin{align*}
S_{2,\infty}^{1}(z)&=\frac{1}{12}\sum_{n=N+1}^{+\infty}\Big[
\sum_{i=1}^{M_6}\gamma_{i,n}^{-1}-\sum_{i=1}^{M_5}\h_{i,n}^{-1}
+\sum_{i=1}^{M_7}\zeta_{i,n}^{-1}-\sum_{i=1}^{M_8}\eta_{i,n}^{-1}
\Big],
\\
S_{2,\infty}^{2}(z)&=\frac{z}{12}\sum_{n=N+1}^{+\infty}\Bigg[
\sum_{i=1}^{M_7}\frac{1}{\zeta_{i,n}(\zeta_{i,n}-z)}
+
\sum_{i=1}^{M_8}\frac{1}{\eta_{i,n}(\eta_{i,n}+z)}
-
\sum_{i=1}^{M_5}\frac{1}{\h_{i,n}(\h_{i,n}-z)}
-
\sum_{i=1}^{M_6}\frac{1}{\gamma_{i,n}(\gamma_{i,n}+z)}
\Bigg],
\end{align*}
we rewrite the term
 $S_{2,\infty}(z)$ in more suitable form
\[
S_{2,\infty}(z)=S_{2,\infty}^{1}(z)+S_{2,\infty}^{2}(z).
\]
It is apparent that the second inequality in  \eqref{2.2} ensures the absolutely convergence of the series $S_{2,\infty}^{1}(z)$, which in turn means
\[
S_{2,\infty}^{1}(z)=O(1)\quad\text{as}\quad |Im\,z|\to+\infty.
\]
Concerning the term $S_{2,\infty}^{2}(z),$  we can take advantage of the straightforward relations as $|Im\, z|\to+\infty$ (which are simple consequences of  \textbf{H2} and \eqref{3.6*}):
\begin{align*}
\sum_{n=N+1}^{+\infty}\sum_{i=1}^{M_7}\frac{1}{\zeta_{i,n}|\zeta_{i,n}-z|}&\leq\sum_{n=1}^{+\infty}\sum_{i=1}^{M_7}\frac{C}{\zeta_{i,n}^{2}}=O(1),\quad
\sum_{n=N+1}^{+\infty}\sum_{i=1}^{M_6}\frac{1}{\gamma_{i,n}|\gamma_{i,n}+z|}\leq\sum_{n=1}^{+\infty}\sum_{i=1}^{M_6}\frac{C}{\gamma_{i,n}^{2}}=O(1),\\
\sum_{n=N+1}^{+\infty}\sum_{i=1}^{M_5}\frac{1}{\h_{i,n}|\h_{i,n}-z|}&\leq\sum_{n=1}^{+\infty}\sum_{i=1}^{M_5}\frac{C}{\h_{i,n}^{2}}=O(1),\quad
\sum_{n=N+1}^{+\infty}\sum_{i=1}^{M_8}\frac{1}{\eta_{i,n}|\eta_{i,n}+z|}\leq\sum_{n=1}^{+\infty}\sum_{i=1}^{M_8}\frac{C}{\eta_{i,n}^{2}}=O(1).
\end{align*}
Accordingly, we end up with the equality
\[
S_{2,\infty}^{2}(z)=C z+O(1)\quad \text{as}\quad|Im\, z|\to+\infty.
\]
Collecting this equality  with representation of ${S_{2,\infty}^{1}(z)}$, we deduce
\begin{equation}\label{3.11}
S_{2,\infty}(z)=C_{7}z+O(1)\quad\text{as } |Im\, z|\to+\infty.
\end{equation}
At this point, assuming additionally either (i) or (ii) conditions in Theorem \ref{t3.1}, we show that
 the constant $C_7$ in \eqref{3.11} equals to zero.
In light of arguments leading to \eqref{3.11}, we are left to verify that $S_{2,\infty}^{2}(z)=O(1)$ as $|Im\, z|\to+\infty$ and either (i) or (ii) holds. Indeed, the technical calculations arrive at equalities:
\begin{align*}
S_{2,\infty}^{2}(z)&=\frac{z}{12}\sum_{n=N+1}^{+\infty}\Bigg[C_{3}\sum_{i=1}^{M_5}\Big\{
\frac{1}{(\zeta_{i,n}-z)(\zeta_{i,n}+C_{3}-z)(C_3+\zeta_{i,n})}
+\frac{1}{\zeta_{i,n}(C_3+\zeta_{i,n})(\zeta_{i,n}-z)}
\Big\}
\\
&
-
C_{4}
\sum_{i=1}^{M_6}\Big\{
\frac{1}{(\gamma_{i,n}+z)(\gamma_{i,n}+C_{4}+z)(C_4+\gamma_{i,n})}
+\frac{1}{\gamma_{i,n}(C_4+\gamma_{i,n})(\gamma_{i,n}+z)}
\Big\}
\Bigg]
\end{align*}
if (i) holds while in the case of (ii)
\begin{align*}
S_{2,\infty}^{2}(z)&=\frac{z}{12}\sum_{n=N+1}^{+\infty}\Bigg[
-\sum_{i=1}^{M_6}\frac{\gamma_{i,n}^{2}+(C_5+\gamma_{i,n})^{2}-C_5z}{\gamma_{i,n}(C_5+\gamma_{i,n})(C_5+\gamma_{i,n}-z)(\gamma_{i,n}+z)}\\
&
+
\sum_{i=1}^{M_7}\frac{(C_6+\zeta_{i,n})^{2}+\zeta_{i,n}^{2}+C_6z}{\zeta_{i,n}(C_6+\zeta_{i,n})(\zeta_{i,n}-z)(C_6+\zeta_{i,n}+z)}
\Bigg].
\end{align*}
Then, appealing to assumptions \eqref{3.6*}, \textbf{H2} and Proposition \ref{p3.1} to control the right-hand sides in the relations above, we conclude
\begin{equation*}\label{3.12}
S_{2,\infty}^{2}(z)=O(1)\text{ and  } S_{2,\infty}(z)=O(1)\quad\text{as }|Im\, z|\to+\infty,
\end{equation*}
and if either (i) or (ii) is fulfilled.

\noindent $\bullet$ Concerning the term $S_{3,\infty}(z)$, we are left to examine the following two relations:
\begin{equation}\label{3.13}
|S_{3,\infty}(z)|\leq C|z^{2}|\ln|z|\quad\text{as }|Im\,z|\to+\infty,
\end{equation}
and, under the additional assumption (i) or (ii) of Theorem \ref{t3.1},
\begin{equation}\label{3.14}
S_{3,\infty}(z)=C_{8}z+O(1).
\end{equation}

Coming to inequality \eqref{3.13}, we first remark that
\[
\Big|\frac{z}{\gamma_{i,n}+z}\Big|<1,\quad \Big|\frac{z}{\zeta_{i,n}-z}\Big|<1,\quad \Big|\frac{z}{\eta_{i,n}+z}\Big|<1,\quad
\Big|\frac{z}{\h_{i,n}-z}\Big|<1
\]
if $|Im\, z|\to +\infty$, $n\geq N$ and \eqref{3.6*} holds. Then, taking into account these relations and employing the
Taylor expansion, we rewrite $S_{3,\infty}(z)$ in the form
\begin{equation}\label{3.15}
S_{3,\infty}(z)=\sum_{j=1}^{4}S_{3,\infty}^{j},
\end{equation}
where we set
\begin{align*}
S^{1}_{3,\infty}(z)&=z^{2}
\sum_{n=N+1}^{+\infty}\Big[
\sum_{i=1}^{M_5}\frac{1}{\h_{i,n}-z}-\sum_{i=1}^{M_6}\frac{1}{\gamma_{i,n}+z}+
\sum_{i=1}^{M_8}\frac{1}{\eta_{i,n}+z}-
\sum_{i=1}^{M_7}\frac{1}{\zeta_{i,n}-z}
\Big];\end{align*}
\begin{align*}
S^{2}_{3,\infty}(z)&=\frac{z^{2}}{2}
\sum_{n=N+1}^{+\infty}\Big[
\sum_{i=1}^{M_6}\frac{\gamma_{i,n}}{(\gamma_{i,n}+z)^{2}}-\sum_{i=1}^{M_5}\frac{\h_{i,n}}{(\h_{i,n}-z)^{2}}+
\sum_{i=1}^{M_7}\frac{\zeta_{i,n}}{(\zeta_{i,n}-z)^{2}}
-
\sum_{i=1}^{M_8}\frac{\eta_{i,n}}{(\eta_{i,n}+z)^{2}}
\Big];\end{align*}
\begin{align*}
S^{3}_{3,\infty}(z)&=\frac{z^{3}}{3}
\sum_{n=N+1}^{+\infty}\Big[
\sum_{i=1}^{M_6}\frac{\gamma_{i,n}}{(\gamma_{i,n}+z)^{3}}+\sum_{i=1}^{M_5}\frac{\h_{i,n}}{(\h_{i,n}-z)^{3}}-
\sum_{i=1}^{M_7}\frac{\zeta_{i,n}}{(\zeta_{i,n}-z)^{3}}
-
\sum_{i=1}^{M_8}\frac{\eta_{i,n}}{(\eta_{i,n}+z)^{3}}
\Big];\end{align*}
\begin{align*}
S^{4}_{3,\infty}(z)&=z^{3}
\sum_{n=N+1}^{+\infty}\Big[
\sum_{i=1}^{M_6}\frac{\gamma_{i,n}}{(\gamma_{i,n}+z)^{3}}\sum_{j=4}^{+\infty}\frac{1}{j}\Big(\frac{z}{\gamma_{i,n}+z}\Big)^{j-3}+\sum_{i=1}^{M_5}\frac{\h_{i,n}}{(\h_{i,n}-z)^{3}}\sum_{j=4}^{+\infty}\frac{(-1)^{j-1}}{j}\Big(\frac{z}{\h_{i,n}-z}\Big)^{j-3}
\\&-
\sum_{i=1}^{M_7}\frac{\zeta_{i,n}}{(\zeta_{i,n}-z)^{3}}
\sum_{j=4}^{+\infty}\frac{(-1)^{j-1}}{j}\Big(\frac{z}{\zeta_{i,n}-z}\Big)^{j-3}
-
\sum_{i=1}^{M_8}\frac{\eta_{i,n}}{(\eta_{i,n}+z)^{3}}
\sum_{j=4}^{+\infty}\frac{1}{j}\Big(\frac{z}{\eta_{i,n}+z}\Big)^{j-3}
\Big].
\end{align*}
After that, we examine each term $S_{3,\infty}^{j}(z)$.     To this end, exploiting Proposition \ref{p3.1} and assumptions \textbf{H2}, \eqref{3.6*}, we obtain
\begin{align*}
&|S^{1}_{3,\infty}(z)|=\Big|z^{3}
\sum_{n=N+1}^{+\infty}\Big[
\sum_{i=1}^{M_5}\frac{1}{\h_{i,n}(\h_{i,n}-z)}+\sum_{i=1}^{M_6}\frac{1}{\gamma_{i,n}(\gamma_{i,n}+z)}-
\sum_{i=1}^{M_8}\frac{1}{\eta_{i,n}(\eta_{i,n}+z)}-
\sum_{i=1}^{M_7}\frac{1}{\zeta_{i,n}(\zeta_{i,n}-z)}\Big]\\
&+
z^{2}
\sum_{n=N+1}^{+\infty}\Big[
\sum_{i=1}^{M_5}\frac{1}{\h_{i,n}}-\sum_{i=1}^{M_6}\frac{1}{\gamma_{i,n}}+
\sum_{i=1}^{M_8}\frac{1}{\eta_{i,n}}-
\sum_{i=1}^{M_7}\frac{1}{\zeta_{i,n}}
\Big]
\Big|
\leq
C|z^{2}|[1+\ln|z|]+O(|z|),
\\
&|S^{2}_{3,\infty}(z)|=\Big|\frac{z^{2}}{2}\Big|\Big|
\sum_{n=N+1}^{+\infty}\Big[
\sum_{i=1}^{M_8}\frac{z(2\eta_{i,n}+z)}{\eta_{i,n}(\eta_{i,n}+z)^2}
-
\sum_{i=1}^{M_5}\frac{z(2\h_{i,n}-z)}{\h_{i,n}(\h_{i,n}-z)^{2}}
-\sum_{i=1}^{M_6}\frac{z(2\gamma_{i,n}+z)}{\gamma_{i,n}(\gamma_{i,n}+z)^2}
+\\&
\sum_{i=1}^{M_7}\frac{z(2\zeta_{i,n}-z)}{\zeta_{i,n}(\zeta_{i,n}-z)^2}\Big]
+
\sum_{n=N+1}^{+\infty}\Big[\sum_{i=1}^{M_6}\frac{1}{\gamma_{i,n}}-
\sum_{i=1}^{M_5}\frac{1}{\h_{i,n}}
-
\sum_{i=1}^{M_8}\frac{1}{\eta_{i,n}}+
\sum_{i=1}^{M_7}\frac{1}{\zeta_{i,n}}
\Big]
\Big|
\leq
C|z^{2}|[1+\ln|z|]+O(|z|),\\
&|S^{3}_{3,\infty}(z)|=\Big|\frac{z^{2}}{3}\Big|\Big|
\sum_{n=N+1}^{+\infty}\Big[
\sum_{i=1}^{M_6}\frac{z}{(\gamma_{i,n}+z)^2}
-
\sum_{i=1}^{M_8}\frac{z}{(\eta_{i,n}+z)^2}
+
\sum_{i=1}^{M_5}\frac{z}{(\h_{i,n}-z)^{2}}
-
\sum_{i=1}^{M_7}\frac{z}{(\zeta_{i,n}-z)^2}\Big]
\\&
-
\sum_{n=N+1}^{+\infty}\Big[\sum_{i=1}^{M_6}\frac{z^{2}}{(\gamma_{i,n}+z)^{3}}+
\sum_{i=1}^{M_5}\frac{z^{2}}{(\h_{i,n}-z)^3}
-
\sum_{i=1}^{M_8}\frac{z^{2}}{(\eta_{i,n}+z)^3}-
\sum_{i=1}^{M_7}\frac{z^{2}}{(\zeta_{i,n}-z)^{3}}
\Big]
\Big|
\leq
C(|z^{2}|+O(|z|),\end{align*}
\[|S_{3,\infty}^{4}(z)|\leq C[1+|z^{2}|+|z^{2}|\ln|z|]\quad\text{as}\quad |Im\, z|\to+\infty.\]
Hence, collecting these estimates with \eqref{3.15}, we end up with inequality \eqref{3.13}.

Finally, we are left to achieve asymptotic \eqref{3.14} in the case of additional assumption  (i) or (ii) in Theorem \ref{t3.1}.
Here, we will carry out the detailed proof of \eqref{3.14} if (i) holds. The proof of \eqref{3.14} in the case of (ii) is almost identical and left to the interested reader.
In virtue of assumption (i), the terms $S_{3,\infty}^{j}(z),$ $j\in\{1,2,3,4\},$ can be rewritten in the form
\[
S_{3,\infty}^{1}(z)=-z\sum_{n=N+1}^{+\infty}\Big[
\sum_{i=1}^{M_6}\frac{C_4z}{(\gamma_{i,n}+z)(\gamma_{i,n}+C_4+z)}+
\sum_{i=1}^{M_5}\frac{C_3z}{(\zeta_{i,n}+C_3-z)(\zeta_{i,n}-z)}
\Big];\]
\begin{align*}
S_{3,\infty}^{2}(z)&=
\frac{z}{2}\sum_{n=N+1}^{+\infty}\Big(
\sum_{i=1}^{M_6}\Big[\frac{C_4z\gamma_{i,n}(2\gamma_{i,n}+C_4+2z)}{(\gamma_{i,n}+z)^{2}(\gamma_{i,n}+C_4+z)^{2}}-
\frac{C_4 z}{(\gamma_{i,n}+C_4+z)^{2}}
\Big]\\
&
+
\sum_{i=1}^{M_5}\Big[\frac{C_3z\zeta_{i,n}(2\zeta_{i,n}+C_3-2z)}{(\zeta_{i,n}+C_3-z)^{2}(\zeta_{i,n}-z)^{2}}
-\frac{C_3z}{(\zeta_{i,n}+C_3-z)^{2}}\Big]
\Big);\end{align*}
\begin{align*}
S_{3,\infty}^{3}(z)&=
\frac{z}{3}\sum_{n=N+1}^{+\infty}\Big(
\sum_{i=1}^{M_6}\Big[\frac{C_4z^{2}\gamma_{i,n}}{(\gamma_{i,n}+z)^{3}(\gamma_{i,n}+C_4+z)}-
\frac{C_4 z^{2}}{(\gamma_{i,n}+C_4+z)^{3}}\\
&
+\frac{C_4z^{2}\gamma_{i,n}}{(\gamma_{i,n}+z)^{2}(\gamma_{i,n}+C_4+z)^{2}}
+\frac{C_4z^{2}\gamma_{i,n}}{(\gamma_{i,n}+z)(\gamma_{i,n}+C_4+z)^{3}}
\Big]\\
&
-
\sum_{i=1}^{M_5}\Big[\frac{C_3z^{2}\zeta_{i,n}}{(\zeta_{i,n}+C_3-z)^{3}(\zeta_{i,n}-z)}
+\frac{C_3z^{2}}{(\zeta_{i,n}+C_3-z)^{3}}\\&
+
\frac{C_3z^{2}\zeta_{i,n}}{(\zeta_{i,n}+C_3-z)^{2}(\zeta_{i,n}-z)^{2}}
+
\frac{C_3z^{2}\zeta_{i,n}}{(\zeta_{i,n}+C_3-z)^{3}(\zeta_{i,n}-z)}
\Big]
\Big);\end{align*}
\begin{align*}
S_{3,\infty}^{4}(z)&=\sum_{n=N+1}^{+\infty}\Big(
\sum_{i=1}^{M_6}\Big[
\gamma_{i,n}\Big(\frac{z}{\gamma_{i,n}+z}\Big)^{3}\sum_{j=4}^{+\infty}\frac{1}{j}\Big(\frac{z}{\gamma_{i,n}+z}\Big)^{j-3}\\&
-
(C_4+\gamma_{i,n})\Big(\frac{z}{\gamma_{i,n}+C_4+z}\Big)^{3}\sum_{j=4}^{+\infty}\frac{1}{j}\Big(\frac{z}{\gamma_{i,n}+z+C_4}\Big)^{j-3}
\Big]\\
&-
\sum_{i=1}^{M_5}\Big[
\zeta_{i,n}\Big(\frac{z}{\zeta_{i,n}-z}\Big)^{3}\sum_{j=4}^{+\infty}\frac{(-1)^{j-1}}{j}\Big(\frac{z}{\zeta_{i,n}-z}\Big)^{j-3}
\\
&
-
(C_3+\zeta_{i,n})\Big(\frac{z}{\zeta_{i,n}+C_3-z}\Big)^{3}\sum_{j=4}^{+\infty}\frac{(-1)^{j-1}}{j}\Big(\frac{z}{\zeta_{i,n}-z+C_3}\Big)^{j-3}
\Big]
\Big).
\end{align*}
Then, taking advantage of Proposition \ref{p3.1} to evaluate  the right-hand sides of these relations, we immediately end up with equality \eqref{3.14}.

In summary, collecting relations for $S_{j,N}(z)$ and $S_{j,\infty}(z)$, we complete the proof of asymptotic \eqref{2.5} in the case of $\beta=1$. Besides, in light of Subsection \ref{s3.2}, we finish the proof of Theorem \ref{t3.1} if $\beta=1$. \qed

\begin{remark}\label{3.10*}
Recasting arguments of Section \ref{s3.2} (see the proof of \eqref{3.6}) in the case of $\Omega_1=1$, we easily obtain the asymptotic \eqref{2.5} with $\Phi(z)=O(1)$ as $|Im z|\to+\infty$ and $Re z$ satisfying assumptions of Theorem \ref{t3.1}.
\end{remark}

\subsection{Conclusion of the Proof of  Theorem \ref{t3.1}}
\label{s3.3}
\noindent
In order to complete the proof of Theorem \ref{t3.1}, we are left to reduce equation \eqref{i.1} with $\beta\neq 1$ to equation \eqref{3.1}. To this end, we introduce new variable $y=z/\beta$ and new unknown function $V(y,\sigma)=\mathcal{Y}(\beta y,\sigma)$ and rewrite equation \eqref{i.1} as
\[
(a_1\sigma+a_2\sigma^{\nu})V(y+1,\sigma)-\Omega(\beta y)V(y,\sigma)=\F(\beta y, \sigma).
\]
After that, setting
\[
\bar{\Omega}(y)=\Omega(\beta y),\quad \bar{\F}(y,\sigma)=\F(\beta y,\sigma),
\]
we deduce
\begin{equation}\label{3.16}
(a_1\sigma+a_2\sigma^{\nu})V(y+1,\sigma)-\bar{\Omega}(y)V(y,\sigma)=\bar{\F}(y, \sigma).
\end{equation}
Then, putting
\begin{align*}
&\bar{\d}_{i}^{1}=\frac{\d_{i}^{1}}{\beta},\quad \bar{\d}_{i}^{2}=\frac{\d_{i}^{2}}{\beta},\quad
\bar{\d}_{i}^{3}=\frac{\d_{i}^{3}}{\beta},\, \bar{\d}_{i}^{4}=\frac{\d_{i}^{4}}{\beta},\quad
\bar{\d}_{0}=\d_{0}\beta^{M_1+M_2-M_3-M_4},\, \bar{A}=A\beta^{2},\,\bar{B}=B\beta,\\
&\bar{\gamma}_{i,n}=\frac{\gamma_{i,n}}{|\beta|},\,\bar{\zeta}_{i,n}=\frac{\zeta_{i,n}}{|\beta|},
\,\bar{\eta}_{i,n}=\frac{\eta_{i,n}}{|\beta|},\,\bar{\h}_{i,n}=\frac{\h_{i,n}}{|\beta|},
\end{align*}
we conclude that $\bar{\Omega}(y)$ possesses the properties of $\Omega(z)$ in the case of $\beta=1$.

Finally, setting $\bar{\F}(y,\sigma)\equiv 0$ in \eqref{3.16} and recasting the arguments from Sections \ref{s3.1}-\ref{s3.2}, we complete the proof of Theorem \ref{t3.1} in the general case of $\beta$. \qed


\section{Proof of Theorem \ref{t3.2}}
\label{s4}
\noindent
In light of Theorem \ref{t3.1} and arguments in Section \ref{s3.3}, in order to prove Theorem \ref{t3.2}, it is enough to verify that \eqref{2.11} is a particular solution of \eqref{i.1} in the case of $\beta=1$. First, we analyze the case of $d_0=1$, i.e. if $\ell_{d_0}=\ell_1$.

Being within assumptions of Theorem \ref{t3.2}, we fixate a suitable periodic function $\P_1:=\P(z)$ in \eqref{2.4*} and set
 $\mathcal{Y}_{h}(z,\sigma):=\mathcal{Y}_{h}(z,\sigma;\P_1)$ in \eqref{2.11}. Then, we substitute \eqref{2.11} in the left-hand side of \eqref{i.1}.
Taking into account that $\mathcal{Y}_{h}(z,\sigma)$
solves homogenous equation \eqref{i.1}, we end up with relations
\begin{align*}
&(a_1\sigma+a_2\sigma^{\nu})\mathcal{Y}_{ih}(z+1,\sigma)-\Omega(z)\mathcal{Y}_{ih}(z,\sigma)
\\&
=\frac{\mathcal{Y}_{h}(z+1,\sigma)}{\mathcal{K}_1(0)2i}\Big[
\int_{\ell_{1}}\frac{\F(z+1+\xi,\sigma)\mathcal{K}(\xi)}{\mathcal{Y}_{h}(z+2+\xi,\sigma)}d\xi
-
\int_{\ell_{1}}\frac{\F(z+\xi,\sigma)\mathcal{K}(\xi)}{\mathcal{Y}_{h}(z+1+\xi,\sigma)}d\xi\Big]\\
&
=
\frac{\mathcal{Y}_{h}(z+1,\sigma)}{\mathcal{K}_1(0)2i}\Big[
\int_{\ell_{0}}\frac{\F(z+\xi,\sigma)\mathcal{K}(\xi)}{\mathcal{Y}_{h}(z+1+\xi,\sigma)}d\xi
-
\int_{\ell_{1}}\frac{\F(z+\xi,\sigma)\mathcal{K}(\xi)}{\mathcal{Y}_{h}(z+1+\xi,\sigma)}d\xi\Big].
\end{align*}
In order to reach the last equality, we exploit assumptions \textbf{H4, H5}.
At last,
Cauchy's residue theorem arrives at the  desired equality
\[
(a_1\sigma+a_2\sigma^{\nu})\mathcal{Y}_{ih}(z+1,\sigma)-\Omega(z)\mathcal{Y}_{ih}(z,\sigma)
=\frac{\mathcal{Y}_{h}(z+1,\sigma)}{\mathcal{K}_1(0)}\pi\text{res}_{\xi=0}\frac{\F(z+\xi,\sigma)\mathcal{K}(\xi)}{\mathcal{Y}_{h}(z+1+\xi,\sigma)}
=\F(z,\sigma).
\]
It is worth noting that, the last relations are simple consequence of \textbf{H4-H5} and \eqref{2.6}, \eqref{2.7} and the fact of the function $\frac{\F(z+\xi,\sigma)}{\mathcal{Y}_{h}(z+1+\xi,\sigma)}$ has no poles if $Re\,\xi \in[-1,0]$ and $Re\, z$ satisfies \eqref{2.6} and \eqref{2.7}. Summarizing, we complete the proof of Theorem \ref{t3.2} if $\ell_{d_0}=\ell_1$ in \eqref{2.11}. We remark that the case $d_0=0$, i.e. $\ell_{d_{0}}=\ell_{0}$ is studied with the same arguments.

Finally, collecting \eqref{2.6}, \eqref{2.7} with Cauchy theorem, we arrive at the identity
\[
\int_{\ell_{1}}\frac{\F(z+\xi,\sigma)\mathcal{K}(\xi)}{(a_1\sigma+a_2\sigma^{\nu})\mathcal{Y}_{h}(z+1+\xi,\sigma)}d\xi=
\int_{\ell_{d_0}}\frac{\F(z+\xi,\sigma)\mathcal{K}(\xi)}{(a_1\sigma+a_2\sigma^{\nu})\mathcal{Y}_{h}(z+1+\xi,\sigma)}d\xi,\quad d_0\in(0,1),
\]
which means that \eqref{2.11} with $d_{0}\in(0,1)$ solves inhomogeneous equation \eqref{i.1}. That completes the proof of Theorem \ref{t3.2}. \qed

Observing the proof of Theorems \ref{t3.1} and \ref{t3.2}, we conclude that arguments of Sections \ref{s2}-\ref{s4} can be extended to equation \eqref{i.1} with  the coefficient $\Omega(z)$ containing  complex  sequences. Indeed, in this case, instead of \textbf{H2} and \textbf{H3}, we state the following  assumptions on the sequences.

\noindent\textbf{H6:} For each  $i$ infinite sequences $\{\h_{i,n}\}_{n=1}^{\infty}$, $\{\gamma_{i,n}\}_{i=1}^{\infty},$
$\{\zeta_{i,n}\}_{n=1}^{\infty}$ and $\{\eta_{i,n}\}_{n=1}^{\infty}$ are complex-valued and satisfy the second estimate in \eqref{2.2}. Moreover, the following relations hold
\begin{align*}
&0<Re\,\h_{i,n}<Re\,\h_{i,n+1},\, 0<Re\,\gamma_{i,n}<Re\,\gamma_{i,n+1},\, 0<Re\,\zeta_{i,n}<Re\,\zeta_{i,n+1},\, 0<Re\,\eta_{i,n}<Re\,\eta_{i,n+1};\\
&0\leq Im\,\h_{i,n}\leq Im\,\h_{i,n+1},\, 0\leq Im\,\gamma_{i,n}\leq Im\,\gamma_{i,n+1},\, 0\leq Im\,\zeta_{i,n}\leq Im\,\zeta_{i,n+1},\, 0\leq Im\,\eta_{i,n}\leq Im\,\eta_{i,n+1};\\
&
\sum_{n=1}^{\infty}\Big[\sum_{i=1}^{M_5}(\widehat{\h}_{i,n})^{-2}+\sum_{i=1}^{M_6}(\widehat{\zeta}_{i,n})^{-2}+
\sum_{i=1}^{M_7}(\widehat{\gamma}_{i,n})^{-2}+\sum_{i=1}^{M_8}(\widehat{\eta}_{i,n})^{-2}
\Big]<C,
\end{align*}
where for each  $i$ we set
\begin{align*}
&\widehat{\h}_{i,n}=\begin{cases}
Re\, \widehat{\h}_{i,n},\quad\text{if}\, Im\, \widehat{\h}_{i,n}\equiv const.\,\text{for all }n,\\
\min\{Re\, \widehat{\h}_{i,n},Im\, \widehat{\h}_{i,n}\}, \quad \text{otherwise,}
\end{cases}\quad
\widehat{\gamma}_{i,n}=\begin{cases}
Re\, \widehat{\gamma}_{i,n},\quad\text{if}\quad Im\, \widehat{\gamma}_{i,n}\equiv const.\,\text{for all }n,\\
\min\{Re\, \widehat{\gamma}_{i,n},Im\, \widehat{\gamma}_{i,n}\}, \quad \text{otherwise,}
\end{cases}
\\
&\widehat{\zeta}_{i,n}=\begin{cases}
Re\, \widehat{\zeta}_{i,n},\quad\text{if}\quad Im\, \widehat{\zeta}_{i,n}\equiv const.\,\text{for all }n,\\
\min\{Re\, \widehat{\zeta}_{i,n},Im\, \widehat{\zeta}_{i,n}\}, \quad \text{otherwise,}
\end{cases}\quad
\widehat{\eta}_{i,n}=\begin{cases}
Re\, \widehat{\eta}_{i,n},\quad\text{if}\quad Im\, \widehat{\eta}_{i,n}\equiv const.\,\text{for all }n,\\
\min\{Re\, \widehat{\eta}_{i,n},Im\, \widehat{\eta}_{i,n}\}, \quad \text{otherwise.}
\end{cases}
\end{align*}

\noindent\textbf{H7:} Finite sequences: $\{\d_{i}^{1}\}_{i=1}^{M_1},$ $\{\d_{i}^{2}\}_{i=1}^{M_2},$ $\{\d_{i}^{3}\}_{i=1}^{M_3},$
$\{\d_{i}^{4}\}_{i=1}^{M_4},$ are complex-valued and $\d_{i}^{1}\neq 0$
for all $i\in\{1,2,...,M_1\}$, $\d_{i}^{3}\neq 0$
for all $i\in\{1,2,...,M_3\}$.

After that, recasting the arguments of Sections \ref{s3} and \ref{s4} with nonessential modifications, we deduce the following result.
\begin{theorem}\label{t4.1}
Under assumptions \textbf{H1,H6,H7} and \eqref{2.3} and \eqref{2.4}, a general solution of homogenous equation \eqref{i.1} is given with \eqref{2.5} and, besides, statements \textbf{(s.1)} and \textbf{(s.2)} of Theorem \ref{t3.1} hold.

If, in addition, assumptions \textbf{H4-H5} and \eqref{2.6} and (2.7) hold, then a general solution of inhomogeneous equation \eqref{i.1} is given with \eqref{2.11} and \eqref{2.11*}.
\end{theorem}


\section{Solvability of other functional equations}
\label{s5}
\noindent
In this section, we will exploit the technique developed in Sections \ref{s2}-\ref{s4} to find solutions of another class of functional difference equations. To this end, we focus on the equation:
\begin{equation}\label{5.1}
(a_1\sigma+a_2\sigma^{\nu})Y(z+\beta,\sigma)-\s(z)Y(z,\sigma)=\F(z,\sigma),
\end{equation}
where $Y=Y(z,\sigma)$ is unknown function, quantities $a_1,$ $a_2$, $\sigma$, $\nu$, $\beta$ meet requirement \textbf{H1}, and  $\F$ is a given function specified below. Coming to the variable coefficient $\s:=\s(z),$ we will analyze here different kinds of this function.
Namely, introducing entire functions $\s_1:=\s_1(z)$ and $\s_2:=\s_2(z)$, $\s_1(z)\neq 0,$ and $\s_2(z)\neq 0,$ we will say that

\noindent $\bullet$ $\s(z)$ is a function of the first kind (\textbf{FFK}), if $\s(z)=\s_1(z)$;

\noindent $\bullet$ $\s(z)$ is called a function of the second kind (\textbf{FSK}), if $\s(z)=\frac{1}{\s_2(z)}$;

\noindent $\bullet$ and, finally, $\s(z)=\frac{\s_1(z)}{\s_2(z)}$ is a function of the third kind (\textbf{FTK}).

Here, we analyze \eqref{5.1} in the case of $\s(z)$ is either an entire function or is presented as a quotient of two entire functions.
It is apparent that, $\s(z)$ of the second kind is a particular case of the \textbf{FTK} case (i.e. with $\s_1(z)=1$).

In our analysis, we will exploit the following strategy. In the first step, appealing to properties of entire functions, we describe sufficient conditions on the function $\s(z)$, which allow us to present $\s(z)$ as a finite or an infinite product. Thus, we reduce equation \eqref{5.1} to \eqref{i.1} with the corresponding coefficient $\Omega(z)$ constructed by $\s(z)$. Besides, we give several examples of the function $\s(z)$ together with its factorization. Then, we apply Theorems \ref{t3.1}-\ref{t3.2} or \ref{t4.1} in
 order to solve equation \eqref{5.1}. Finally, we provide several examples of \eqref{5.1} (with explicit view of $\s(z)$ and $\F(z,\sigma)$)
illustrating the obtained results.

\subsection{Factorization of $\mathcal{S}(z)$}
\label{s5.1}
\noindent At this point, we start by specifying the functions $\s_1(z)$, $\s_2(z)$ and, accordingly, the function $\s(z)$. Let $\z_n$ and $\y_n$, $n\in\mathbb{N},$ be zeros of the functions $\s_1(z)$ and $\s_2(z)$, correspondingly, lying in the plane $Re\, z>0$; and let $\z_{-n}$ and $\y_{-n}$   be zeros of these functions being located in the plane $Re\, z<0$. The number of these zeros will be indicated below.
Denoting
\[
\widehat{\z}_{\pm n}=\begin{cases}
|Re\, \z_{\pm n}|,\quad\text{if}\, Im\, \z_{\pm n}\equiv const.\text{ for all n},\\
\min\{|Re\, \z_{\pm n}|, |Im\, \z_{\pm n}|\},\qquad \text{otherwise},
\end{cases}
\quad
\widehat{\y}_{\pm n}=\begin{cases}
|Re\, \y_{\pm n}|,\quad\text{if}\, Im\, \y_{\pm n}\equiv const.\text{ for all n},\\
\min\{|Re\, \y_{\pm n}|, |Im\, \y_{\pm n}|\},\qquad \text{otherwise},
\end{cases}
\]
we state our assumptions.

\noindent\textbf{H8 (Assumption on $\s_1$ and $\s_2$):} We assume that entire functions $\s_1(z)$ and $\s_2(z)$ have order $\kappa_{1}$ and $\kappa_2$, respectively, $0\leq\kappa_1\leq 2,$ $0\leq\kappa_2\leq 2.$
Besides, $\s_1$ and $\s_2$ have either a finite or an infinite number of zeros.

\noindent\textbf{H9 (Assumption on the zeros of $\s_1$ and $\s_2$):} We require that all zeros of $\s_1$ and $\s_2$ repeated according to multiplicity  are arranged as:
\[
0<|\z_n|\leq|\z_{n+1}|,\quad n\in\{1,2,..., N_1\},\quad
0<|\z_{-n}|\leq|\z_{-n-1}|,\quad n\in\{1,2,..., N_2\},
\]
\[
0<|\y_n|\leq|\y_{n+1}|,\quad n\in\{1,2,..., N_3\},\quad
0<|\y_{-n}|\leq|\y_{-n-1}|,\quad n\in\{1,2,..., N_4\}.
\]
We assume also that $\z_0=0$ is the $\mu_1-$ multiple root of $\s_1$, while $\y_0=0$ is the $\mu_2-$ multiple root of $\s_2$.
The case $\mu_1=0$ or/and $\mu_2=0$ being taken to mean $\s_1(0)\neq 0$ or/and $\s_2(0)\neq 0$.

\noindent\textbf{H10 (Assumption on the sequences of zeros):} In the case of $N_j=+\infty$, $j\in\{1,2,3,4\},$ we assume
that the sequences $\{|Re\,\z_{\pm n}|\}_{n=1}^{+\infty} $, $\{|Re\,\y_{\pm n}|\}_{n=1}^{+\infty} $ are
strictly monotonically increasing, and
\[
\underset{n\to+\infty}{\lim}|\z_{-n}|=+\infty,\quad \underset{n\to+\infty}{\lim}|\z_{n}|=+\infty,\quad
\underset{n\to+\infty}{\lim}|\y_{-n}|=+\infty,\quad \underset{n\to+\infty}{\lim}|\y_{n}|=+\infty,
\]
Concerning to  $\{|Im\,\z_{\pm n}|\}_{n=1}^{+\infty} $, $\{|Im\,\y_{\pm n}|\}_{n=1}^{+\infty} $, we assume either all terms of the corresponding  sequences are constant, or these sequence do not decrease.
In addition, we require the convergence of the  series: $\sum_{n=1}^{+\infty}\widehat{\z\,}^{-2}_{n},\, \sum_{n=1}^{+\infty}\widehat{\z\,}^{-2}_{-n}$, $\sum_{n=1}^{+\infty}\widehat{\y\,}^{-2}_{n},\, \sum_{n=1}^{+\infty}\widehat{\y\,}^{-2}_{-n}$ and, besides
\[
\sum_{n=1}^{+\infty}\Big|\frac{1}{\widehat{\z}_{n}}+\frac{1}{\widehat{\z}_{-n}}\Big|<+\infty\, \text{ in the \bf{FFK} case,}\,
\sum_{n=1}^{+\infty}\Big|\frac{1}{\widehat{\y}_{n}}+\frac{1}{\widehat{\y}_{-n}}\Big|<+\infty\, \text{ in the \bf{FSK} case,}
\]
\[\sum_{n=1}^{+\infty}\Big|\frac{1}{\widehat{\z}_{n}}+\frac{1}{\widehat{\z}_{-n}}+\frac{1}{\widehat{\y}_{n}}+\frac{1}{\widehat{\y}_{-n}}\Big|<+\infty\quad\text{in the \bf{FTK} case.}
\]
Denote by
\[
P_1(z)=
\begin{cases}
0,\qquad\qquad\qquad\qquad\qquad\text{if}\quad \sum_{n=1}^{+\infty}\Big(\frac{1}{|\z_n|}+\frac{1}{|\z_{-n}|}\Big)<+\infty,\\
z\sum_{n=1}^{+\infty}\Big(\frac{1}{|\z_n|}+\frac{1}{|\z_{-n}|}\Big),\quad\text{if}\quad \sum_{n=1}^{+\infty}\Big(\frac{1}{|\z_n|}+\frac{1}{|\z_{-n}|}\Big)=+\infty,
\end{cases}
\]
\[
P_2(z)=
\begin{cases}
0,\qquad\qquad\qquad\qquad\qquad\text{if}\quad \sum_{n=1}^{+\infty}\Big(\frac{1}{|\y_n|}+\frac{1}{|\y_{-n}|}\Big)<+\infty,\\
z\sum_{n=1}^{+\infty}\Big(\frac{1}{|\y_n|}+\frac{1}{|\y_{-n}|}\Big),\quad\text{if}\quad \sum_{n=1}^{+\infty}\Big(\frac{1}{|\y_n|}+\frac{1}{|\y_{-n}|}\Big)=+\infty.
\end{cases}
\]
The following result is a direct consequence of the Weierstrass factorization theorem \cite[Chapter 7, $\S$ 2.3]{Ma}.
\begin{lemma}\label{l5.1}
Let assumptions \textbf{H8-H10} hold. If $\s_1$ and $\s_2$ have a finite number of zeroes, then there are representations:
\begin{align*}
\s_1(z)&=A_{0}^{I}z^{\mu_1}(-1)^{N_1}\prod_{n=1}^{N_1}(\z_n-z)\prod_{n=1}^{N_2}(z-\z_{-n})\exp\{g_1(z)\},\\
\s_2(z)&=A_{0}^{II}z^{\mu_2}(-1)^{N_3}\prod_{n=1}^{N_3}(\y_n-z)\prod_{n=1}^{N_4}(z-\y_{-n})\exp\{g_2(z)\},
\end{align*}
with $g_1(z)$ and $g_2(z)$ being polynomials whose degree is less than integer part of $\kappa_1$ and $\kappa_2$, respectively.

In the case of infinite sequences of zeros to $\s_1$ and $\s_2$, the function $\s(z)$ is presented as
\begin{equation}\label{5.1*}
\s(z)=
\begin{cases}
z^{\mu_1}A_0^{I}e^{g_1(z)+P_1(z)}\prod_{n=1}^{+\infty}\frac{(\z_n-z)(z-\z_{-n})}{\z_{n}(-\z_{-n})}\qquad\qquad\qquad\qquad\qquad\qquad\text{in the \bf{FFK} case},\\
\,\\
z^{-\mu_2}A_0^{II}e^{-g_2(z)-P_2(z)}\prod_{n=1}^{+\infty}\frac{\y_{n}(-\y_{-n})}{(\y_n-z)(z-\y_{-n})}\quad\qquad\qquad\qquad\qquad\quad\text{in the \bf{FSK} case},\\
\,\\
z^{\mu_1-\mu_2}\frac{A_0^{I}}{A_0^{II}}e^{g_1(z)-g_2(z)+P_1(z)-P_2(z)}\prod_{n=1}^{+\infty}\frac{(\z_n-z)(z-\z_{-n})\y_n(-\y_{-n})}
{(\y_n-z)(z-\y_{-n})\z_{n}(-\z_{-n})}\quad\text{in the \bf{FTK} case}.
\end{cases}
\end{equation}
\end{lemma}
As a consequence of the above results, we are left with the problem of finding a factorization to  given below functions $\s(z)$. First, we start to analyze with \textbf{FFK} case. To this end, we denote
\begin{equation}\label{5.2**}
\theta_{+}:=\theta_{+}(\theta)=\begin{cases}
\theta,\qquad\text{if }\theta\in[0,\pi),\\
\theta-\pi,\quad\text{if }\theta\in[\pi,2\pi),
\end{cases}
\quad
\theta_{-}:=\theta_{-}(\theta)=\begin{cases}
\theta,\qquad\text{if }\theta\in[-\pi/2,\pi/2),\\
\theta-\pi,\quad\text{if }\theta\in[\pi/2,3\pi/2),\\
\theta-2\pi,\quad\text{if }\theta\in[3\pi/2,2\pi).
\end{cases}
\end{equation}
Moreover, for simplicity consideration, we belive
$
\frac{\sin\theta_{+}}{\theta_{+}}=1$ if $\theta_{+}=0$ or  $\pi$.

\begin{proposition}\label{p5.1}
Let $\theta\in[0,2\pi)$,  $\theta\neq \pm\pi/2,3\pi/2$, and $q_0\in\R$, $q_0\neq 0$. Then, there are the following relations:

\noindent\textbf{(i)}   $\sin(q_0z\pm\theta)=(-1)^{\lfloor\theta/\pi\rfloor}q_0\frac{\sin\theta_{+}}{\theta_{+}}\big(z\pm\frac{\theta_{+}}{q_0}\big)
    \prod\limits_{n=1}^{+\infty}\frac{\Big(\frac{\pi n\mp\theta_{+}}{q_0}-z\Big) \Big(\frac{\pi n\pm\theta_{+}}{q_0}+z\Big)}
    {\Big(\frac{\pi n +\theta_{+}}{q_0}\Big)\Big(\frac{\pi n -\theta_{+}}{q_0}\Big)},$

    where $\lfloor\theta\rfloor$ is the floor function of $\theta$ (i.e. the greatest integer less than or equal to $\theta$);

\noindent\textbf{(ii)}
        $\cos(q_0z\pm\theta)=(-1)^{\lfloor\frac{\theta}{\pi}+\frac{1}{2}\rfloor}\cos\theta_{-}
    \prod\limits_{n=1}^{+\infty}\frac{\Big(\frac{\pi(2 n-1)\mp 2\theta_{-}}{2q_0}-z\Big) \Big(\frac{\pi(2 n-1)\pm 2\theta_{-}}{2q_0}+z\Big)}
    {\Big(\frac{\pi(2n-1)+2\theta_{-}}{2q_0}\Big)\Big(\frac{\pi(2 n-1) -2\theta_{-}}{2q_0}\Big)},$

\noindent\textbf{(iii)} $
        \tan(q_0z\pm\theta)=(-1)^{\lfloor\frac{\theta}{\pi}\rfloor-\lfloor\frac{2\theta+\pi}{2\pi}\rfloor}
        \frac{\sin\theta_{+}}{\theta_{+}\cos\theta_{-}}
        \bigg(z\pm\frac{\theta_{+}}{q_0}\bigg)\prod\limits_{n=1}^{+\infty}
                \frac{\Big(\frac{\pi n\mp\theta_{+}}{q_0}-z\Big)}
        {\Big(\frac{\pi(2 n-1)\mp 2\theta_{-}}{2q_0}-z\Big)}    \frac{\Big(\frac{\pi n\pm\theta_{+}}{q_0}+z\Big)}
        {\Big(\frac{\pi (2n-1)\pm 2\theta_{-}}{2q_0}+z\Big)}$

        $\times
            \frac{\Big(\frac{\pi(2 n-1)-2\theta_{-}}{2q_0}\Big) \Big(\frac{\pi(2 n-1)+2\theta_{-}}{2q_0}\Big)}
        {\Big(\frac{\pi n-\theta_{+}}{q_0}\Big)\Big(\frac{\pi n+\theta_{+}}{q_0}\Big)}.
        $
                        \end{proposition}
\begin{proof}
First of all, we remark that in the case of $\theta=0$ and $q_0=1$, relations \textbf{(i)-(iii)} boil down to well-known decompositions of
$\sin z=z\prod_{n=1}^{+\infty}\Big(1-\frac{z^{2}}{n^{2}\pi^{2}}\Big),$
$\cos z=\prod_{n=1}^{+\infty}\Big(1-\frac{4z^{2}}{(2n-1)^{2}\pi^{2}}\Big)$ and their quotient. Actually, the proof of these statements is a simple consequence of Lemma \ref{l5.1}. Indeed, it is apparent that the functions $\sin(q_0z\pm\theta)$ and $\cos(q_0z\pm\theta)$ can be considered as the function $\s(z)$ of the \textbf{FFK} while  $\tan(q_0z\pm\theta)$ is the \textbf{FTK}.

We begin with getting representation (i) and (ii). In light of \cite[Chapter 7, \S1]{Ma}, we can conclude that $\sin(q_0z\pm\theta)$ and $\cos(q_0z\pm\theta)$ have the order $1$, and all their zeros are given by:

\noindent$\bullet$ $\z_{0}=\mp\frac{\theta}{q_0},$ $\z_{n}=\frac{\pi n\mp\theta_{+}}{q_0},$ $\z_{-n}=\frac{\pi n\pm\theta_{+}}{q_0}$, $n\in\mathbb{N},$ in the case of $\sin(q_0z\pm\theta),$

\noindent$\bullet$ $\y_{n}=\frac{\pi(2 n-1)\mp\theta_{-}}{2q_0},$ $\y_{-n}=-\frac{\pi(2 n-1)\pm 2\theta_{-}}{2q_0}$, $n\in\mathbb{N},$ in the case of $\cos(q_0z\pm\theta).$

These values tell us  that assumptions \textbf{H8-H10} are satisfied  and therefore we can apply Lemma \ref{l5.1}. As a result, we have
\begin{align}\label{5.2}\notag
\sin(q_0z\pm\theta)&=A_{0}^{\pm}e^{z\Big(A_{1}^{\pm}\pm
\sum_{n=1}^{+\infty}\frac{2q_0\theta_{+}}{n^{2}\pi^{2}-\theta_{+}^{2}}\Big)}
\Big(z\pm\frac{\theta_{+}}{q_0}\Big)
\prod_{n=1}^{+\infty}
\frac{\bigg(\frac{\pi n\mp\theta_{+}}{q_0}-z\bigg)\bigg(\frac{\pi n\pm\theta_{+}}{q_0}+z\bigg)}
{\bigg(\frac{\pi n+\theta_{+}}{q_0}\bigg)\bigg(\frac{\pi n-\theta_{+}}{q_0}\bigg)},
\\
\cos(q_0z\pm\theta)&=B_{0}^{\pm}
e^{z\Big(B_{1}^{\pm}\mp
\sum_{n=1}^{+\infty}\frac{4q_0\theta_{-}}{(2n-1)^{2}\pi^{2}-4\theta_{-}^{2}}\Big)}
\prod_{n=1}^{+\infty}
\frac{\bigg(\frac{\pi(2 n-1)\mp 2\theta_{-}}{2q_0}-z\bigg)\bigg(\frac{\pi(2 n-1)\pm 2\theta_{-}}{2q_0}+z\bigg)}
{\bigg(\frac{\pi(2 n-1)-2\theta_{-}}{2q_0}\bigg)\bigg(\frac{\pi(2 n-1)+2\theta_{-}}{2q_0}\bigg)}
\end{align}
where numbers $A_{0}^{\pm},$ $A_{1}^{\pm},$ $B_{0}^{\pm}$ and $B_{1}^{\pm}$ are unknown coefficients which will be identified below.

To this end, we will exploit the easy verified identities:
$
\sin(q_0 z\pm\theta)=(-1)^{\lfloor \theta/\pi\rfloor}\sin(q_0 z\pm\theta_{+}),$ $
\cos(q_0 z\pm\theta)=(-1)^{\lfloor \frac{2\theta+\pi}{2\pi}\rfloor}\cos(q_0 z\pm\theta_{-}),
$
if $\theta\in[0,2\pi)$.
Then, substituting $z=0$ to relations \eqref{5.2} and taking into account identities above, we arrive at the equalities
\begin{equation}\label{5.2*}
A_{0}^{\pm}=\frac{(-1)^{\lfloor\theta/\pi\rfloor}q_0\sin\theta_{+}}{\theta_{+}}\quad
\text{and}\quad
B_{0}^{\pm}=(-1)^{\lfloor\frac{2\theta+\pi}{2\pi}\rfloor}\cos\theta_{-}.
\end{equation}
It is apparent, that the proof of statement (i) and (ii) of this proposition follows immediately from the relations:
$$
A_{1}^{\pm}=\mp\sum_{n=1}^{+\infty}\frac{2\theta_{+}q_0}{n^{2}\pi^{2}-\theta^{2}_{+}} \quad\text{and}\quad
B_{1}^{\pm}=\pm\sum_{n=1}^{+\infty}\frac{4\theta_{-}q_0}{(2n-1)^{2}\pi^{2}-4\theta^{2}_{-}},
$$
which will be verified here below.

Here, we will carry out the detailed proof of the equality to $A_{1}^{+}$. The proof of the other identities is almost identical. Next, we take advantage of the straightforward relations:
\begin{align}\label{5.3*}\notag
\underset{z\to-\frac{\theta_{+}}{q_0}}{\lim}\frac{\sin(q_{0}z+\theta)}{q_0z+\theta_{+}}&=(-1)^{\lfloor\frac{\theta}{\pi}\rfloor},\quad
\underset{z\to\frac{\pi-\theta_{+}}{q_0}}{\lim}\,\frac{\sin(q_{0}z+\theta)}{q_0z-\pi+\theta_{+}}=(-1)^{1+\lfloor\frac{\theta}{\pi}\rfloor},\quad
\\
\sum_{n=1}^{+\infty}\frac{2q_{0}\theta_{+}}{\pi^{2}n^{2}-\theta_{+}^{2}}&
=-\frac{q_{0}}{\pi}\int_{0}^{1}\frac{x^{\theta_{+}/\pi}-x^{-\theta_{+}/\pi}}{1-x}dx=
-\frac{q_0}{\theta_{+}}+q_0\cot\theta_{+}
.
\end{align}
It is worth noting that, the last equality  is consequence of formulas (0.244(1)) and (3.231(3)) in \cite{GR}. Then, substituting $z=\frac{-\theta_{+}}{q_0}$ and $z=\frac{\pi-\theta_{+}}{\theta_{q_{0}}}$ in the representation for $\sin(q_0z+\theta)$ in \eqref{5.2} and taking into account the relations above, we deduce
\[
\begin{cases}
(-1)^{\lfloor\theta/\pi\rfloor}=\frac{\sin\theta_{+}}{\theta_{+}}(-1)^{\lfloor\theta/\pi\rfloor}\prod_{n=1}^{+\infty}\frac{(\pi n)^{2}}{\pi^{2}n^{2}-\theta_{+}^{2}}\exp\Big\{-\frac{\theta_{+}}{q_0}[A_1^{+}+q_{0}\cot\theta_{+}-q_{0}/\theta_{+}]\Big\},\\
(-1)^{1+\lfloor\theta/\pi\rfloor}=\frac{\sin\theta_{+}}{\theta_{+}}(-1)^{1+\lfloor\theta/\pi\rfloor}\frac{\pi\prod_{n=1}^{+\infty}\pi( n+1)\prod_{n=2}^{+\infty}\pi( n-1)}{\prod_{n=1}^{+\infty}(\pi^{2}n^{2}-\theta_{+}^{2})}\exp\Big\{\frac{\pi-\theta_{+}}{q_0}[A_1^{+}+q_{0}\cot\theta_{+}-q_{0}/\theta_{+}]\Big\}.
\end{cases}
\]
Thus, we end up with the equality
\begin{equation}\label{5.3**}
\exp\Big\{\frac{\pi}{q_0}[A_1^{+}+q_{0}\cot\theta_{+}-q_{0}/\theta_{+}]\Big\}=1,
\end{equation}
which provides the desired relations to $A_1^{+}$. This completes the proof of statements \textbf{(i)} and \textbf{(ii)} in this claim.

Finally, the representation in \textbf{(iii)} is simple consequence of Lemma \ref{5.1} and relations in \textbf{(i)} and \textbf{(ii)} of this proposition. Hence, the proof of Proposition \ref{p5.1} is finished.
\end{proof}

Our next results are connected with a factorization of the functions:
\begin{align*}
\s^{+}&:=\s^{+}(z;\theta_{1},\theta_2,q_1,q_2)=\sin(z-\theta_1)+q_2\sin(q_1z-\theta_2),\\
\s^{-}&:=\s^{-}(z;\theta_{1},\theta_2,q_1,q_2)=\sin(z-\theta_1)-q_2\sin(q_1z-\theta_2),
\end{align*}
and their quotient.
Here $\theta_1$ and $\theta_2\in[0,2\pi)$, $q_2$ and $q_1\in\R$, $q_2\neq 0,$ $q_1\neq 0$.
First, denoting the zeros of the functions  $\s^{\pm}$ in the segment $[0,4\pi q),$ with fixed $q\in\mathbb{N}$ by
\[
z_{i}^{\pm}:=z_{i}^{\pm}(q_1,q_2,\theta_1,\theta_2),\quad i\in\{0,1,...,\K^{\pm}\},
\]
and introducing the sets for any integer $k$:
\begin{align*}
\ss(k,\theta_1,q_2)&=[\theta_1-\arcsin q_2+2\pi k;\theta_1+\arcsin q_2+2\pi k]
\\&\cup
[\theta_1+\pi-\arcsin q_2+2\pi k;\theta_1+\pi+\arcsin q_2+2\pi k],\\
\ss^{\star}(k,\theta_2,q_1)&=\Big[\frac{\pi k+\theta_2}{q_1}-\frac{\pi}{2q_1};\frac{\pi k+\theta_2}{q_1}+\frac{\pi}{2q_1}\Big],
\end{align*}
we assert the following result.
\begin{corollary}\label{c5.1}
Let $\theta_{1},\theta_2\in[0,\pi]$, $q_2,q_1\in\R$,  $p,q\in\N$ and let $q/p$ be not reducible fraction. We assume that
$$q_2>0,\quad q_2\neq 1, \quad\text{and}\quad p>2q>1, \quad q_1=p/2q.$$
Then there are the following statements:
\begin{description}
    \item[i] $\s^{\pm}(z;\theta_1,\theta_2,q_1,q_2)$ are entire periodic functions with the period $\T_q=4\pi q.$ Besides, if $\theta_i=0$ or $\pi,$ $i=1,2,$ these functions are odd, while in the case of $\theta_1=\theta_2=\pi/2$, $\s^{\pm}$ are even.
    \item[ii] All zeros $\z^{\pm}_{i}$ of $\s^{\pm}$ are real and their number in the segment $[0,4\pi q)$  equals to  $\K^{\pm}:=\K^{\pm}(\theta_1,\theta_2,$ $q_1, q_2)$.
    \item[iii] If $q_2>1$, then $\K^{\pm}=2p$ and all zeros are simple and satisfy relations:
\begin{align*}
0&\leq\z_0^{\pm}<\z_1^{\pm}<...<\z_{\K^{\pm}-1}^{\pm}<4\pi q,\quad
 \z_{i}^{+}\in\ss^{\star}(i,\theta_{2},q_1),\\
\z_{i}^{-}&\in
\begin{cases}
\ss^{\star}(i-1,\theta_2,q_1)\quad\text{if } \theta_2\in(\pi/2,\pi),\, \theta_1\in(0,\pi)\, \text{and }q_2\sin\theta_2<\sin\theta_1;\\
\ss^{\star}(i+1,\theta_2,q_1)\quad\text{if } \theta_2\in(0,\min\{q_1\theta_1,\pi/2\}),\, \theta_1\in(0,\pi)\, \text{and }q_2\sin\theta_2<\sin\theta_1;\\
\ss^{\star}(i,\theta_2,q_1)\quad\text{ otherwise}
\end{cases}
\end{align*}
for each $i\in\{0,1,...,2p-1\}.$
    \item[iv] If $q_2<1$ and $q_2\neq 1/q_1,$ then all zeros of $\s^{\pm}(z;\theta_1,
    \theta_2,q_1,q_2)$ are simple and strictly increasing, i.e. $\z_{i}^{\pm}<\z_{i+1}^{\pm}$, and $\K^{\pm}\in[4q,2p]$, while in the case of $q_2=1/q_1$ the zeros are non-decreasing and  some of them  may have the third order,  $\K^{\pm}\in(4q,6p)$. Moreover,
    $
    \z_{i}^{\pm}\in\bigcup_{k=l_1}^{l_2}\ss(k,\theta_1,q_2),$ $i\in\{0,1,...,\K^{\pm}-1,\}
    $   where
\[
    l_1=\begin{cases}
    0,\qquad\text{if }0<\theta_1+\arcsin q_2\leq \pi,\\
    -1,\quad\text{if }\pi<\theta_1+\arcsin q_2\leq 3\pi/2,
    \end{cases}
    \quad
    l_2=\begin{cases}
    2q-1,\quad\text{if }0<\theta_1-\arcsin q_2\leq \pi,\\
    2q,\qquad\text{ if }0\leq \arcsin q_2-\theta_1\leq \pi.
    \end{cases}
    \]
    \item[v] All zeros of the function $\s^{+}$ are given with
\[
    \z_{i,0}^{+}=\z_{i}^{+},\quad \z_{i,n}^{+}=\z_{i}^{+}+n\T_q,\quad
    \z_{i,-n}^{+}=
    \begin{cases}
    \z_{i}^{+}-n\T_q,\, \text{if }\theta_{1},\theta_2\in(0,\pi),\\
    -\z_{i,n}^{+},\quad\text{ if }\theta_1,\theta_2=0\text{ or }\pi,
    \end{cases}
\]
    for each $i\in\{0,1,...,\K^{+}-1\},$ and $ n\in\N,$  while the zeros of $\s^{-}$ are presented with
    \[
    \z_{i,0}^{-}=\z_{i}^{-},\quad
    \z_{i,n}^{-}=\z_{i}^{-}+n\T_q,\quad
    \z_{i,-n}^{-}=
    \begin{cases}
    \z_{i}^{-}-n\T_q,\quad\text{if }\theta_{1},\theta_2\in(0,\pi),\\
    -\z_{i,n}^{-},\quad\quad\text{ if }\theta_1,\theta_2=0\text{ or }\pi,
    \end{cases}
    \]
for each $i\in\{0,1,...,\K^{-}-1\},$ and $ n\in\N,$
        \item[vi] $\z_{0}^{\pm}=0$ solves the equation  $\s^{\pm}(z;\theta_1,\theta_2,q_1,q_2)=0$ iff the equality holds
    \begin{equation}\label{5.3}
    \sin\theta_1\pm q_2\sin\theta_2=0.
    \end{equation}
Besides, $\z=\theta_1\pm n\T_q, $ $n\in\N$ are zeros of $\s^{\pm}$ if
\[
\frac{|q_1\theta_1-\theta_2|}{2\pi}-\bigg\lfloor\frac{|q_1\theta_1-\theta_2|}{2\pi}\bigg\rfloor =0.
\]
\item[vii] If either $\theta_1$ or $\theta_2\in(\pi,2\pi],$ then the relations hold
\[
\s^{\pm}(z;\theta_1,\theta_2,q_1,q_2)=
\begin{cases}
-\s^{\pm}(z;\hat{\theta}_1,\hat{\theta}_2,q_1,q_2),\quad\text{if } \theta_1,\theta_2\in(\pi,2\pi],\\
\s^{\pm}(z;\theta_1,\hat{\theta}_2,q_1,q_2),\qquad\text{if } \theta_1\in[0,\pi],\,\theta_2\in(\pi,2\pi],\\
 -\s^{\mp}(z;\hat{\theta}_1,\theta_2,q_1,q_2),\quad\text{if } \theta_1\in(\pi,2\pi],\,\theta_2\in[0,\pi],
\end{cases}
\]
where $\hat{\theta}_1=\theta_1-\pi,$ $\hat{\theta}_2=\theta_2-\pi.$
\end{description}
\end{corollary}
\begin{proof}
This claim in the case of $\theta_{1},\theta_2=0$ or $\pi$ is proved in \cite[Proposition 3.1]{BV5}. The case of $\s^{-}(z;\theta_1,\theta_2,$ $q_1,q_2)$ is analyzed in \cite[Section 3.1]{BV1}. Thus, recasting the same arguments in the case of $\s^{+}(z;\theta_1,\theta_2,q_1,q_2)$ and performing standard technical calculations, we complete the proof of Corollary \ref{c5.1}.
\end{proof}

 Setting
\[
\mu_1^{+}=\begin{cases}
1,\quad\text{if }z=0\text{ is a simple root of }\s^{+},\\
3,\quad\text{if }z=0\text{ is a 3-multiple root of }\s^{+},
\end{cases}
\quad
\mu_1^{-}=\begin{cases}
1,\quad\text{if }z=0\text{ is a simple root of }\s^{-},\\
3,\quad\text{if }z=0\text{ is a 3-multiple root of }\s^{-},
\end{cases}
\]
we assert the following result which is simple consequence of Lemma \ref{l5.1} and Corollary \ref{c5.1}.
\begin{proposition}\label{p5.2}
Let assumptions of Corollary \ref{c5.1} hold. Then, there are the following decompositions.
\begin{description}
    \item[i] If $\sin\theta_1-q_2\sin\theta_2\neq 0$ and $\theta_1,\theta_2\in(0,\pi)$, then
    \[
    \s^{-}(z;\theta_1,\theta_2,q_1,q_2)=-[\sin\theta_1-q_2\sin\theta_2]\prod\limits_{i=0}^{\K^{-}-1}\frac{[\z_{i}^{-}-z]}{\z_{i}^{-}}
    \prod\limits_{n=1}^{\infty}\prod\limits_{i=0}^{\K^{-}-1}\frac{[\z_{i,n}^{-}-z][-\z_{i,-n}^{-}+z]}{\z_{i,n}^{-}(-\z_{i,-n}^{-})}.
    \]
    In the case of  $\sin\theta_1-q_2\sin\theta_2=0$ and $\theta_1,\theta_2\in(0,\pi)$, there holds
    \begin{align*}
    \s^{-}(z;\theta_1,\theta_2,q_1,q_2)&=(-1)^{\lfloor\mu_{1}^{-}/2\rfloor} z^{\mu_1^{-}}[\cos\theta_1-q_1^{\mu_1^{-}}q_2\cos\theta_2]\prod\limits_{i=\mu_1^{-}+1}^{\K^{-}-1}\frac{[\z_{i}^{-}-z]}{\z_{i}^{-}}\\
    &
    \times
    \prod\limits_{n=1}^{\infty}\prod\limits_{i=0}^{\K^{-}-1}\frac{[\z_{i,n}^{-}-z][-\z_{i,-n}^{-}+z]}{\z_{i,n}^{-}(-\z_{i,-n}^{-})}.
    \end{align*}
    \item[ii]
        If $\sin\theta_1+q_2\sin\theta_2\neq 0$ and $\theta_1,\theta_2\in(0,\pi)$, then
    \[
    \s^{+}(z;\theta_1,\theta_2,q_1,q_2)=-[\sin\theta_1+q_2\sin\theta_2]\prod\limits_{i=0}^{\K^{+}-1}\frac{[\z_{i}^{+}-z]}{\z_{i}^{+}}
    \prod\limits_{n=1}^{\infty}\prod\limits_{i=0}^{\K^{+}-1}\frac{[\z_{i,n}^{+}-z][-\z_{i,-n}^{+}+z]}{\z_{i,n}^{+}(-\z_{i,-n}^{+})}.
    \]
    In the case of  $\sin\theta_1+q_2\sin\theta_2=0$ and $\theta_1,\theta_2\in(0,\pi)$, there holds
    \begin{align*}
    \s^{+}(z;\theta_1,\theta_2,q_1,q_2)&=(-1)^{\lfloor\mu_{1}^{+}/2\rfloor} z^{\mu_1^{+}}[\cos\theta_1+q_1^{\mu_1^{+}}q_2\cos\theta_2]\prod\limits_{i=\mu_1^{+}+1}^{\K^{+}-1}\frac{[\z_{i}^{+}-z]}{\z_{i}^{+}}\\
    &
    \times
    \prod\limits_{n=1}^{\infty}\prod\limits_{i=0}^{\K^{+}-1}\frac{[\z_{i,n}^{+}-z][-\z_{i,-n}^{+}+z]}{\z_{i,n}^{+}(-\z_{i,-n}^{+})}.
    \end{align*}
    \item[iii] If $\theta_1=\theta_2=0$, then
    \[
    \s^{+}(z;0,0,q_1,q_2)=(1+q_1q_2)z
\prod\limits_{i=1}^{\K^{+}-1}\frac{(\z_{i}^{+}-z)(\z_{i}^{+}+z)}{(\z_{i}^{+})^{2}}
\prod\limits_{n=1}^{\infty}\prod\limits_{i=0}^{\K^{+}-1}\frac{[\z_{i,n}^{+}-z][\z_{i,n}^{+}+z]}{(\z_{i,n}^{+})^{2}},\]
\[
\s^{-}(z;0,0,q_1,1/q_1)=-(1-q_1^{3}q_2)z^{3}
\prod\limits_{i=4}^{\K^{-}-1}\frac{(\z_{i}^{-}-z)(\z_{i}^{-}+z)}{(\z_{i}^{-})^{2}}
\prod\limits_{n=1}^{\infty}\prod\limits_{i=0}^{\K^{-}-1}\frac{[\z_{i,n}^{-}-z][\z_{i,n}^{-}+z]}{(\z_{i,n}^{-})^{2}},\]
\[
\s^{-}(z;0,0,q_1,q_2)=(1-q_1q_2)z
\prod\limits_{i=1}^{\K^{-}-1}\frac{(\z_{i}^{-}-z)(\z_{i}^{-}+z)}{(\z_{i}^{-})^{2}}
\prod\limits_{n=1}^{\infty}\prod\limits_{i=0}^{\K^{-}-1}\frac{[\z_{i,n}^{-}-z][\z_{i,n}^{-}+z]}{(\z_{i,n}^{-})^{2}},\quad \text{if }\, q_2q_1\neq 1.
    \]
        Besides, the equalities are fulfilled
    \begin{align*}
    \s^{\pm}(z;\pi,\pi,q_1,q_2)&=-\s^{\pm}(z;0,0,q_1,q_2),\quad \s^{\pm}(z;0,\pi,q_1,q_2)=\s^{\mp}(z;0,0,q_1,q_2),\\
    \s^{\pm}(z;\pi,0,q_1,q_2)&=-\s^{\mp}(z;0,0,q_1,q_2).
    \end{align*}

    \item[iv] If either $\theta_1\in(\pi,2\pi)$ or $\theta_2\in(\pi,2\pi)$ then
    \[
    \s^{-}(z;\theta_1,\theta_2,q_1,q_2)=
    \begin{cases}
    -\s^{-}(z;\theta_1-\pi,\theta_2-\pi,q_1,q_2),\quad\text{if }\theta_1,\theta_2\in(\pi,2\pi),\\
    \s^{+}(z;\theta_1,\theta_2-\pi,q_1,q_2),\quad\text{if }\theta_1\in[0,\pi],\, \theta_2\in(\pi,2\pi),\\
    -\s^{+}(z;\theta_1-\pi,\theta_2,q_1,q_2),\quad\text{if }\theta_1\in(\pi,2\pi),\, \theta_2\in[0,\pi],
    \end{cases}
    \]
        \[
    \s^{+}(z;\theta_1,\theta_2,q_1,q_2)=
    \begin{cases}
    -\s^{+}(z;\theta_1-\pi,\theta_2-\pi,q_1,q_2),\quad\text{if }\theta_1,\theta_2\in(\pi,2\pi),\\
    \s^{-}(z;\theta_1,\theta_2-\pi,q_1,q_2),\quad\text{if }\theta_1\in[0,\pi],\, \theta_2\in(\pi,2\pi),\\
    -\s^{-}(z;\theta_1-\pi,\theta_2,q_1,q_2),\quad\text{if }\theta_1\in(\pi,2\pi),\, \theta_2\in[0,\pi].
    \end{cases}
    \]
    \item[v] For $\theta_1,\theta_2\in[0,2\pi)$ there are relations:
    \begin{align*}
    \sin(z+\theta_1)-q_2\sin(q_1z+\theta_2)&=\s^{-}(z;2\pi-\theta_1,2\pi-\theta_2,q_1,q_2);\\
    \sin(z+\theta_1)+q_2\sin(q_1z+\theta_2)&=\s^{+}(z;2\pi-\theta_1,2\pi-\theta_2,q_1,q_2);\\
    \cos(z-\theta_1)\pm q_2\cos(q_1z-\theta_2)&=
    \begin{cases}
    -\s^{\pm}(z;\frac{\pi}{2}+\theta_1,\frac{\pi}{2}+\theta_2,q_1,q_2),\text{ if }\theta_1,\theta_2\in[0,3\pi/2),\\
    -\s^{\pm}(z;\theta_1-\frac{3\pi}{2},\theta_2-\frac{3\pi}{2},q_1,q_2),\text{ if }\theta_1,\theta_2\in[3\pi/2,2\pi),\\
    -\s^{\pm}(z;\frac{\pi}{2}+\theta_1,\theta_2-\frac{3\pi}{2},q_1,q_2),\text{ if }\theta_1\in[0,3\pi/2),\,\theta_2\in[3\pi/2,2\pi),\\
    -\s^{\pm}(z;\theta_1-\frac{3\pi}{2},\theta_2+\frac{\pi}{2},q_1,q_2),\text{ if }\theta_2\in[0,3\pi/2),\,\theta_1\in[3\pi/2,2\pi).
    \end{cases}
    \end{align*}
\end{description}
\end{proposition}
\begin{proof}
It is worth noting that statement (i) in the case of $\sin\theta_1-q_2\sin\theta_2\neq 0$ is proved in \cite[Section 3]{BV5}, while statement (iii) concerning the function $\s^{-}(z;0,0,q_1,q_2)$ is verified in \cite[Proposition 3.1]{BV1}. It is apparent that the proof of statement (i) is a simple consequence of Lemma \ref{l5.1} and Corollary \ref{c5.1}. As for statements (ii) and (iii), they are examined with the arguments exploiting in the proof of claim (i). Coming statements (iv) and (v), they are follows from properties of the functions $\sin z,$ $\cos z$ and standard calculations.

Thus, in order to prove Proposition \ref{p5.2}, it is enough to verify statements in the point (i). First, we analyze the case of $\sin\theta_1-q_2\sin\theta_2\neq 0$ and
$\theta_1,\theta_2\in(0,\pi)$. By virtue of (i) in Corollary \ref{c5.1}, we easily  conclude that $\s^{-}(z;\theta_1,\theta_2,q_1,q_2)$ is the function $\s(z)$ of \textbf{FFK} (see Lemma \ref{l5.1}). Hence, taking into account Corollary \ref{c5.1} and the decomposition of $\s(z)$ in Lemma \ref{l5.1}, we obtain the equality
\begin{equation}\label{5.4}
\s^{-}(z;\theta_1,\theta_2,q_1,q_2)=A_0\exp\{A_1 z+P_1(z)\}\prod\limits_{i=0}^{\K^{-}-1}\frac{[\z_{i}^{-}-z]}{\z_{i}^{-}}
\prod\limits_{n=1}^{+\infty}\prod\limits_{i=0}^{\K^{-}-1}\frac{[\z_{i,n}^{-}-z][z-\z_{i,-n}^{-}]}{\z_{i,n}^{-}(\z_{i,-n}^{-})}
\end{equation}
with
$
P_1(z)=z\sum_{n=1}^{+\infty}\sum_{i=0}^{\K^{-}-1}\Big(\frac{1}{\z_{i,n}^{-}}+\frac{1}{\z_{i,-n}^{-}}\Big),
$
and unknown coefficients $A_0$ and $A_1$ which will be identified below. In order to compute the series in $P_1(z),$ we employ the arguments applying in calculation of $A_{1}^{\pm}$ (see \eqref{5.3*}) in the proof of Proposition \ref{p5.1}. Thus, we have
\[
P_1(z)=z\sum_{i=0}^{\K^{-}-1}\frac{2\z_{i}^{-}}{\T_q^{2}}\sum_{n=1}^{+\infty}\frac{\T_q^{2}}{n^{2}\T_q^{2}-(\z^{-}_{i})^{2}}
=-z\sum_{i=0}^{\K^{-}-1}\Big[\frac{1}{\z_i^{-}}-\frac{\pi\cot(\pi\z_i^{-}/\T_q)}{\T_q}\Big].
\]
As a result, we are left to check that the relations hold
\begin{equation}\label{5.5}
A_{0}=-[\sin\theta_1-q_2\sin\theta_2],\quad A_1=\sum_{i=0}^{\K^{-}-1}\Big[\frac{1}{\z_i^{-}}-\frac{\pi\cot(\pi\z_i^{-}/\T_q)}{\T_q}\Big].
\end{equation}
To this end, we recast the arguments from the proof of Proposition \ref{p5.1} which provides the explicit values of $A_0^{\pm}$ and $A_1^{\pm}$
(see \eqref{5.2*} and \eqref{5.3**}).

\noindent$\bullet$ Indeed, substituting $z=0$ to \eqref{5.4} and taking into account the inequality $\sin\theta_1-q_2\sin\theta_2\neq 0$, we end up with the first equality in \eqref{5.5}.

\noindent$\bullet$ Coming to the equality for finding $A_1$, let us first write (by means of (i) in Proposition \ref{p5.1})
\[
\frac{[z-\z_i^{-}]}{\z_i^{-}}\prod\limits_{n=1}^{+\infty}\frac{[\z_{i,n}^{-}-z][z-\z_{i,-n}^{-}]}{\z_{i,n}^{-}(-\z_{i,-n}^{-})}
=\frac{\sin\frac{z-\z_{i}^{-}}{4q}}{\sin\frac{\z_{i}^{-}}{4q}},\quad i\in\{0,1,2,...,\K^{-}-1\}.
\]
 Then, simple tedious calculations entail
\begin{align*}
&\s^{-}(2\T_q;\theta_1,\theta_2,q_1,q_2)=-[\sin\theta_1-q_2\sin\theta_2]=\s^{-}(4\T_q;\theta_1,\theta_2,q_1,q_2),\\
&\frac{[z-\z_i^{-}]}{\z_i^{-}}\prod\limits_{n=1}^{+\infty}\frac{[\z_{i,n}^{-}-z][z-\z_{i,-n}^{-}]}{\z_{i,n}^{-}(-\z_{i,-n}^{-})}\Big|_{z=2\T_q}=-1=
\frac{[z-\z_i^{-}]}{\z_i^{-}}\prod\limits_{n=1}^{+\infty}\frac{[\z_{i,n}^{-}-z][z-\z_{i,-n}^{-}]}{\z_{i,n}^{-}(-\z_{i,-n}^{-})}\Big|_{z=4\T_q}.
\end{align*}
Finally, substituting $z=2\T_q$ and $z=4\T_q$ to \eqref{5.4} and using relations above, we arrive at the system
\[
\begin{cases}
(-1)^{2\K^{-}}\exp\Big\{2\T_q\Big(A_1-\sum_{i=0}^{\K^{-}-1}\Big[\frac{1}{\z_{i}^{-}}-\frac{\pi\cot\frac{\pi\z_{i}^{-}}{\T_q}}{\T_q}\Big]\Big)\Big\}=1,\\
(-1)^{2\K^{-}}\exp\Big\{4\T_q\Big(A_1-\sum_{i=0}^{\K^{-}-1}\Big[\frac{1}{\z_{i}^{-}}-\frac{\pi\cot\frac{\pi\z_{i}^{-}}{\T_q}}{\T_q}\Big]\Big)\Big\}=1.
\end{cases}
\]
In conclusion, solving this system, we reach the second equality in \eqref{5.5}.

Now, we treat the function $\s^{-}$ if $\sin\theta_1-q_2\sin\theta_2=0$. In this case, Corollary \ref{c5.1} tells us that
$
\s^{-}(0;\theta_1,\theta_2,q_1,q_2)=0.
$
Then, Lemma \ref{l5.1} provides the relation
\[
\s^{-}(z;\theta_1,\theta_2,q_1,q_2)=B_0\exp\{B_1 z+P_1(z)\}z^{\mu_1^{-}}\prod\limits_{i=\mu_{1}^{-}+1}^{\K^{-}-1}\frac{(\z_i^{-}-z)}{\z_{i}^{-}}\prod\limits_{n=1}^{+\infty}\prod\limits_{j=0}^{\K^{-}-1}\frac{[\z_{j,n}^{-}-z][z-\z_{j,-n}^{-}]}{\z_{j,n}^{-}(-\z_{j,-n}^{-})},
\]
where $B_0$ and $B_1$ are unknown coefficients which will be specified below.
To this  end, we recast the procedure leading to identification of $A_0$ and $A_1$ and obtain the desired equalities: $B_1=A_1,
$ and $
B_0=(-1)^{\lfloor\mu_1^{-}/2\rfloor}[\cos\theta_1-q_2q_1^{\mu_1^{-}}\cos\theta_2].$
This completes the proof of Proposition \ref{p5.2}.
\end{proof}

Our next results are related to factorization of the function $\s(z)$ in the \textbf{FTK} case. First, we need to introduce the following  quantities depending on $\theta_1,\theta_2\in[0,2\pi),$   and real positive number $q_3$:
\begin{align*}
q_2^{\pm}&=\begin{cases}
\frac{1\mp q_3}{1\pm q_3},\qquad\text{if } 0<q_3<1,\\
-\frac{1\mp q_3}{1\pm q_3},\quad\text{ if } 1<q_3,
\end{cases}\quad
\theta_2^{\star}=
\begin{cases}
\theta_2+\theta_1,\qquad\qquad\qquad\text{ if }0\leq\theta_2+\theta_1<2\pi,\\
\theta_2-\theta_1-2\pi\lfloor\frac{\theta_2+\theta_1}{2\pi}\rfloor,\quad\text{if }\theta_2+\theta_1\geq 2\pi,
\end{cases}
\\
\theta_{+}^{\star}(\theta_1+\theta_2)&=
\begin{cases}
\theta_{+}(\theta_1+\theta_2),\qquad\quad\text{if }\theta_1+\theta_2\in[0,2\pi),\\
\theta_{+}(\theta_1+\theta_2-2\pi),\quad\text{if }\theta_1+\theta_2\in[2\pi,4\pi)
\end{cases}
\end{align*}
(as for the definition of $\theta_{+}$ see \eqref{5.2*}).
\begin{proposition}\label{p5.3}
Let $0<\omega_1\leq \omega_2$, and $q_3>0$, and let
 $0\leq\theta_1\leq\theta_2<2\pi$.
Then  the following factorizations hold for  functions $\tan(\omega_1z-\theta_1)\pm q_3\tan(\omega_2 z-\theta_2)$.
\begin{description}
    \item[i] If $q_3= 1$ then
    \begin{align*}
    &\tan(\omega_1z-\theta_1)- \tan(\omega_2 z-\theta_2)=-\frac{(-1)^{\lfloor\frac{\theta_2-\theta_1}{\pi}\rfloor-\lfloor\frac{2\theta_2+\pi}{2\pi}\rfloor-\lfloor\frac{2\theta_1+\pi}{2\pi}\rfloor}(\omega_2-\omega_1)\sin\theta_{+}^{\star}(\theta_2-\theta_1)}
    {\theta_{+}(\theta_2-\theta_1)\cos\theta_{-}(\theta_1)\cos\theta_{-}(\theta_2)}\\
    &
    \times
    \Big(z-\frac{\theta_{+}(\theta_2-\theta_1)}{\omega_2-\omega_1}\Big)\prod\limits_{n=1}^{+\infty}
    \frac{\bigg(\frac{\pi n+\theta_{+}(\theta_2-\theta_1)}{\omega_2-\omega_1}-z\bigg)
    \bigg(\frac{\pi n-\theta_{+}(\theta_2-\theta_1)}{\omega_2-\omega_1}+z\bigg)}
    {\bigg(\frac{\pi n+\theta_{+}(\theta_2-\theta_1)}{\omega_2-\omega_1}\bigg)
    \bigg(\frac{\pi n-\theta_{+}(\theta_2-\theta_1)}{\omega_2-\omega_1}\bigg)}\\
    &
    \times
    \frac{\bigg(\frac{\pi (2n-1)+2\theta_{-}(\theta_1)}{2\omega_1}\bigg)\bigg(\frac{\pi (2n-1)-2\theta_{-}(\theta_1)}{2\omega_1}\bigg)\bigg(\frac{\pi (2n-1)+2\theta_{-}(\theta_2)}{2\omega_2}\bigg)\bigg(\frac{\pi (2n-1)-2\theta_{-}(\theta_2)}{2\omega_2}\bigg)}
    {\bigg(\frac{\pi (2n-1)+2\theta_{-}(\theta_1)}{2\omega_1}-z\bigg)\bigg(\frac{\pi (2n-1)-2\theta_{-}(\theta_1)}{2\omega_1}+z\bigg)\bigg(\frac{\pi (2n-1)+2\theta_{-}(\theta_2)}{2\omega_2}-z\bigg)\bigg(\frac{\pi (2n-1)-2\theta_{-}(\theta_2)}{2\omega_2}+z\bigg)},
        \end{align*}
            \begin{align*}
    &\tan(\omega_1z-\theta_1)+\tan(\omega_2 z-\theta_2)=\frac{(-1)^{\lfloor\frac{\theta_2-\theta_1}{\pi}\rfloor-\lfloor\frac{2\theta_2+\pi}{2\pi}\rfloor-\lfloor\frac{2\theta_1+\pi}{2\pi}\rfloor}(\omega_2+\omega_1)\sin\theta^{\star}_{+}(\theta_2+\theta_1)}
    {\theta_{+}^{\star}(\theta_2+\theta_1)\cos\theta_{-}(\theta_1)\cos\theta_{-}(\theta_2)}\\
    &
    \times
    \Big(z+\frac{\theta^{\star}_{+}(\theta_2+\theta_1)}{\omega_2+\omega_1}\Big)\prod\limits_{n=1}^{+\infty}
    \frac{\bigg(\frac{\pi n+\theta^{\star}_{+}(\theta_2+\theta_1)}{\omega_2+\omega_1}-z\bigg)
    \bigg(\frac{\pi n-\theta^{\star}_{+}(\theta_2+\theta_1)}{\omega_2+\omega_1}+z\bigg)}
    {\bigg(\frac{\pi n+\theta^{\star}_{+}(\theta_2+\theta_1)}{\omega_2+\omega_1}\bigg)
    \bigg(\frac{\pi n-\theta^{\star}_{+}(\theta_2+\theta_1)}{\omega_2+\omega_1}\bigg)}\\
    &
    \times
    \frac{\bigg(\frac{\pi (2n-1)+2\theta_{-}(\theta_1)}{2\omega_1}\bigg)\bigg(\frac{\pi (2n-1)-2\theta_{-}(\theta_1)}{2\omega_1}\bigg)
    \bigg(\frac{\pi (2n-1)+2\theta_{-}(\theta_2)}{2\omega_2}\bigg)\bigg(\frac{\pi (2n-1)-2\theta_{-}(\theta_2)}{2\omega_2}\bigg)}
    {\bigg(\frac{\pi (2n-1)+2\theta_{-}(\theta_1)}{2\omega_1}-z\bigg)\bigg(\frac{\pi (2n-1)-2\theta_{-}(\theta_1)}{2\omega_1}+z\bigg)\bigg(\frac{\pi (2n-1)+2\theta_{-}(\theta_2)}{2\omega_2}-z\bigg)\bigg(\frac{\pi (2n-1)-2\theta_{-}(\theta_2)}{2\omega_2}+z\bigg)}.
        \end{align*}
        \item[ii] If $q_3\neq 1$ and, in addition, $\omega_2=\frac{1+q_1}{q_1-1}\omega_1$, where $q_1$ meets the requirements of Corollary \ref{c5.1}, then in the case of $q_3\in(0,1)$ the factorization holds
        \begin{align*}
        &\tan(\omega_1z-\theta_1)\pm q_3\tan(\omega_2 z-\theta_2)=-\frac{(1\pm q_3)\s^{-}((\omega_2-\omega_1)z;\theta_2-\theta_1,\theta_{2}^{\star},\frac{\omega_1+\omega_2}{\omega_2-\omega_1},q_{2}^{\pm})}
        {2(-1)^{\lfloor\frac{2\theta_2+\pi}{2\pi}\rfloor+\lfloor\frac{2\theta_1+\pi}{2\pi}\rfloor}\cos\theta_{-}(\theta_1)\cos\theta_{-}(\theta_2)}\\
        &
        \times
        \prod\limits_{n=1}^{+\infty}
        \frac{\bigg(\frac{\pi (2n-1)+2\theta_{-}(\theta_1)}{2\omega_1}\bigg)\bigg(\frac{\pi (2n-1)-2\theta_{-}(\theta_1)}{2\omega_1}\bigg)
    \bigg(\frac{\pi (2n-1)+2\theta_{-}(\theta_2)}{2\omega_2}\bigg)\bigg(\frac{\pi (2n-1)-2\theta_{-}(\theta_2)}{2\omega_2}\bigg)}
    {\bigg(\frac{\pi (2n-1)+2\theta_{-}(\theta_1)}{2\omega_1}-z\bigg)\bigg(\frac{\pi (2n-1)-2\theta_{-}(\theta_1)}{2\omega_1}+z\bigg)\bigg(\frac{\pi (2n-1)+2\theta_{-}(\theta_2)}{2\omega_2}-z\bigg)\bigg(\frac{\pi (2n-1)-2\theta_{-}(\theta_2)}{2\omega_2}+z\bigg)},
        \end{align*}
        while in the case of $q_3>1$, there is the decomposition
        \begin{align*}
        &\tan(\omega_1z-\theta_1)\pm q_3\tan(\omega_2 z-\theta_2)=-\frac{(1\pm q_3)\s^{+}((\omega_2-\omega_1)z;\theta_2-\theta_1,\theta_{2}^{\star},\frac{\omega_1+\omega_2}{\omega_2-\omega_1},q_{2}^{\pm})}
        {2(-1)^{\lfloor\frac{2\theta_2+\pi}{2\pi}\rfloor+\lfloor\frac{2\theta_1+\pi}{2\pi}\rfloor}\cos\theta_{-}(\theta_1)\cos\theta_{-}(\theta_2)}\\
        &
        \times
        \prod\limits_{n=1}^{+\infty}
        \frac{\bigg(\frac{\pi (2n-1)+2\theta_{-}(\theta_1)}{2\omega_1}\bigg)\bigg(\frac{\pi (2n-1)-2\theta_{-}(\theta_1)}{2\omega_1}\bigg)
    \bigg(\frac{\pi (2n-1)+2\theta_{-}(\theta_2)}{2\omega_2}\bigg)\bigg(\frac{\pi (2n-1)-2\theta_{-}(\theta_2)}{2\omega_2}\bigg)}
    {\bigg(\frac{\pi (2n-1)+2\theta_{-}(\theta_1)}{2\omega_1}-z\bigg)\bigg(\frac{\pi (2n-1)-2\theta_{-}(\theta_1)}{2\omega_1}+z\bigg)\bigg(\frac{\pi (2n-1)+2\theta_{-}(\theta_2)}{2\omega_2}-z\bigg)\bigg(\frac{\pi (2n-1)-2\theta_{-}(\theta_2)}{2\omega_2}+z\bigg)}.
        \end{align*}
            \item[iii] If, in addition, assumptions of Corollary \ref{c5.1} hold, and $\bar{\theta}_1$, $\bar{\theta}_2\in[0,2\pi)$, $q_{1}^{\star}=\frac{p^{\star}}{q^{\star}}$, where $p^{\star},$ $q^{\star}\in\N$, $p^{\star}>2q^{\star}>1,$ and $q^{\star}/p^{\star}$ being irreducible fraction and $q_{2}^{\star}>0,$ $q_2^{\star}\neq 1.$ Then
            \begin{align*}
            \frac{\sin(z-\theta_1)-q_2\sin(q_1 z-\theta_2)}{\sin(z-\theta^{\star}_1)\mp q^{\star}_2\sin(q^{\star}_1 z-\theta^{\star}_2)}
            &=\frac{\s^{-}(z;\theta_1,\theta_2,q_1,q_2)}{\s^{\mp}(z;\theta^{\star}_1,\theta^{\star}_2,q^{\star}_1,q^{\star}_2)},\\
            \frac{\sin(z-\theta_1)+q_2\sin(q_1 z-\theta_2)}{\sin(z-\theta^{\star}_1)\mp q^{\star}_2\sin(q^{\star}_1 z-\theta^{\star}_2)}
            &=\frac{\s^{+}(z;\theta_1,\theta_2,q_1,q_2)}{\s^{\mp}(z;\theta^{\star}_1,\theta^{\star}_2,q^{\star}_1,q^{\star}_2)},
            \end{align*}
            where factorizations of $\s^{\mp}(z;\theta^{\star}_1,\theta^{\star}_2,q^{\star}_1,q^{\star}_2)$ and $\s^{+}(z;\theta_1,\theta_2,q_1,q_2)$ are given in Proposition \ref{p5.2}.
\end{description}
\end{proposition}
\begin{proof}
First of all, we remark that statement (i) in the case of $\omega_1=0$ and $\theta_1\in[0,\pi/2]$ is proved in \cite[\S5]{SF} and \cite[\S3]{BV4}. It is apparent that the statement (iii) is a simple consequence of Proposition \ref{p5.2} and the definition of functions $\s^{\pm}$. Concerning the proof of statements in (i) and (ii), it follows immediately from Propositions \ref{p5.1} and \ref{p5.2} and the easily verified identities below
\begin{align*}
&\tan(\omega_1z-\theta_1)\mp\tan(\omega_2z-\theta_2)=\mp\frac{\sin[(\omega_2\mp\omega_1)z\mp(\theta_2-\theta_1)]}
{\cos(\omega_1z-\theta_1)\cos(\omega_2z-\theta_2)},\\
&\tan(\omega_1z-\theta_1)\pm q_3\tan(\omega_2z-\theta_2)=-\frac{(1\pm q_3)}{2\cos(\omega_1z-\theta_1)\cos(\omega_2z-\theta_2)}\\
&
\times
\begin{cases}
\Big[\sin(y-\theta_2+\theta_1)-q_{2}^{+}\sin\Big(\frac{\omega_1+\omega_2}{\omega_2-\omega_1}y-\theta_1-\theta_2\Big)\Big]\Big|_{y=(\omega_2-\omega_1)z},\quad\text{if }q_3\in(0,1),\\
\Big[\sin(y-\theta_2+\theta_1)+q_{2}^{-}\sin\Big(\frac{\omega_1+\omega_2}{\omega_2-\omega_1}y-\theta_1-\theta_2\Big)\Big]\Big|_{y=(\omega_2-\omega_1)z},\quad\text{if }q_3>1.
\end{cases}
\end{align*}
That completes the proof of this claim.
\end{proof}
To test the results proposed in Corollary \ref{5.1} and Proposition \ref{p5.2}, the example of $\s^{+}(z;\theta_1,\theta_2,q_1,q_2)$ is examined here below, where, in particular, the zeros of $\s^{+}$ are found in the analytical form.
\begin{example}\label{e5.1}
Consider the function $\s^{+}(z;\theta_1,2\theta_1,2,q_2)=\sin(z-\theta_1)+q_2\sin 2(z-\theta_1)$ with  $\theta_2=2\theta_1$ and $\theta_1\in[\pi/4,\pi/3]$, $q=1$, $p=4$ (i.e. $q_1=2$).
\end{example}
In this example, it is easy to verify that
$\T_q=4\pi$ and $
\sin\theta_1+q_2\sin \theta_2\neq 0,
 $
for  $\theta_1\in[\pi/4,\pi/3]$ and $\theta_2\in[\pi/2,2\pi/3]$.
Performing the simple calculations, we  arrive at the explicit form of the zeros in the segment $[0,4\pi)$ which are listed in
Table \ref{tab:table1}. Besides, all zeros of $\s^{+}(z;\theta_1,2\theta_1,2,q_2)$ are given by
\[
\z_{i,n}^{+}=\z_{i}^{+}+4\pi n,\quad \z_{i,-n}^{+}=\z_{i}^{+}-4\pi n, \quad n\in\N, \, i=0,1,...,\K^{+}-1.
\]
\begin{table}[htbp]
  \begin{center}
    \caption{Quantities of $\K^{+}$ and $\z_i^{+}$  in Example \ref{e5.1}}
    \label{tab:table1}
    \begin{tabular}{c|c|c|c}
      \hline
     $q_2$ &$q_2>1/2$& $q_2=1/2$&  $0<q_2<1/2$   \\
             \hline
    $\K^{+}$ & 8 & 8  & 4 \\
        \hline
      $\z_0^{+}$ & $\theta_1$ &$\theta_1$  & $\theta_1$ \\
            \hline
$\z_1^{+}$ & $\pi+\theta_1-\arccos\frac{1}{2q_2}$ & $\pi+\theta_1$ & $\pi+\theta_1$  \\
\hline
     $\z_2^{+}$  & $\pi+\theta_1$ & $\pi+\theta_1$ & $2\pi+\theta_1$\\
    \hline
        $\z_3^{+}$ & $\pi+\theta_1+\arccos\frac{1}{2q_2}$ & $\pi+\theta_1$ & $3\pi+\theta_1$\\
        \hline
        $\z_4^{+}$ & $2\pi+\theta_1$ & $2\pi+\theta_1$ & --\\
        \hline
         $\z_5^{+}$ & $3\pi+\theta_1-\arccos\frac{1}{2q_2}$ & $3\pi+\theta_1$ & -- \\
        \hline
         $\z_6^{+}$ & $3\pi+\theta_1$ & $3\pi+\theta_1$ & -- \\
        \hline
         $\z_7^{+}$ & $3\pi+\theta_1+\arccos\frac{1}{2q_2}$ & $3\pi+\theta_1$ & -- \\
        \hline
            \end{tabular}
  \end{center}
\end{table}
Examination of  these zeros arrives at the conclusions:

\noindent$\bullet$ if $q_2>1$, then $\z_{i}^{+}\in\ss^{\star}(i,2\theta_1,2)=\Big[\frac{i\pi}{2}+\theta_1-\frac{\pi}{4},\frac{i\pi}{2}+\theta_1+\frac{\pi}{4}\Big]$, $i\in\{0,1,...,7\},$ and $0<\z_{i}^{+}<\z_{i+1}^{+}$;

\noindent$\bullet$ if $1/2<q_2<1$, then $0<\z_{i}^{+}<\z_{i+1}^{+}$, $i\in\{0,1,...,7\},$  and
\begin{align*}
\z_{i}^{+}\in\ss(0,\theta_1,q_2)&=[\theta_1-\arcsin q_2,\theta_1+\arcsin q_2]\\
&\cup[\pi+\theta_1-\arcsin q_2,\pi+\theta_1+\arcsin q_2],\quad i\in\{0,1,2,3\},\\
\z_{i}^{+}\in\ss(1,\theta_1,q_2)&=[2\pi+\theta_1-\arcsin q_2,2\pi+\theta_1+\arcsin q_2]\\
&\cup[3\pi+\theta_1-\arcsin q_2,3\pi+\theta_1+\arcsin q_2],\quad i\in\{4,5,6,7\};
\end{align*}

\noindent$\bullet$ if $q_2=1/2$, then there are 3-multiple roots of $\s^{+}$ and
$
\z_{0}^{+},\z_{1}^{+}=\z_{2}^{+}=\z_{3}^{+} \in\ss(0,\theta_1,1/2)$ and $
\z_{4}^{+},\z_{5}^{+}=\z_{6}^{+}=\z_{7}^{+} \in\ss(1,\theta_1,1/2);
$

\noindent$\bullet$ if $0<q_2<1/2$, then all zeroes are simple and
$
\z_{0}^{+},\z_{1}^{+} \in\ss(0,\theta_1,q_2)$ and $
\z_{2}^{+},\z_{3}^{+} \in\ss(1,\theta_1,q_2).
$

Summarizing, we conclude that the outcomes in  Example \ref{e5.1} are fitting to Corollary \ref{c5.1}. Finally, statement (ii) in Proposition \ref{p5.2} provides the following factorization of the function $\s^{+}(z;\theta_1,2\theta_1,2,q_2)$:
\begin{align*}
\sin(z-\theta_1)+q_2\sin 2(z-\theta_1)&=-[\sin\theta_1+q_2\sin2\theta_1]\\
&
\times
\begin{cases}
\prod\limits_{i=0}^{7}\frac{\z_{i}^{+}-z}{\z_i^{+}}
                \prod\limits_{n=1}^{+\infty}\prod\limits_{i=0}^{7}\frac{[\z_{i}^{+}+4\pi n-z][4\pi n-\z_{i}^{+}+z]}
                {[\z_{i}^{+}+4\pi n][4\pi n-\z_{i}^{+}]},\, \text{ if }q_2> 1/2,\\
                                \prod\limits_{i=0}^{3}\frac{\z_{i}^{+}-z}{\z_i^{+}}
                \prod\limits_{n=1}^{+\infty}\prod\limits_{i=0}^{3}\frac{[\z_{i}^{+}+4\pi n-z][4\pi n-\z_{i}^{+}+z]}
                {[\z_{i}^{+}+4\pi n][4\pi n-\z_{i}^{+}]},\, \text{ if }0<q_2< 1/2,
\end{cases}
\end{align*}
\begin{align*}
&\sin(z-\theta_1)+\frac{1}{2}\sin 2(z-\theta_1)=-\Big[\sin\theta_1+\frac{1}{2}\sin2\theta_1\Big]
\frac{(\z_{0}^{+}-z)}{\z_0^{+}}
        \Big(\frac{\z_{1}^{+}-z}{\z_1^{+}}\Big)^{3}
        \frac{(\z_{4}^{+}-z)}{\z_4^{+}}\Big(\frac{\z_{5}^{+}-z}{\z_5^{+}}\Big)^{3}\\
        &
        \times
            \prod\limits_{n=1}^{+\infty}\frac{[\z_{0}^{+}+4\pi n-z][4\pi n-\z_{0}^{+}+z]}
                {[\z_{0}^{+}+4\pi n][4\pi n-\z_{0}^{+}]}
                \Big(\frac{[\z_{1}^{+}+4\pi n-z][4\pi n-\z_{1}^{+}+z]}
                {[\z_{1}^{+}+4\pi n][4\pi n-\z_{1}^{+}]}\Big)^{3} \\
                &\times
                                \frac{[\z_{4}^{+}+4\pi n-z][4\pi n-\z_{4}^{+}+z]}
                {[\z_{4}^{+}+4\pi n][4\pi n-\z_{4}^{+}]}\Big(\frac{[\z_{5}^{+}+4\pi n-z][4\pi n-\z_{5}^{+}+z]}
                {[\z_{5}^{+}+4\pi n][4\pi n-\z_{5}^{+}]}\Big)^{3}.
\end{align*}


\subsection{Solvability of \eqref{5.1}}
\label{s5.2}
\noindent
In this section, assuming \textbf{H8-H10} and appealing to Theorems \ref{t3.1}-\ref{t3.2}, \ref{t4.1}, we will construct a solution of equation \eqref{5.1}. Indeed, assumptions \textbf{H8-H10} provide the validity of Lemma \ref{l5.1}, which in turn means that the function $\s(z)$ has the form similar to $\Omega(z)$. To conclude this fact, it is enough to compare representation \eqref{5.1*} with \eqref{i.2}. Thus, we can apply Theorem \ref{t3.1} or \ref{t4.1} in the case of $\F\equiv 0$ and Theorem \ref{t3.2} or \ref{t4.1} for $\F\neq 0$ to equation \eqref{5.1}.

First, we make additional assumptions.

  \noindent\textbf{H11:} If $\s(z)=\s_1(z),$ then  we require
        \begin{equation}\label{5.6}
            Re\, \Big(\frac{z-\z_{-n}}{\beta}\Big)\neq -m,\quad\text{if}  \quad\beta>0,\quad \text{and}\quad
Re\, \Big(\frac{z-\z_{n}}{|\beta|}\Big)\neq -m,\quad\text{if}  \quad\beta<0,
        \end{equation}
for $m\in\N\cup\{0\},$ and $n\in\{1,2,..., N_1\}$ if $\s_1$ has the finite number of zeros, and $n\in\N$ otherwise.
In addition, if $\mu_1\neq 0$, we assume that the inequality holds:
$
Re(z/\beta)\neq-m.
$

In the case of $\s(z)=1/\s_2(z)$, we require
\begin{equation}\label{5.7}
            Re\, \Big(\frac{z-\y_{n}}{\beta}\Big)\neq 1+m,\quad\text{if}  \quad\beta>0,\quad\text{and}\quad
Re\, \Big(\frac{z-\y_{-n}}{|\beta|}\Big)\neq-1 -m,\quad\text{if}  \quad\beta<0,
        \end{equation}
for  $m\in\N\cup\{0\},$ and $n\in\{1,2,..., N_3\}$  if $\s_2$ has the finite number of zeros, and $n\in\N$  otherwise.

If $\s(z)=\s_1(z)/\s_2(z)$, then inequalities \eqref{5.6} and \eqref{5.7} hold simultaneously.

At this point, for simplicity consideration, we put
\[
A_0^{I}\exp\{g_1(z)+P_1(z)\}:=A_0^{I}\exp\{B_1^{I}(z)\},\quad A_0^{II}\exp\{-g_2(z)-P_2(z)\}:=A_0^{II}\exp\{B_1^{II}(z)\},
\]
and then introduce the functions:
\begin{align*}
\mathcal{R}^{I}(n)&=sgn\, (\beta)\Big(\frac{\z_{-n}}{|\beta|}\Big[\ln\Big(-\frac{\z_{-n}}{|\beta|}-1\Big)\Big]
+\frac{\z_{n}}{|\beta|}\Big[\ln\Big(\frac{\z_{n}}{|\beta|}-1\Big)\Big]
\Big),\\
\mathcal{R}^{II}(n)&=-sgn\, (\beta)\Big(\frac{\y_{-n}}{|\beta|}\Big[\ln\Big(-\frac{\y_{-n}}{|\beta|}-1\Big)\Big]
+\frac{\y_{n}}{|\beta|}\Big[\ln\Big(\frac{\y_{n}}{|\beta|}-1\Big)\Big]
\Big),
\end{align*}
 and
\begin{align*}
\mathbb{L}^{I}(z)&=
\begin{cases}
\frac{\prod\limits_{n=1}^{N_2}\Gamma\Big(\frac{z-\z_{-n}}{\beta}\Big)\Big(\Gamma(z/\beta)\Big)^{\mu_1}}
{\prod\limits_{n=1}^{N_1}\Gamma\Big(\frac{\z_{n}-z}{\beta}+1\Big)}\exp\{i\pi z(N_1+1)\},\quad\text{if }\s_1 \text{ has a finite number of zeros},\\
\Big(\Gamma(z/\beta)\Big)^{\mu_1}\prod\limits_{n=1}^{+\infty}\frac{\Gamma\Big(\frac{z-\z_{-n}}{\beta}\Big)}
{\Gamma\Big(\frac{\z_{n}-z}{\beta}+1\Big)}\Big[\frac{\z_n(-\z_{-n})}{\beta^{2}}\Big]^{\frac{\beta-2z}{2\beta}}\exp\{\mathcal{R}^{I}\},\quad\text{if }\s_1 \text{ has an infinite number }\\
\qquad\qquad\qquad\qquad\qquad\qquad\qquad\qquad\qquad\qquad\qquad\qquad\text{of zeros and }\beta>0,\\
\Big(\Gamma(z/\beta)\Big)^{\mu_1}\prod\limits_{n=1}^{+\infty}\frac{\Gamma\Big(\frac{\z_{n}-z}{|\beta|}\Big)}
{\Gamma\Big(\frac{z-\z_{-n}}{|\beta|}+1\Big)}\Big[\frac{\z_n(-\z_{-n})}{\beta^{2}}\Big]^{\frac{\beta-2z}{2\beta}}\exp\{\mathcal{R}^{I}\},\quad\text{if }\s_1 \text{ has an infinite number}\\
\qquad\qquad\qquad\qquad\qquad\qquad\qquad\qquad\qquad\qquad\qquad\quad\quad\text{ of zeros and }\beta<0,
\end{cases}
\end{align*}
\begin{align*}\
\mathbb{L}^{II}(z)&=
\begin{cases}
\Big(\Gamma\big(\frac{z}{\beta}\big)\Big)^{-\mu_2}\frac{\prod\limits_{n=1}^{N_3}\Gamma\Big(\frac{\y_{n}-z}{\beta}+1\Big)}
{\prod\limits_{n=1}^{N_4}\Gamma\Big(\frac{z-\y_{-n}}{\beta}\Big)}\exp\{i\pi z(N_3+1)\},\quad\text{if }\s_2 \text{ has a finite number of zeros},\\
\Big(\Gamma(z/\beta)\Big)^{-\mu_2}\prod\limits_{n=1}^{+\infty}\frac{\Gamma\Big(\frac{\y_{n}-z}{\beta}+1\Big)}
{\Gamma\Big(\frac{z-\y_{-n}}{\beta}\Big)}\Big[\frac{\y_n(-\y_{-n})}{\beta^{2}}\Big]^{\frac{2z-\beta}{2\beta}}\exp\{\mathcal{R}^{II}\},\quad\text{if }\s_2 \text{ has an infinite number }\\
\qquad\qquad\qquad\qquad\qquad\qquad\qquad\qquad\qquad\qquad\qquad\qquad\quad\text{ of zeros and }\beta>0,\\
\Big(\Gamma(z/\beta)\Big)^{-\mu_2}\prod\limits_{n=1}^{+\infty}\frac{\Gamma\Big(1+\frac{z-\y_{-n}}{|\beta|}\Big)}
{\Gamma\Big(\frac{\y_{n}-z}{|\beta|}\Big)}\Big[\frac{\y_n(-\y_{-n})}{\beta^{2}}\Big]^{\frac{2z-\beta}{2\beta}}\exp\{\mathcal{R}^{II}\},\quad\text{if }\s_1 \text{ has an infinite number }\\
\qquad\qquad\qquad\qquad\qquad\qquad\qquad\qquad\qquad\qquad\qquad\qquad\qquad\text{ of zeros and }\beta<0.
\end{cases}\end{align*}
Finally, denoting
\[
\d_0^{\star}=
\begin{cases}
A_{0}^{I}\beta^{-\mu_1}\quad\text{in the \bf{FFK} case},\\
\, \\
\frac{\beta^{\mu_2}}{A_0^{II}}\quad\text{in the \bf{FSK} case},\\
\,\\
\frac{\beta^{\mu_2-\mu_1}A_{0}^{I}}{A_0^{II}}\quad\text{in the \bf{FTK} case},
\end{cases}\quad
B^{\star}=
\begin{cases}
B_1^{I}\quad\quad\text{in the \bf{FFK} case},\\
\, \\
B_1^{II}\quad\quad\text{in the \bf{FSK} case},\\
\, \\
B_1^{I}-B_2^{II}\quad\text{in the \bf{FTK} case},
\end{cases}\]
\[
\mathbb{L}^{\star}(z)=
\begin{cases}
\mathbb{L}^{I}\quad\quad\text{in the \bf{FFK} case},\\
\, \\
\mathbb{L}^{II}\quad\quad\text{in the \bf{FSK} case},\\
\, \\
\frac{\mathbb{L}^{I}}{\mathbb{L}^{II}}\quad\text{in the \bf{FTK} case},
\end{cases}
\]
and appealing to Theorem \ref{t3.1}, we end up with the assertion.
\begin{theorem}\label{t5.1}
Let assumptions \textbf{H1} and \textbf{H8-H10} hold, $\F\equiv 0$. Then a general solution $Y_h(z,\sigma;\P)$ of homogenous equation \eqref{5.1} is given by
\begin{equation}\label{5.8*}
Y_{h}(z,\sigma;\P)=\exp\Big\{\frac{B^{\star}}{2\beta}z^{2}-\frac{B^{\star}}{2}z\Big\}(\d_0^{\star})^{\frac{2z-\beta}{2\beta}}
\P(z/\beta)(a_1\sigma+a_2\sigma)^{\frac{\beta-2z}{2\beta}}\mathbb{L}^{\star}(z),
\end{equation}
where $\P(z)$ is a function describing in Theorem \ref{t3.1}.

Moreover, statements \textbf{(s.1), (s.2)} of  Theorem \ref{t3.1} with $\mathbb{L}(z):=\mathbb{L}^{\star}(z)$ and
$\mathbb{L}_1(z):=\mathbb{L}^{\star}(z)$ hold.
\end{theorem}

Coming to inhomogeneous equation \eqref{5.1}, we state additional hypotheses.

\noindent \textbf{H12:} For $d_0\in[0,1]$ and $m\in\N\cup\{0\},$ we assume
\begin{equation*}\label{5.8}
Re\Big(\frac{z-\z_n}{\beta}\Big)\neq d_0+m,\quad n\in\{1,2,...,N_1\},
\end{equation*}
in the case of a finite number of zeros at  $\s_1$. Otherwise, we require that  these inequalities  hold with $n\in\N$ if $\beta>0$, while in the case of $\beta<0$ the inequalities are fulfilled
\begin{equation*}\label{5.8**}
Re\Big(\frac{z-\z_{-n}}{|\beta|}\Big)\neq -d_0-m,\quad n\in\N.
\end{equation*}
In the case of a finite number of zeros at $\s_2$, we require that
\begin{equation*}\label{5.9}
Re(z/\beta)\neq -1-m+d_0\quad \text{if}\quad \mu_2\neq 0;\quad\text{and }
Re\Big(\frac{z-\y_{-n}}{\beta}\Big)\neq -m-1+d_0,\quad n\in\{1,2,...,N_4\}.
\end{equation*}
Otherwise, we require that  these relations hold with $n\in\N$ if $\beta>0$, while in the case of $\beta<0$ the inequalities are fulfilled
\begin{equation*}\label{5.9***}
Re(z/\beta)\neq -1-m+d_0\quad\text{if}\quad\mu_2\neq 0;\quad\text{and }
Re\Big(\frac{z-\y_{n}}{|\beta|}\Big)\neq m+1-d_0,\quad n\in\N.
\end{equation*}
In the case of the function $\s_1/\s_2$, we assume that all inequalities  hold simultaneously.

To make the following assumption, we set $\P_1:=\P(z/\beta)$ for each chosen periodic function and
$
Y_h(z,\sigma)=Y_{h}(z,\sigma;\P_1).
$

\noindent \textbf{H13} We assume that $\F(z,\sigma)$ is an analytic function for $z\in\C$ and each $\sigma$ satisfying \textbf{H1}. Moreover, for some functions $\P_1(z/\beta)$ and $\mathcal{K}_{1}(\xi)$, the equality holds
\begin{equation}\label{5.9**}
\underset{|Im\xi|\to+\infty}{\lim}\Big|\frac{\mathcal{K}(\xi)\F(z+\beta\xi,\sigma)}{Y_h(z+\beta\xi+\beta,\sigma)}\Big|=0.
\end{equation}
for all $z$ satisfying \textbf{H.11} and \textbf{H.12} and each $Re\xi\in[-1,0]$.

Now, taking into account Theorems \ref{t3.2}, \ref{t4.1} and \ref{t5.1} we conclude.
\begin{theorem}\label{t5.2}
Let assumptions of Theorem \ref{t5.1} and hypotheses \textbf{H5, H12, H13} hold. Then a general solution of \eqref{5.1} has a form
$
Y_g(z,\sigma)=Y_h(z,\sigma;\P)+Y_{ih}(z,\sigma),
$
where $Y_{h}(z,\sigma;\P)$ is defined in Theorem \ref{t5.1} and $Y_{ih}(z,\sigma)$ being a particular solution is given by
\begin{equation}\label{5.9*}
Y_{ih}(z,\sigma)=\frac{Y_h(z,\sigma)}{\mathcal{K}_1(0)}\int\limits_{\ell_{d_0}}\frac{\F(z+\beta\xi,\sigma)\mathcal{K}(\xi)}{(a_1\sigma+a_2\sigma^{\nu})Y_{h}(z+\beta\xi+\beta,\sigma)}d\xi.
\end{equation}
\end{theorem}

Finally, to demonstrate results of Theorems \ref{t5.1}-\ref{t5.2}, we conclude this section with examples of explicit solutions to \eqref{5.1} with the coefficient $\s$ presented in Propositions \ref{p5.1}-\ref{p5.3}. In light of formulas \eqref{5.8*} and \eqref{5.9*}, we give below the explicit forms to $\mathbb{L}^{\star}(z),$ $B^{\star}$, $\d_0^{\star}$ and the kernel $\mathcal{K}(\xi)$. As for $\mathcal{K}(\xi)$, we describe several kinds of this function which correspond to the various types of the functions $\P(z/\beta)$ and $\F$. For simplicity consideration, we analyze \eqref{5.1} in the case of $\beta=1$.

We begin with consideration of the homogenous case of \eqref{5.1}. By virtue of Propositions \ref{p5.1}-\ref{p5.3}, we deduce that $B^{\star}\equiv 0$, which means the term $\exp\{\frac{B^{\star}z^{2}}{2\beta}-\frac{B^{\star}z}{2}\}$ equals to $1$ in \eqref{5.8*}.  To this end, for
$q_1^{\star}=\frac{p^{\star}}{2q^{\star}}>1,$ $q_2^{\star}>0$, with $p^{\star}$ and $q^{\star}$ having the properties of the numbers $p,q$,
 we first denote the zeros of the function $\s^{\pm}(z;\theta_1,\theta_2,q_1^{\star},q_2^{\star})$ in the segment $[0,4\pi q^{\star})$ by
\[
\z_{i}^{\star,+}:=\z_{i}^{\star,+}(\theta_1,\theta_2,q_1^{\star},q_2^{\star}),\quad i\in\{1,2,...,\K^{\star,+}\},\quad\text{and}\quad
\z_{i}^{\star,-}:=\z_{i}^{\star,-}(\theta_1,\theta_2,q_1^{\star},q_2^{\star}),\quad i\in\{1,2,...,\K^{\star,-}\},
\]
where $\K^{\star,+}$ and $\K^{\star,-}$ are the corresponding number of the  zeros in $[0,4\pi q^{\star})$.
 Then, collecting Theorem \ref{t5.1} with Propositions \ref{p5.1}-\ref{p5.3}, we end up with the outputs listed in Tables \ref{tab:table3}-\ref{tab:table2}.

\begin{longtable}{l|l}
  \caption{The quantity of $\d_{0}^{\star}$ for $\s(z)$ listed in Table \ref{tab:table2}}
    \label{tab:table3}\\
    \hline
    $\s(z)$ &$\d_{0}^{\star}$  \\
                \endfirsthead
                \hline
                            $\s(z)$ &$\d_{0}^{\star}$   \\
                \endhead
             \hline
                                                                        $\sin(q_0z+\theta),\quad
        \theta\in[0,2\pi)$ & $(-1)^{\lfloor\theta/\pi\rfloor}q_0\frac{\sin\theta_{+}}{\theta_{+}}$\\
        \hline
      $\cos(q_0z-\theta),\, \theta\in[0,2\pi),\, \theta\neq \frac{\pi}{2}, \frac{3\pi}{2}$ &
                                $(-1)^{\lfloor\frac{2\theta+\pi}{2\pi}\rfloor}\cos\theta_{-}(\theta)$
                \\
            \hline
  $\s^{+}(z;0,0,q_1,q_2)$ & $(1+q_2q_1)\prod\limits_{i=1}^{\K^{+}-1}(\z_{i}^{+})^{-2}$ \\
    \hline
       $\s^{-}(z;0,0,q_1,q_2),\, q_2\neq\frac{1}{q_1}$ &
            $(1-q_2q_1)\prod\limits_{i=1}^{\K^{-}-1}(\z_{i}^{-})^{-2}$
            \\
    \hline
       $\s^{-}(z;0,0,q_1,q_2),\, q_2=\frac{1}{q_1}$ &
            $-(1-q_2q_1^{3})\prod\limits_{i=4}^{\K^{-}-1}(\z_{i}^{-})^{-2}$  \\
        \hline
                  $\begin{array}{l}\s^{-}(z;\theta_1,\theta_2,q_1,q_2),\, \text{if } \theta_1,\,\theta_2\in(0,\pi),\\
            \sin\theta_1-q_2\sin\theta_2\neq 0\end{array}$ &
            $-(\sin\theta_1-q_2\sin\theta_2)\prod\limits_{i=0}^{\K^{-}-1}(\z_{i}^{-})^{-1}$  \\
        \hline
                  $\begin{array}{l}\s^{-}(z;\theta_1,\theta_2,q_1,q_2),\, \text{if } \theta_1,\,\theta_2\in(0,\pi),\\
            \sin\theta_1-q_2\sin\theta_2= 0\end{array}$ &
            $(-1)^{\lfloor\mu_1^{-}/2\rfloor}[\cos\theta_1-q_2 q_1^{\mu_1^{-}}\cos\theta_2]\prod\limits_{i=\mu_1^{-1}+1}^{\K^{-}-1}(\z_{i}^{-})^{-1}$
                \\
        \hline
                  $\begin{array}{l}\s^{+}(z;\theta_1,\theta_2,q_1,q_2),\ \text{ if } \theta_1,\,\theta_2\in(0,\pi)\\
            \sin\theta_1+q_2\sin\theta_2\neq 0 \end{array}$ &
            $-[\sin\theta_1+q_2\sin\theta_2]\prod\limits_{i=0}^{\K^{+}-1}(\z_{i}^{+})^{-1}$  \\
        \hline
                  $\begin{array}{l}\s^{+}(z;\theta_1,\theta_2,q_1,q_2),\, \text{if } \theta_1,\,\theta_2\in(0,\pi),\\
            \sin\theta_1+q_2\sin\theta_2= 0\end{array}$ &
            $(-1)^{\lfloor\mu_1^{+}/2\rfloor}[\cos\theta_1+q_2 q_{1}^{\mu_1^{+}}\cos\theta_2]
            \prod\limits_{i=\mu_1^{+}+1}^{\K^{+}-1}(\z_{i}^{+})^{-1}$ \\
        \hline
                                      $\begin{array}{l}\frac{\s^{-}(z;\theta_1,\theta_2,q_1,q_2)}{\s^{+}(z;0,0,q_1^{\star},q_2^{\star})},
                            \text{ if } \theta_1,\,\theta_2\in(0,\pi),\\
            \sin\theta_1-q_2\sin\theta_2= 0 \end{array}$  &
                                $\frac{(-1)^{\lfloor\mu_1^{-}/2\rfloor}(\cos\theta_1-q_2 q_{1}^{\mu_1^{-}}\cos\theta_2)}{(1+q_2^{\star}q_1^{\star})}
                \prod\limits_{i=1}^{\K^{\star,+}-1}(\z_{i}^{\star,+})^{2}
                \prod\limits_{i=\mu_1^{-}+1}^{\K^{-}-1}(\z_{i}^{-})^{-1}$
                                    \\
            \hline
                         $\begin{array}{l}\frac{\s^{-}(z;\theta_1,\theta_2,q_1,q_2)}{\s^{-}(z;0,0,q_1^{\star},q_2^{\star})},\,
                            \text{if } \theta_1,\,\theta_2\in(0,\pi),\\
            \sin\theta_1-q_2\sin\theta_2= 0, \,
            q^{\star}_2\neq1/q_1^{\star}\end{array}$  &
                $\frac{(-1)^{\lfloor\mu_1^{-}/2\rfloor}(\cos\theta_1-q_2 q_{1}^{\mu_1^{-}}\cos\theta_2)}{(1-q_2^{\star}q_1^{\star})}
                \prod\limits_{i=1}^{\K^{\star,-}-1}(\z_{i}^{\star,-})^{2}
                \prod\limits_{i=\mu_1^{-}+1}^{\K^{-}-1}(\z_{i}^{-})^{-1}$
                        \\
                    \hline
                                    $\begin{array}{l}\frac{\s^{-}(z;\theta_1,\theta_2,q_1,q_2)}{\s^{-}(z;0,0,q_1^{\star},q_2^{\star})},\,
                            \text{if } \theta_1,\,\theta_2\in(0,\pi),\\
            \sin\theta_1-q_2\sin\theta_2= 0,\,
            q^{\star}_2=1/q_1^{\star}\end{array}$  &
                $\frac{(-1)^{\lfloor\mu_{1}^{-}/2\rfloor-1}(\cos\theta_1-q_2 q_{1}^{\mu_1^{-}}\cos\theta_2)}{(1-q_2^{\star}(q_1^{\star})^{3})}
                \prod\limits_{i=4}^{\K^{\star,-}-1}(\z_{i}^{\star,-})^{2}
                \prod\limits_{i=\mu_1^{-}+1}^{\K^{-}-1}(\z_{i}^{-})^{-1}$
                        \\
                    \hline
                                                            $\begin{array}{l}\frac{\s^{+}(z;\theta_1,\theta_2,q_1,q_2)}{\s^{+}(z;0,0,q_1^{\star},q_2^{\star})},
                            \text{ if } \theta_1,\,\theta_2\in(0,\pi),\\
            \sin\theta_1+q_2\sin\theta_2\neq 0 \end{array}$  &
            $-\frac{(\sin\theta_1+q_2 \sin\theta_2)}{(1+q_2^{\star}q_1^{\star})}
                \prod\limits_{i=1}^{\K^{\star,+}-1}(\z_{i}^{\star,+})^{2}
                \prod\limits_{i=0}^{\K^{+}-1}(\z_{i}^{+})^{-1}$
                        \\
                    \hline
                                        $\begin{array}{l}\frac{\s^{+}(z;\theta_1,\theta_2,q_1,q_2)}{\s^{+}(z;0,0,q_1^{\star},q_2^{\star})},
                            \text{ if } \theta_1,\,\theta_2\in(0,\pi),\\
            \sin\theta_1+q_2\sin\theta_2= 0 \end{array}$  &
                $\frac{(-1)^{\lfloor\mu_{1}^{+}/2\rfloor}(\cos\theta_1+q_2 q_{1}^{\mu_1^{+}}\cos\theta_2)}{(1+q_2^{\star}q_1^{\star})}
                \prod\limits_{i=1}^{\K^{\star,+}-1}(\z_{i}^{\star,+})^{2}
                \prod\limits_{i=\mu_1^{+}+1}^{\K^{+}-1}(\z_{i}^{+})^{-1}$
                        \\
                    \hline
\end{longtable}\begin{longtable}{l|l}
  \caption[ ]{The function $\mathbb{L}^{\star}(z)$ for different kinds of $\s(z)$}
    \label{tab:table2}\\
    \hline
    $\s(z)$ &$\mathbb{L}^{\star}(z)$   \\
                \endfirsthead
                \hline
                    $\s(z)$ &$\mathbb{L}^{\star}(z)$   \\
                \endhead
        \hline
$\s^{+}(z;0,0,q_1,q_2)$ & $\begin{array}{l}
    \Gamma(z)\prod\limits_{i=1}^{\K^{+}-1}\frac{\Gamma(\z_{i}^{+}+z)}{\Gamma(\z_{i}^{+}-z+1)}
    \prod\limits_{n=1}^{+\infty}
    \prod\limits_{i=0}^{\K^{+}-1} \frac{\Gamma(\z_{i,n}^{+}+z)}{\Gamma(\z_{i,n}^{+}-z+1)}[\z_{i,n}^{+}]^{1-2z}
    \end{array}$ \\
    \hline
       $\s^{-}(z;0,0,q_1,q_2),\, q_2\neq\frac{1}{q_1}$ &
            $\begin{array}{l}
    \Gamma(z)\prod\limits_{i=1}^{\K^{-}-1}\frac{\Gamma(\z_{i}^{-}+z)}{\Gamma(\z_{i}^{-}-z+1)}
    \prod\limits_{n=1}^{+\infty}
    \prod\limits_{i=0}^{\K^{-}-1} \frac{\Gamma(\z_{i,n}^{-}+z)}{\Gamma(\z_{i,n}^{-}-z+1)}[\z_{i,n}^{-}]^{1-2z}
    \end{array}$
            \\
    \hline
       $\s^{-}(z;0,0,q_1,q_2),\, q_2=\frac{1}{q_1}$ &
            $\begin{array}{l}
    [\Gamma(z)]^{3}\prod\limits_{i=4}^{\K^{-}-1}\frac{\Gamma(\z_{i}^{-}+z)}{\Gamma(\z_{i}^{-}-z+1)}
    \prod\limits_{n=1}^{+\infty}
    \prod\limits_{i=0}^{\K^{-}-1} \frac{\Gamma(\z_{i,n}^{-}+z)}{\Gamma(\z_{i,n}^{-}-z+1)}[\z_{i,n}^{-}]^{1-2z}
    \end{array}$   \\

                \hline
                  $\begin{array}{l}\s^{-}(z;\theta_1,\theta_2,q_1,q_2), \text{if }\theta_1,\,\theta_2\in(0,\pi)\\
            \text{and }\sin\theta_1-q_2\sin\theta_2\neq 0 \end{array}$ &
            $\begin{array}{l}
    \prod\limits_{i=0}^{\K^{-}-1}\frac{1}{\Gamma(\z_{i}^{-}-z+1)}
    \prod\limits_{n=1}^{+\infty}\prod\limits_{i=0}^{\K^{-}-1} \frac{\Gamma(z-\z_{i,-n}^{-})[(-\z_{i,-n}^{-})\z_{i,n}^{-}]^{\frac{1-2z}{2}}}{\Gamma(\z_{i,n}^{-}-z+1)}\\
    \times\exp\{\z_{i,-n}^{-}[\ln(-\z_{i,-n}^{-})-1]+\z_{i,n}^{-}[\ln(\z_{i,n}^{-})-1]\}
    \end{array}$  \\
        \hline
                  $\begin{array}{l}\s^{-}(z;\theta_1,\theta_2,q_1,q_2), \text{if }\theta_1,\,\theta_2\in(0,\pi)\\
            \text{and }\sin\theta_1-q_2\sin\theta_2= 0\end{array}$ &
            $\begin{array}{l}
\prod\limits_{i=\mu_{1}^{-}+1}^{\K^{-}-1}\frac{
[\Gamma(z)]^{\mu_1^{-}}}{\Gamma(\z_{i}^{-}-z+1)}\prod\limits_{n=1}^{+\infty}
    \prod\limits_{i=0}^{\K^{-}-1} \frac{\Gamma(z-\z_{i,-n}^{-})[(-\z_{i,-n}^{-})\z_{i,n}^{-}]^{\frac{1-2z}{2}}}{\Gamma(\z_{i,n}^{-}-z+1)}\\
    \times\exp\{\z_{i,-n}^{-}[\ln(-\z_{i,-n}^{-})-1]+\z_{i,n}^{-}[\ln(\z_{i,n}^{-})-1]\}
    \end{array}$  \\
        \hline
                  $\begin{array}{l}\s^{+}(z;\theta_1,\theta_2,q_1,q_2), \text{ if } \theta_1,\,\theta_2\in(0,\pi)\\
            \sin\theta_1+q_2\sin\theta_2\neq 0\end{array}$ &
            $\begin{array}{l}
    \prod\limits_{i=0}^{\K^{+}-1}\frac{1}{\Gamma(\z_{i}^{+}-z+1)}
    \prod\limits_{n=1}^{+\infty}
    \prod\limits_{i=0}^{\K^{+}-1} \frac{\Gamma(z-\z_{i,-n}^{+})[(-\z_{i,-n}^{+})\z_{i,n}^{+}]^{\frac{1-2z}{2}}}{\Gamma(\z_{i,n}^{+}-z+1)}\\
    \times\exp\{\z_{i,-n}^{+}[\ln(-\z_{i,-n}^{+})-1]+\z_{i,n}^{+}[\ln(\z_{i,n}^{+})-1]\}
    \end{array}$  \\
        \hline
                  $\begin{array}{l}\s^{+}(z;\theta_1,\theta_2,q_1,q_2)\, \text{ if } \theta_1,\,\theta_2\in(0,\pi)\\
            \text{and}\, \sin\theta_1+q_2\sin\theta_2= 0\end{array}$ &
            $\begin{array}{l}
                \prod\limits_{i=\mu_1^{+}+1}^{\K^{+}-1}\frac{[\Gamma(z)]^{\mu_1^{+}}}{\Gamma(\z_{i}^{+}-z+1)}
    \prod\limits_{n=1}^{+\infty}
    \prod\limits_{i=0}^{\K^{+}-1} \frac{\Gamma(z-\z_{i,-n}^{+})[(-\z_{i,-n}^{+})\z_{i,n}^{+}]^{\frac{1-2z}{2}}}{\Gamma(\z_{i,n}^{+}-z+1)}\\
    \times\exp\{\z_{i,-n}^{+}[\ln(-\z_{i,-n}^{+})-1]+\z_{i,n}^{+}[\ln(\z_{i,n}^{+})-1]\}
    \end{array}$  \\
\hline
                                             $\sin(q_0z+\theta),$ if
        $\theta\in[0,2\pi)$ & $\begin{array}{l}\Gamma\Big(z+\frac{\theta_{+}}{q_0}\Big)\prod\limits_{n=1}^{+\infty}
        \frac{\Gamma\Big(\frac{\pi n-\theta_{+}}{q_0}+z\Big)}
        {\Gamma\Big(\frac{\pi n+\theta_{+}}{q_0}-z+1\Big)}
        \Big(\frac{\pi^{2}n^{2}-\theta_{+}^{2}}{q_0^{2}}\Big)^{\frac{1}{2}-z}\\
        \times\exp\Big\{\frac{\pi n-\theta_{+}}{q_0}\Big[1-\ln\frac{\pi n-\theta_{+}}{q_0}\Big]
        +\frac{\pi n+\theta_{+}}{q_0}\Big[\ln\frac{\pi n+\theta_{+}}{q_0}-1\Big]
        \Big\}
        \end{array}$ \\
        \hline
      $\cos(q_0z-\theta)$, if $\theta\in[0,2\pi),\, \theta\neq\frac{\pi}{2},\frac{3\pi}{2} $ &
                                $\begin{array}{l}\prod\limits_{n=1}^{+\infty}
        \frac{\Gamma\Big(\frac{\pi (2n-1)-2\theta_{-}}{2q_0}+z\Big)}
        {\Gamma\Big(\frac{\pi(2 n-1)+2\theta_{-}}{2q_0}-z+1\Big)}
        \Big(\frac{\pi^{2}(2n-1)^{2}-4\theta_{-}^{2}}{4q_0^{2}}\Big)^{\frac{1}{2}-z}\\
        \times\exp\Big\{\frac{\pi (2n-1)-2\theta_{-}}{2q_0}\Big[1-\ln\frac{\pi(2 n-1)-2\theta_{-}}{2q_0}\Big]\\
        +\frac{\pi(2 n-1)+2\theta_{-}}{2q_0}\Big[\ln\frac{\pi(2 n-1)+2\theta_{-}}{2q_0}-1\Big]
        \Big\}
        \end{array}$
                \\
        \hline
                          $\begin{array}{l}\frac{\s^{-}(z;\theta_1,\theta_2,q_1,q_2)}{\s^{+}(z;0,0,q_1^{\star},q_2^{\star})},\,
                            \text{ if } \theta_1,\,\theta_2\in(0,\pi)\\
            \text{and } \sin\theta_1-q_2\sin\theta_2= 0\end{array}$  &
                $\begin{array}{l}
    \frac{[\Gamma(z)]^{\mu_1^{-}-1}\prod\limits_{i=\mu_1^{-}+1}^{\K^{-}-1}\Gamma(\z_{i}^{-}-z+1)}
    {\prod\limits_{i=1}^{\K^{\star,+}-1}\Gamma(\z_{i}^{\star,+}+z)\prod\limits_{i=\mu_1^{-}+1}^{\K^{-}-1}\Gamma(\z_{i}^{-}-z+1)}
    \prod\limits_{n=1}^{+\infty}
    \frac{\prod\limits_{i=0}^{\K^{-}-1} \Gamma(z-\z_{i,-n}^{-})}
    {\prod\limits_{i=0}^{\K^{-}-1} \Gamma(z-\z_{i,-n}^{-}+1)}
    \\
    \times
    \frac{\prod\limits_{i=0}^{\K^{\star,+}-1}\Gamma(\z_{i,n}^{\star,+}-z+1)}
    {\prod\limits_{i=0}^{\K^{\star,+}-1}\Gamma(\z_{i,n}^{\star,+}+z)}
    \Big[\frac{(-\z_{i,-n}^{-})\z_{i,n}^{-}}{(\z_{i,n}^{\star,+})^{2}}\Big]^{\frac{1-2z}{2}}
    \\
    \times
        \exp\Big\{\sum\limits_{i=0}^{\K^{-}-1}(\z_{i,-n}^{-}[\ln(-\z_{i,-n}^{-})-1]
    +\z_{i,n}^{-}[\ln(\z_{i,n}^{-})-1])\Big \}
    \end{array}$
                                \\
            \hline
                         $\begin{array}{l}\frac{\s^{-}(z;\theta_1,\theta_2,q_1,q_2)}{\s^{-}(z;0,0,q_1^{\star},q_2^{\star})},\,
                            \text{ if }\theta_1,\,\theta_2\in(0,\pi) \\
            \text{and } \sin\theta_1-q_2\sin\theta_2= 0,
            q^{\star}_2\neq1/q_1^{\star}\end{array}$  &
                $\begin{array}{l}
\frac{[\Gamma(z)]^{\mu_1^{-}-1}}{\prod\limits_{i=\mu_1^{-}+1}^{\K^{-}-1}\Gamma(\z_{i}^{-}-z+1)}
    \frac{\prod\limits_{i=1}^{\K^{\star,-}-1}\Gamma(\z_{i}^{\star,-}-z+1)}
    {\prod\limits_{i=1}^{\K^{\star,-}-1}\Gamma(\z_{i}^{\star,-}+z)}\prod\limits_{n=1}^{+\infty}
        \frac{\prod\limits_{i=0}^{\K^{-}-1} \Gamma(z-\z_{i,-n}^{-})}{\prod\limits_{i=0}^{\K^{-}-1}\Gamma(z-\z_{i,-n}^{-}+1)}
            \\
    \times
                \frac{\prod\limits_{i=0}^{\K^{\star,-}-1}
    \Gamma(\z_{i,n}^{\star,-}-z+1)}
    {\prod\limits_{i=0}^{\K^{\star,-}-1}\Gamma(\z_{i,n}^{\star,-}+z)}
    \frac{\prod\limits_{i=0}^{\K^{\star,-}-1}(\z_{i,n}^{\star,-})^{2z-1}}
    {\prod\limits_{i=0}^{\K^{-}-1}
    [\z_{i,n}^{-}(-\z_{i,-n}^{-})]^{\frac{2z-1}{2}}}
    \\
    \times
        \exp\Big\{\sum\limits_{i=0}^{\K^{-}-1}(\z_{i,-n}^{-}[\ln(-\z_{i,-n}^{-})-1]
    +\z_{i,n}^{-}[\ln(\z_{i,n}^{-})-1])\Big \}
    \end{array}$
                        \\
                    \hline

                    $\begin{array}{l}\frac{\s^{-}(z;\theta_1,\theta_2,q_1,q_2)}{\s^{-}(z;0,0,q_1^{\star},q_2^{\star})},
                            \,
                            \text{ if } \theta_1,\,\theta_2\in(0,\pi),\\
            \sin\theta_1-q_2\sin\theta_2= 0,\,
            q^{\star}_2=1/q_1^{\star}\end{array}$  &
                $\begin{array}{l}
\frac{[\Gamma(z)]^{\mu_1^{-}-3}}{\prod\limits_{i=\mu_1^{-}+1}^{\K^{-}-1}\Gamma(\z_{i}^{-}-z+1)}
\frac{
\prod\limits_{i=4}^{\K^{\star,-}-1}\Gamma(\z_{i}^{\star,-}-z+1)}
    {   \prod\limits_{i=4}^{\K^{\star,-}-1}\Gamma(\z_{i}^{\star,-}+z)}
                    \prod\limits_{n=1}^{+\infty}
        \frac{\prod\limits_{i=0}^{\K^{-}-1} \Gamma(z-\z_{i,-n}^{-})}{\prod\limits_{i=0}^{\K^{-}-1}\Gamma(z-\z_{i,-n}^{-}+1)}
            \\
    \times
            \frac{\prod\limits_{i=0}^{\K^{\star,-}-1}   \Gamma(\z_{i,n}^{\star,-}-z+1)}
        {\prod\limits_{i=0}^{\K^{\star,-}-1}\Gamma(\z_{i,n}^{\star,-}+z)}
        \frac{\prod\limits_{i=0}^{\K^{\star,-}-1}(\z_{i,n}^{\star,-})^{2z-1}}
    {\prod\limits_{i=0}^{\K^{-}-1}
    [\z_{i,n}^{-}(-\z_{i,-n}^{-})]^{\frac{2z-1}{2}}}
    \\
    \times
    \exp\Big\{\sum\limits_{i=0}^{\K^{-}-1}(\z_{i,-n}^{-}[\ln(-\z_{i,-n}^{-})-1]
            +\z_{i,n}^{-}[\ln(\z_{i,n}^{-})-1])\Big \}
    \end{array}$
                        \\
                    \hline
                                        $\begin{array}{l}\frac{\s^{+}(z;\theta_1,\theta_2,q_1,q_2)}{\s^{+}(z;0,0,q_1^{\star},q_2^{\star})},\,
                            \text{ if } \theta_1,\,\theta_2\in(0,\pi),\\
            \sin\theta_1+q_2\sin\theta_2\neq 0\end{array}$  &
                $\begin{array}{l}
    \frac{\prod\limits_{i=1}^{\K^{\star,+}-1}\Gamma(\z_{i}^{\star,+}-z+1)}
    {\prod\limits_{i=0}^{\K^{+}-1}\Gamma(\z_{i}^{+}-z+1)\Gamma(z)\prod\limits_{i=1}^{\K^{\star,+}-1}\Gamma(\z_{i}^{\star,+}+z)}
    \prod\limits_{n=1}^{+\infty}
    \frac{\prod\limits_{i=0}^{\K^{+}-1} \Gamma(z-\z_{i,-n}^{+})}{\prod\limits_{i=0}^{\K^{+}-1}\Gamma(z-\z_{i,-n}^{+}+1)}
            \\
    \times
        \frac{\prod\limits_{i=0}^{\K^{\star,+}-1}
    \Gamma(\z_{i,n}^{\star,+}-z+1)}
    {\prod\limits_{i=0}^{\K^{\star,+}-1}\Gamma(\z_{i,n}^{\star,+}+z)}
        \frac{\prod\limits_{i=0}^{\K^{\star,+}-1}(\z_{i,n}^{\star,+})^{2z-1}}
    {\prod\limits_{i=0}^{\K^{+}-1}
    [\z_{i,n}^{+}(-\z_{i,-n}^{+})]^{\frac{2z-1}{2}}}
    \\
    \times
    \exp\Big\{\sum\limits_{i=0}^{\K^{+}-1}(\z_{i,-n}^{+}[\ln(-\z_{i,-n}^{+})-1]
    +\z_{i,n}^{+}[\ln(\z_{i,n}^{+})-1])\Big \}
    \end{array}$
                        \\
                    \hline
                                        $\begin{array}{l}\frac{\s^{+}(z;\theta_1,\theta_2,q_1,q_2)}{\s^{+}(z;0,0,q_1^{\star},q_2^{\star})},\,
                            \text{ if } \theta_1,\,\theta_2\in(0,\pi),\\
            \sin\theta_1+q_2\sin\theta_2= 0\end{array}$  &
                $\begin{array}{l}
\frac{[\Gamma(z)]^{\mu_1^{+}-1}}{\prod\limits_{i=\mu_1^{+}+1}^{\K^{+}-1}\Gamma(\z_{i}^{+}-z+1)}
    \frac{\prod\limits_{i=1}^{\K^{\star,+}-1}\Gamma(\z_{i}^{\star,+}-z+1)}
    {\prod\limits_{i=1}^{\K^{\star,+}-1}\Gamma(\z_{i}^{\star,+}+z)}\prod\limits_{n=1}^{+\infty}
    \frac{\prod\limits_{i=0}^{\K^{+}-1} \Gamma(z-\z_{i,-n}^{+})}{\prod\limits_{i=0}^{\K^{+}-1}\Gamma(z-\z_{i,-n}^{+}+1)}
            \\
    \times
                \frac{\prod\limits_{i=0}^{\K^{\star,+}-1}
    \Gamma(\z_{i,n}^{\star,+}-z+1)}
    {\prod\limits_{i=0}^{\K^{\star,+}-1}\Gamma(\z_{i,n}^{\star,+}+z)}
    \frac{\prod\limits_{i=0}^{\K^{\star,+}-1}(\z_{i,n}^{\star,+})^{2z-1}}
    {\prod\limits_{i=0}^{\K^{+}-1}
    [\z_{i,n}^{+}(-\z_{i,-n}^{+})]^{\frac{2z-1}{2}}}
    \\
    \times
        \exp\Big\{\sum\limits_{i=0}^{\K^{+}-1}(\z_{i,-n}^{+}[\ln(-\z_{i,-n}^{+})-1]
    +\z_{i,n}^{+}[\ln(\z_{i,n}^{+})-1])\Big \}
    \end{array}$
                        \\
                    \hline
\end{longtable}

\noindent It is worth noting that, exploiting the same arguments allows one to obtain the representation of $\mathbb{L}^{\star}(z)$ and $\d_0^{\star}$ for the rest coefficients $\s(z)$ described in Propositions \ref{p5.1}-\ref{p5.3}.

Summarizing, in order to obtain an explicit solution of homogenous equation \eqref{5.1} in the case of the coefficient $\s$ listed in Table \ref{tab:table2}, we are left to substitute the corresponding values of $\mathbb{L}^{\star}(z)$ and $\d_0^{\star}$ from Tables
 \ref{tab:table2}-\ref{tab:table3} to expression \eqref{5.8*} with $B^{\star}=0$.
\begin{remark}\label{r5.1}
Appealing to the well-known relations:
$
\sinh(q_0 z)=-i\sin(iq_0 z),$ $\cosh(q_0 z)=\cos(iq_0 z),
$
and performing the inessential modifications in the arguments above, we obtain the very same results for homogenous equation \eqref{5.1} with $\s(z)=\sinh(q_0 z),$ $\cosh(q_0 z)$,  $\tanh (q_0z)$,  $\coth(q_0 z)$ and their linear expressions.
\end{remark}
At this point, we analyze inhomogeneous equation \eqref{5.1} with
$\s(z)=\frac{\s^{+}(z;\theta_1,\theta_2,q_1,q_2)}{\s^{+}(z;0,0,q_1,q_2^{\star})}$ and demonstrate  particular solution \eqref{5.9*} with different kinds of $\P_1(z)$.
\begin{example}\label{e5.4}
We consider \eqref{5.1} with $\s(z)=\frac{\s^{+}(z;\theta_1,\theta_2,q_1,q_2)}{\s^{+}(z;0,0,q_1,q_2^{\star})}$ where $\theta_1,\theta_2\in(0,\pi)$ and $\sin\theta_1-q_2\sin\theta_2\neq 0.$ In this example, we analyze three kinds of $\P_1(z)$:
\begin{align*}
\P_{I}&:=\P_{1}(z)=1,\quad \P_{II}:=\P_{1}(z)=\prod\limits_{i=1}^{\K^{\star,+}-1}\sin\pi(\z_{i}^{\star,+}-z+1)\exp\{i\pi(\z_{i}^{\star,+}-z+1)\},\\
\P_{III}&:=\P_{1}(z)=\frac{\prod\limits_{i=1}^{\K^{\star,+}-1}\sin\pi(\z_{i}^{\star,+}-z+1)\exp\{i\pi(\z_{i}^{\star,+}-z+1)\}}
{\prod\limits_{i=1}^{\K^{+}-1}\sin\pi(\z_{i}^{+}-z+1)\exp\{i\pi(\z_{i}^{+}-z+1)+i\pi z\}\sin\pi z}
\end{align*}
and two types of $\F(z,\sigma)$. Namely, the main difference of these kinds is related with the behavior of $\F(z,\sigma)$ as $|Im z|$ tends to the infinite for each fixed $Re z$ and $\sigma\in\C$, $Re\sigma>0.$

Indeed, we assume that $\F$ does not have any poles for $z\in\C$ and $\sigma\in\C$ with $Re\sigma>0$ and

\noindent   \textbf{i:} either
\begin{equation}\label{5.10}
        \underset{|Im z|\to+\infty}{\lim}\, |F(z,\sigma)|=C,
    \end{equation}
    \noindent   \textbf{ii:} or
    \begin{equation}\label{5.11}
    \underset{|Im z|\to+\infty}{\lim}\, e^{c^{\star}\pi|Im z|}|F(z,\sigma)|=C,
    \end{equation}
    with $c^{\star}=\max\{4,2\K^{\star,+}\}$.
\end{example}
It is apparent that the  functions $\F$ considered in this example meet the requirements of Theorem \ref{t5.2}. It is worth noting that, assumptions \textbf{H12-H13} tell us that if for each $m\in\N\cup\{0\}$ the relations hold
\begin{equation}\label{5.12}
Re z\neq m+\z_{i,n}^{\star,+}+1,\, i\in\{0,...,\K^{\star,+}-1\},\, n\in\N\cup\{0\};\,
Re z\neq\z_{i,-n}^{+}-m,\, i\in\{0,...,\K^{+}-1\},\, n\in\N,
\end{equation}
then the function $\frac{Y_h(z,\sigma;\P)}{\P(z,\sigma)}$ has no poles, while $\frac{\P(z+1+\xi,\sigma)}{Y_h(z+1+\xi,\sigma;\P)}$ has no poles if for  $m\in\N\cup\{0\}$ the system is fulfilled
\begin{equation}\label{5.13}
\begin{cases}
Re z\neq m+d_0+\z_{i,n}^{+},\quad i\in\{0,1,...,\K^{+}-1\},\, n\in\N\cup\{0\},\\
Re z\neq -m-1+d_0,\\
Re z\neq -m+d_0-1-\z_{i,n}^{\star,+},\quad i\in\{0,1,...,\K^{\star,+}-1\},\, n\in\N,\\
Re z\neq -m+d_0-\z_{i}^{+},\quad i\in\{1,...,\K^{\star,+}-1\}.\\
\end{cases}
\end{equation}
Nevertheless, the function $\P(z)$ of the second and the third kinds permit to relax conditions \eqref{5.12} and \eqref{5.13}. Indeed, in the case of $\P:=\P_{II}$, this function is chosen so as to eliminate the poles of the functions $\Gamma(\z_{i}^{\star,+}-z+1),$ $i\in\{1,...,\K^{\star,+}-1\}$, which means the first inequality in \eqref{5.12} should be satisfied only for $n\in\N$. Thus, we can extend the domain of analyticity to the homogenous solution $Y_{h}(z,\sigma;\P_{II})$. However, the such choice of $\P(z)$ provides additional simple nulls to  $Y_{h}(z,\sigma;\P_{II})$ or in turn, additional poles of  $\frac{1}{Y_{h}(z+1+\xi,\sigma;\P_{II})}$. To avoid this problem, we are left to require that $
Re z \neq \z_{i}^{\star,+}+d_0-m,\quad m\in\N.
$
Coming to the function $\P_{III}(z)$, clearly,  the function $\frac{\P_{III}(z)\prod\limits_{i=1}^{\K^{\star,+}-1}\Gamma(\z_{i}^{\star,+}-z+1)}
{\Gamma(z)\prod\limits_{i=0}^{\K^{+}-1}\Gamma(\z_{i}^{+}-z+1)}$
does not have any poles if
\[
Re z\neq \z_{i}^{+}-m+1,\quad i\in\{0,...,\K^{+}-1\},\quad\text{and}\quad
Re z\neq m,
\]
while the function
$\frac{\Gamma(z+\xi+1)\prod\limits_{i=0}^{\K^{+}-1}\Gamma(\z_{i}^{+}-z-\xi)}
{\P_{III}(z+1+\xi)\prod\limits_{i=1}^{\K^{\star,+}-1}\Gamma(\z_{i}^{\star,+}-z-\xi)}
$
does not have any poles if
\[
Re z\neq -m+\z_{i}^{\star,+}+d_0,\quad i\in\{1,2,...,\K^{\star,+}-1\}.
\]
Thus, this choice of $\P(z)$ allows us to eliminate the poles of $\Gamma(\z_{i}^{\star,+}-z+1)$, $i\in\{1,2,...,\K^{\star,+}-1\}$ in the case of the function $Y_{h}(z,\sigma;\P_{III})$ and to remove the poles of $\Gamma(z+\xi+1)$ and $\Gamma(\z_i^{+}-z-\xi),$ $i\in\{0,1,...,\K^{+}-1\}$ in the case of the function $\frac{1}{Y_{h}(z+1+\xi,\sigma;\P_{III})}$.

In summary, to construct the solution of inhomogeneous equation \eqref{5.1}, we are left to select the suitable kernel $\mathcal{K}_{1}(\xi)$ such that equality \eqref{5.9**} holds. Appealing to Theorem \ref{t5.1} and asymptotic \eqref{2.5} and performing technical calculations, we end up with inequalities:
\begin{align*}
&\Big|\frac{e^{-(\varepsilon+\theta_2)|Im\, z|}}{Y_{h}(z,\sigma;\P_{I})}\Big|\leq C
\begin{cases}
e^{\frac{3\pi}{2}Im\, z},\, Im\, z\to+\infty,\\
e^{\frac{\pi}{2}Im\, z},\,  Im\, z\to-\infty,
\end{cases}
\,
\Big|\frac{e^{-(\varepsilon+\theta_2)|Im\, z|}}{Y_{h}(z,\sigma;\P_{II})}\Big|\leq C
\begin{cases}
e^{\frac{3\pi}{2}Im\, z},\qquad\quad\quad\, Im\, z\to+\infty,\\
e^{-2\pi[\K^{\star,+}-\frac{5}{4}]Im\, z},\,  Im\, z\to-\infty,
\end{cases}\\
&\Big|\frac{e^{-(\varepsilon+\theta_2)|Im\, z|}}{Y_{h}(z,\sigma;\P_{III})}\Big|\leq C
\begin{cases}
e^{(\frac{11\pi}{2}-2\pi\K^{\star,+})Im\, z},\, Im\, z\to+\infty,\\
e^{-\frac{3\pi}{2}Im\, z},\qquad\qquad Im\, z\to-\infty,
\end{cases}
\end{align*}
where $\varepsilon>0$ is enough small value, e.g. $\varepsilon<<\frac{\pi}{100}$. Thus, taking into account these relations and behavior of the function $|\F|$ as $|Im z|\to+\infty$ (see \eqref{5.10} and \eqref{5.11}), we arrive at $\mathcal{K}_{1}(\xi):=\mathcal{K}_{1}(\xi;\P_j)$,
$j=I,II,III$ listed in Tables \ref{tab:table4}-\ref{tab:table5}.
 It is apparent that,  decay of $|\F(z,\sigma)|$ as $|Im z|\to+\infty$ permits to relax assumption on $\mathcal{K}_1(\xi,\P_{j}),$
        $j=I,II,III$.
        \begin{longtable}{l|l}

  \caption{The value of  $\mathcal{K}_1(\xi)$ in the case of $\F(z,\sigma)$ satisfying \eqref{5.11}}
    \label{tab:table5}\\
    \hline
    $Y_h(z,\sigma;\P_{j})$ &$\mathcal{K}_1(\xi,\P_{j}),\, j=I,II,III$   \\
                \endfirsthead
        \hline
                   $Y_h(z,\sigma;\P_{j})$ &$\mathcal{K}_1(\xi,\P_{j}),\, j=I,II,III$   \\
                \endhead
             \hline
                                                                        $Y_h(z,\sigma;\P_{I}), Y_h(z,\sigma;\P_{II})$ & $C$ \\
        \hline
                 $Y_h(z,\sigma;\P_{III})$  & $\sin^{2}\pi d\, \sin^{-2}\pi(\xi+d)$ \\
        \hline
        \end{longtable}

\begin{longtable}{l|l}
  \caption[ ]{The value of  $\mathcal{K}_1(\xi)$ in the case of $\F(z,\sigma)$ satisfying \eqref{5.10}}
    \label{tab:table4}\\
    \hline
    $Y_h(z,\sigma;\P_{j})$ &$\mathcal{K}_1(\xi,\P_{j}),\, j=I,II,III$   \\
                \endfirsthead
        \hline
                   $Y_h(z,\sigma;\P_{j})$ &$\mathcal{K}_1(\xi,\P_{j}),\, j=I,II,III$   \\
                \endhead
             \hline
                                                                        $Y_h(z,\sigma;\P_{I})$ & $\sin^{2}\pi d\sin^{-2}\pi(\xi+d)$ \\
        \hline
         $Y_h(z,\sigma;\P_{II})$ if $\pi(\frac{5}{2}-2\K^{\star,+})-\theta_2-\varepsilon>0$ & $\sin^{2}\pi d\sin^{-2}\pi(\xi+d)$ \\
        \hline
         $Y_h(z,\sigma;\P_{II})$ if $\pi(\frac{5}{2}-2\K^{\star,+})-\theta_2-\varepsilon<0$ & $(\sin\pi d)^{2\K^{\star,+}}(\sin\pi(\xi+d))^{-2\K^{\star,+}}$ \\
        \hline
         $Y_h(z,\sigma;\P_{III})$  & $\sin^{4}\pi d\sin^{-4}\pi(\xi+d)$ \\
        \hline
        \end{longtable}

    \noindent   Finally, collecting the results above with Theorems \ref{t5.1}-\ref{t5.2}, we derive the general solution to \eqref{5.1} in the form
        \begin{equation}\label{5.14}
        Y(z,\sigma)=Y_h(z,\sigma;\P)+Y_{h}(z,\sigma;\P_j)\int\limits_{\ell_{d_0}}\frac{\F(z+\xi,\sigma)[\cot\pi\xi+i]\mathcal{K}_{1}(\xi,\P_j)}
        {[a_1\sigma+a_2\sigma^{\nu}]Y_h(z+\xi+1,\sigma;\P_j)}d\xi,\, j=I,II,III,
        \end{equation}
        where $\mathcal{K}_1(\xi;\P_j)$ are given in Tables \ref{tab:table5}-\ref{tab:table4}, and
        \begin{align*}
        Y_h(z,\sigma;\P)&=\bigg[\frac{(\sin\theta_1+q_2\sin\theta_2)\prod\limits_{i=1}^{\K^{\star,+}-1}(\z_i^{\star,+})^{2}}{(1+q_2^{\star}q_1)\prod\limits_{i=0}^{\K^{+}-1}\z_i^{+}}\bigg]^{z-1/2}\frac{e^{i\pi z}(a_1\sigma+a_2\sigma^{\nu})^{\frac{1-2z}{2}}}
        {\prod\limits_{i=0}^{\K^{+}-1}\Gamma(z)\Gamma(\z_{i}^{+}-z+1)}\\
        &
        \times\prod\limits_{i=1}^{\K^{\star,+}-1}\frac{\Gamma(\z_{i}^{\star,+}-z+1)}{\Gamma(\z_{i}^{\star,+}+z)}\prod\limits_{n=1}^{+\infty}
        \prod\limits_{i=0}^{\K^{+}-1}\frac{\Gamma(z-\z_{i,-n}^{+})}{\Gamma(\z_{i,n}^{+}-z+1)}
        \prod\limits_{i=0}^{\K^{\star,+}-1}\frac{\Gamma(\z_{i,n}^{\star,+}-z+1)}{\Gamma(\z_{i,n}^{\star,+}+z)}\\
        &
        \times
        \frac{\prod\limits_{i=0}^{\K^{\star,+}-1}(\z_{i,n}^{\star,+})^{2z-1}}
        {\prod\limits_{i=0}^{\K^{+}-1}[\z_{i,n}^{+}(-\z_{i,-n}^{+})]^{\frac{2z-1}{2}}}
        \exp\Big\{\sum\limits_{i=0}^{\K^{+}-1}(\z_{i,-n}^{+}[\ln(-\z_{i,-n}^{+})-1]+
        \z_{i,n}^{+}[\ln(\z_{i,n}^{+})-1]
        )\Big\}.
        \end{align*}
        It is worth noting that, if we search a bounded solution of \eqref{5.1} if $|Im z|\to+\infty$, then we should put $\P\equiv 0$ in
        $Y_{h}(z,\sigma;\P)$, which means $Y_{h}(z,\sigma;\P)=0$ and   \eqref{5.14} is rewritten as         $
        Y(z,\sigma)=Y_{ih}(z,\sigma).
        $
        \begin{remark}\label{r5.2}
        Clearly, $\F(z,\sigma)=1$ meets requirements of Theorem \ref{t5.2} and satisfies equality \eqref{5.10}, while $\F(z,\sigma)=\exp\{-c^{\star}(z\bar{z})\}$ for each fixed $Re z$ satisfies equality \eqref{5.11}.
        \end{remark}
                \begin{remark}\label{r5.3}
        Observing arguments of Sections \ref{s3}-\ref{s5}, one can easily conclude that all results of Section \ref{s5} (possibly with minor modifications) can be extended to the functions $\s$ presented as either $\s=\prod_{j=1}^{N_{I}}\s_{1}^{j}$ or $\s=\frac{\prod_{j=1}^{N_{I}}\s_{1}^{j}}{\prod_{j=1}^{N_{II}}\s_{2}^{j}}$ with fixed integer  $N_{I}$ and $N_{II}$  and the functions $\s_{1}^{j}$ and $\s_{2}^{j}$ having the properties of the functions $\s_{1}$ and $\s_{2}$, respectively.

        Besides, in light of Remark \ref{r3.5}, the results of Theorems \ref{t5.1} and \ref{t5.2} hold for multidimensional equation similar to \eqref{2.14} with $\mathbf{S}(z)=\prod_{j=1}^{k}\s^{j}(z)$, where $\s^{j}(z)$ obeys properties of the function $\s(z)$.
        \end{remark}
        It is worth noting that, results obtained in Section \ref{s5} mean that \textbf{FDEs} studied in \cite{BF,BV1,BV4,BV5,DV,LK,La,SF} (see also references therein) fall in our analysis.

\section{Explicit solutions of transmission boundary value problems with\\ a fractional dynamic boundary condition in plane corners }
\label{s6}
\noindent In this Section, we discuss a construction of a solution to a transmission problem \eqref{6.1}-\eqref{6.5} via results of Sections \ref{s2} and \ref{s5}. Other words, we demonstrate the application of solutions \eqref{i.1} to look for an explicit solution to the non-classical transmission  boundary value problem.

We assume that $\omega_{0}=q\pi/p\in[0,\pi/2)$ with $p$ and $q$ satisfying assumptions of Corollary \ref{c5.1}.
Coming to the coefficients involved, $a_1$ and $a_2$ meet requirements \textbf{H1}, and $a_3$, $a_4,$ $\mathfrak{K}\neq 1$ are positive quantities,  and $\mathfrak{s}_0$ is a given real number.

As for the right-hand sides in \eqref{6.1}-\eqref{6.5}, we state the following hypothesis.

\noindent\textbf{H14:} We require that for some fixed $\alpha\in(0,1),$
$$f_1,f\in\mathcal{C}([0,T],\mathcal{C}^{1+\alpha}(g)),\,\,f_{2}\in\mathcal{C}([0,T],\mathcal{C}^{\alpha}(\partial G_1\backslash g)),\,\,f_{3}\in\mathcal{C}([0,T],\mathcal{C}^{\alpha}(\partial G_2\backslash g)),
$$
$$f_{0,1}\in\mathcal{C}([0,T],\mathcal{C}^{\alpha}(\bar{G}_1)),\,
f_{0,2}\in\mathcal{C}([0,T],\mathcal{C}^{\alpha}(\bar{G}_2)),$$
and
$$f, f_{1},f_{2},f_{3},f_{0,1},f_{0,2}\equiv 0\quad \text{if either}\quad r\in[0,\varepsilon_1]\cup[R_1,+\infty)\quad \text{or}\quad t\leq 0,
$$
for some fixed small positive $\varepsilon_1$ and enough large $R_1$.

It is apparent that  assumptions on the smoothness of the right-hand sides  may not be enough to prove  existence of a classical smooth solution, but \textbf{H14} is enough to construct explicit solutions to \eqref{6.1}-\eqref{6.5}. Here we focus on the  solution of the problem which  vanishes as $r\to+\infty$.
First of all, we remark that the explicit solutions of classical Dirichlet-transmission problem to Poisson equations  in domain $G_1\cup g\cup G_2$ were searched in
Section 9.4 \cite{V2}. In light of this fact,  we are left to construct solutions  in the case of homogenous equations \eqref{6.1} and homogenous Dirichlet and transmission conditions  \eqref{6.4} and \eqref{6.5}, i.e. we put
\begin{equation}\label{6.0*}
f_{0,1}, f_{0,2},f_{1},f_2,f_3\equiv 0.
\end{equation}
In order to build the solution of \eqref{6.1}-\eqref{6.5}, we will exploit the following strategy. In the first stage, employing special change of variables and then Fourier and Laplace transforms, we reduce problem \eqref{6.1}-\eqref{6.5} to a functional difference equation. The second stage is related with applying Theorems \ref{t3.1} and \ref{t3.2} to this equation, and in turn with obtaining solution via formula like \eqref{2.4*}, \eqref{2.11*} and \eqref{2.11}. Finally, performing the inverse Laplace and Fourier transformation, we conclude with the explicit integral form of the solution to \eqref{6.1}-\eqref{6.5}.

\noindent$\bullet$ \textit{The First Stage.} We start with standard change of variable (see e.g. (2.6) in \cite{BF} or (3.9) in \cite{BV1}) in the case of a boundary value problem stated in the domain with corner points:
\begin{equation}\label{6.7}
x_1=\ln r,\quad x_2=\varphi.
\end{equation}
This map transforms the corners $G_1$ and $G_2$ to the strips $\mathfrak{G}_1$ and $\mathfrak{G}_2$, respectively,
\[
\mathfrak{G}_1=\{(x_1,x_2):\, x_1\in\R,\, -\pi/2<x_2<\omega_0\},\quad
\mathfrak{G}_2=\{(x_1,x_2):\, x_1\in\R,\, \omega_0<x_2<\pi/2\},
\]
while the image of the boundary $g$ is $
\mathfrak{g}=\{(x_1,x_2):\, x_1\in\R,\, x_2=\omega_0\}.
$

Performing the change of variables \eqref{6.7} in relations \eqref{6.1}-\eqref{6.5} and using the same notations for the functions $u_1$ and $u_2$ and $f$ in the new variables, we end up with the problem
\begin{equation}\label{6.8}
\begin{cases}
\Delta u_1=0\qquad\text{in}\qquad \mathfrak{G}_{1,T},\qquad\quad u_1(x_1,x_2,0)=0\qquad\text{in}\qquad \bar{\mathfrak{G}}_{1},\\
\Delta u_2=0\qquad\text{in}\qquad \mathfrak{G}_{2,T},\qquad\quad u_2(x_1,x_2,0)=0\qquad\text{in}\qquad \bar{\mathfrak{G}}_{2},\\
\exp\{(\mathfrak{s}_0+1)x_1\}\Big[a_1\frac{\partial}{\partial t}(u_1-u_2)+a_2\mathbb{D}_{t}^{\nu}(u_1-u_2)\Big]\\
-
\frac{\partial}{\partial x_2}(u_1-u_2)+a_3
\frac{\partial}{\partial x_1}(u_1-u_2)=e^{x_1}f(x_1,t)\qquad\quad\text{on}\qquad \mathfrak{g}_{T},\\
\frac{\partial u_1}{\partial x_2}-\mathfrak{K}\frac{\partial u_2}{\partial x_2}+\mathfrak{K}a_4
\frac{\partial}{\partial x_1}(u_1-u_2)=0\qquad\qquad\qquad \qquad\text{on}\quad \mathfrak{g}_{T},\\
u_1(x_1,-\pi/2,t)=0\quad\text{and}\quad u_2(x_1,\pi/2,t)=0\qquad\qquad\text{for}\quad x_1\in\R,\, t\in[0,T].
\end{cases}
\end{equation}
After that, as usual to the analysis of boundary value problems in domains with singular boundaries (see e.g. \cite[Section 4]{Gri}), we will search a solution to \eqref{6.8} in the form:
\[
u_1(x_1,x_2,t)=e^{\mathfrak{s}x_1}U_1(x,_1,x_2,t),\quad u_2(x_1,x_2,t)=e^{\mathfrak{s}x_1}U_2(x,_1,x_2,t),
\]
where the value $\mathfrak{s}\neq 0$ will be specified below, and new unknown functions solve the problem
\begin{equation*}
\begin{cases}
\Delta U_1+2\mathfrak{s}\frac{\partial U_1}{\partial x_1}+\mathfrak{s}^{2}U_1=0\qquad \text{in}\quad \mathfrak{G}_{1,T},\qquad\quad U_1(x_1,x_2,0)=0\qquad\text{in}\quad \bar{\mathfrak{G}}_{1},\\
\Delta U_2+2\mathfrak{s}\frac{\partial U_2}{\partial x_1}+\mathfrak{s}^{2}U_2=0\qquad\text{in}\quad \mathfrak{G}_{2,T},\quad\qquad U_2(x_1,x_2,0)=0\qquad\text{in}\quad \bar{\mathfrak{G}}_{2},\\
\exp\{(\mathfrak{s}_0+1)x_1\}\Big[a_1\frac{\partial}{\partial t}(U_1-U_2)+a_2\mathbb{D}_{t}^{\nu}(U_1-U_2)\Big]\\
-
\frac{\partial}{\partial x_2}(U_1-U_2)+a_3\Big[
\frac{\partial}{\partial x_1}(U_1-U_2)+\mathfrak{s}(U_1-U_2)\Big]=e^{(1-\mathfrak{s})x_1}f(x_1,t)\quad\text{on}\quad \mathfrak{g}_{T},\\
\frac{\partial U_1}{\partial x_2}-\mathfrak{K}\frac{\partial U_2}{\partial x_2}+\mathfrak{K}a_4
\Big[\frac{\partial}{\partial x_1}(U_1-U_2)+\mathfrak{s}(U_1-U_2)\Big]=0\qquad\qquad\qquad\qquad\text{on}\quad \mathfrak{g}_{T},\\
U_1(x_1,-\pi/2,t)=0\qquad\text{and}\qquad U_2(x_1,\pi/2,t)=0\quad\qquad\qquad \text{for}\quad x_1\in\R,\, t\in[0,T].
\end{cases}
\end{equation*}
It is worth noting that the introducing new functions $U_1$ and $U_2$ and rewriting problem \eqref{6.8} in the terms of these functions permit to study the corresponding problem in either H\"{o}lder or Sobolev spaces instead of using weighted spaces (see for detail e.g. \cite[Section2]{BF} or \cite[Section 3.1]{BV5}).

Then, we denote by $\tilde{U}(\lambda,x_2,t)$ the Fourier transform of $U(x_1,x_2,t)$, and by $\hat{U}(x_1,x_2,\sigma)$ the Laplace transform of $U(x_1,x_2,t)$, and use the notation $*$ instead of $\hat{\tilde{\, }}$. Thus, performing these transformations in the relations above, and setting
$
\mathfrak{x}=i\lambda+\mathfrak{s},$ $\mathfrak{s}^{\star}=\mathfrak{s}_0+1,
 $
we have
\begin{equation}\label{6.9}
\begin{cases}
\frac{\partial^{2}U^{*}_{1}}{\partial x_2^{2}}+\mathfrak{x}^{2}U^{*}_{1}=0,\quad \tilde{U_1}(\lambda,x_2,0)=0,\quad x_2\in(-\pi/2,\omega_0),\\
\frac{\partial^{2}U^{*}_{2}}{\partial x_2^{2}}+\mathfrak{x}^{2}U^{*}_{2}=0,\quad \tilde{U_2}(\lambda,x_2,0)=0,\quad x_2\in(\omega_0,\pi/2),\\
(a_1\sigma+a_2\sigma^{\nu})[U_1^{*}(\lambda+i\mathfrak{s}^{\star},\omega_0,\sigma)-U_2^{*}(\lambda+i\mathfrak{s}^{\star},\omega_0,\sigma)]
-\frac{\partial}{\partial x_2}[U_1^{*}(\lambda,\omega_0,\sigma)-U_2^{*}(\lambda,\omega_0,\sigma)]\\
+a_3\mathfrak{x}[U_1^{*}(\lambda,\omega_0,\sigma)-U_2^{*}(\lambda,\omega_0,\sigma)]=f^{*}(\lambda+i(1-\mathfrak{s}),\sigma),\\
\frac{\partial U_1^{*}}{\partial x_2}(\lambda,\omega_0,\sigma)-\mathfrak{K}\frac{\partial U_2^{*}}{\partial x_2}(\lambda,\omega_0,\sigma)
+\mathfrak{K}a_4\mathfrak{x}[U_1^{*}(\lambda,\omega_0,\sigma)-U_2^{*}(\lambda,\omega_0,\sigma)]=0,\\
U_1^{*}(\lambda,-\pi/2,\sigma)=0,\quad U_2^{*}(\lambda,\pi/2,\sigma)=0.
\end{cases}
\end{equation}
To satisfy equations and boundary conditions on $x_2=\pm\pi/2$, we set
\begin{equation}\label{6.9*}
U_1^{\star}(\lambda,x_2,\sigma)=\mathcal{M}_1(\lambda, \sigma)\sin\mathfrak{x}(x_2+\pi/2),\quad
U_2^{\star}(\lambda,x_2,\sigma)=\mathcal{M}_2(\lambda, \sigma)\sin\mathfrak{x}(x_2-\pi/2),
\end{equation}
where unknown functions $\mathcal{M}_1(\lambda, \sigma)$ and $\mathcal{M}_2(\lambda,\sigma)$ will be found via transmission conditions in \eqref{6.9}. To this end, substituting $U_1^{*}$ and $U^{*}_2$ to the transmission conditions and putting
\[
\mathcal{N}(\lambda)=\frac{\cos\mathfrak{x}(\omega_0-\pi/2)+a_4\sin\mathfrak{x}(\omega_0-\pi/2)}{\cos\mathfrak{x}(\omega_0+\pi/2)+\mathfrak{K}a_4\sin\mathfrak{x}(\omega_0+\pi/2)},
\]
 we arrive at the system with unknown functions $\mathcal{M}_1$ and $\mathcal{M}_2$:
\begin{equation}\label{6.10}
\begin{cases}
\mathcal{M}_1(\lambda,\sigma)=\mathfrak{K}\mathcal{N}(\lambda)\mathcal{M}_2(\lambda,\sigma),\\
[a_1\sigma+a_2\sigma^{\nu}]\mathcal{M}_2(\lambda+i\mathfrak{s}^{\star},\sigma)[\mathfrak{K}\mathcal{N}(\lambda+i\mathfrak{s}^{\star})\sin\mathfrak{x}(\omega_0+\pi/2)-\sin\mathfrak{x}(\omega_0-\pi/2)]\\
-\mathfrak{x}\mathcal{M}_2(\lambda,\sigma)\{a_3\sin\mathfrak{x}(\omega_0-\pi/2)-\cos\mathfrak{x}(\omega_0-\pi/2)
\\
+\mathfrak{K}\mathcal{N}(\lambda)[\cos\mathfrak{x}(\omega_0+\pi/2)-a_3\sin\mathfrak{x}(\omega_0-\pi/2)]\}=f^{*}(\lambda+i(1-\mathfrak{s}),\sigma).
\end{cases}
\end{equation}
In virtue of this system, we are left to look for the function $\mathcal{M}_2$. Denoting
\[
\mathcal{N}_1(\lambda)=\mathfrak{K}\mathcal{N}\sin\mathfrak{x}(\omega_0+\pi/2)-\sin\mathfrak{x}(\omega_0-\pi/2)
\]
and substituting the expression of $\mathcal{M}_1$ to the second equation in \eqref{6.10}, we obtain
\begin{align*}
&\mathfrak{x}\mathcal{M}_2(\lambda,\sigma)\{\mathfrak{K}\mathcal{N}(\lambda)\cos\mathfrak{x}(\omega_0+\pi/2)-a_3\mathcal{N}_1(\lambda)-\cos\mathfrak{x}(\omega_0-\pi/2)\}\\
&
-
[a_1\sigma+a_2\sigma^{\nu}]\mathcal{M}_2(\lambda+i\mathfrak{s}^{\star},\sigma)\mathcal{N}_1(\lambda+i\mathfrak{s}^{\star})
=-f^{*}(\lambda+i(1-\mathfrak{s}),\sigma)
\end{align*}
which in turn is reduced to a functional equation in unknown function $\mathcal{U}:=\mathcal{M}_2(\lambda,\sigma)\mathcal{N}_1(\lambda),$
\[
[a_1\sigma+a_2\sigma^{\nu}]\mathcal{U}(\lambda+i\mathfrak{s}^{\star},\sigma)-\mathfrak{x}\mathcal{G}(\lambda)\mathcal{U}(\lambda,\sigma)
=f^{*}(\lambda+i(1-\mathfrak{s}),\sigma).
\]
Here we put
\[
\mathcal{G}(\lambda)=\frac{\mathfrak{K}\mathcal{N}(\lambda)\cos\mathfrak{x}(\omega_0+\pi/2)-\cos\mathfrak{x}(\omega_0-\pi/2)}{\mathcal{N}_1(\lambda)}-a_3.
\]
\begin{remark}\label{r6.1}
It is apparent that if $\mathfrak{s}_0=-1,$ i.e. $\mathfrak{s}^{\star}=0$, then the function
$
\mathcal{U}=\frac{f^{*}(\lambda+i(1-\mathfrak{s}),\sigma)}{a_1\sigma+a_2\sigma^{\nu}-\mathfrak{x}\,\mathcal{G}(\lambda)}
$
solves the equation above, which in turn entails
\[
U^{*}_1=\mathfrak{K}\sin\mathfrak{x}(x_2+\pi/2)\mathcal{U}(\lambda,\sigma)\mathcal{N}(\lambda)\mathcal{N}_1(\lambda),\qquad
U^{*}_2=\sin\mathfrak{x}(x_2-\pi/2)\mathcal{U}(\lambda,\sigma)\mathcal{N}_1(\lambda).
\]
\end{remark}
Taking into account this remark, we will assume further $\mathfrak{s}^{\star}\neq 0$ and introduce the new variable
\begin{equation}\label{6.11*}
\lambda=i\mathfrak{s}^{\star}\rho,
\end{equation}
and new unknown function $
V(\rho,\sigma)=\mathcal{U}(i\mathfrak{s}^{\star}\rho,\sigma),
$
we rewrite the equation above in the form
\begin{equation}\label{6.11}
[a_1\sigma+a_2\sigma^{\nu}]V(\rho+1,\sigma)-[\mathfrak{s}-\mathfrak{s}^{\star}\rho]\mathcal{G}(i\mathfrak{s}^{\star}\rho)V(\rho,\sigma)=\mathfrak{F}(\rho,\sigma),
\end{equation}
where $\mathfrak{F}(\rho,\sigma):=f^{*}(i\mathfrak{s}^{\star}\rho+i(1-\mathfrak{s}),\sigma)$.

\noindent Thus, we have reduced original problem \eqref{6.1}-\eqref{6.5} to functional equation \eqref{6.11}.

\noindent\textit{The Second Stage.} In order to apply the results of Sections \ref{s2} and \ref{s5} to \eqref{6.11}, we are left to factorize the function $\mathcal{G}(i\mathfrak{s}^{\star}\rho)$. To this end, coming to the function $\mathcal{G}(\lambda)$, we substitute $\mathcal{N}(\lambda)$ and $\mathcal{N}_1(\lambda)$ to the expression of $\mathcal{G}(\lambda)$ and performing tedious technical calculations, we derive
\[
\mathcal{G}(\lambda)=\frac{\cos\pi\mathfrak{x}-\frac{a_3(1+\mathfrak{K}\, )+2\mathfrak{K}\, a_4}{\mathfrak{K}-1}\sin\pi\mathfrak{x}+\cos2\omega_0\mathfrak{x}-a_3\sin2\omega_0\mathfrak{x}}
{\sin2\omega_0\mathfrak{x}+\frac{\mathfrak{K}+1}{\mathfrak{K}-1}\sin\pi\mathfrak{x}}.
\]
Finally, denoting
\begin{align*}
&q_2^{*}=\frac{\mathfrak{K}+1}{|\mathfrak{K}-1|},\,q_2=\Big(\frac{[a_3(\mathfrak{K}+1)+2\mathfrak{K}\, a_4]^{2}+(\mathfrak{K}-1)^{2}}{(\mathfrak{K}-1)^{2}(1+a_3^{2})}\Big)^{\frac{1}{2}},\,\,
 \sin\theta_1=(1+a_3^{2})^{-\frac{1}{2}},\, \cos\theta_1=a_3(1+a_3^{2})^{-\frac{1}{2}},\\
&\sin\theta_2=\Big(1+\frac{[2\mathfrak{K}\, a_4+a_3(1+\mathfrak{K})]^{2}}{[\mathfrak{K}-1]^{2}}\Big)^{-\frac{1}{2}},\,
\cos\theta_2=\frac{[a_3(\mathfrak{K}+1)+2\mathfrak{K}\, a_4]\text{sgn}(\mathfrak{K}-1)}{\sqrt{(\mathfrak{K}-1)^{2}+[2\mathfrak{K}\, a_4+a_3(1+\mathfrak{K})]^{2}}},
\end{align*}
we rewrite $\mathcal{G}(\lambda)$ in the form
\[
\mathcal{G}(\lambda)=-
\frac{\sqrt{1+a_3^{2}}[\sin(2\mathfrak{x}\omega_0-\theta_1)+q_2\sin(\pi\mathfrak{x}-\theta_2)]}
{\sin 2\omega_0\mathfrak{x}+\text{sgn}(\mathfrak{K}-1)q_2^{*}\sin\pi\mathfrak{x}}.
\]
Keeping in mind of assumptions on the coefficients, we easily conclude
\begin{equation}\label{6.12*}
\theta_1\in(0,\pi/2)\quad\text{and}\quad \theta_2\in(0,\pi),\quad q_2^{*}>1,\quad q_2>1,\quad \sin\theta_1+q_2\sin\theta_2=\frac{2}{\sqrt{1+a_3^{2}}}.
\end{equation}
Then, performing the change of variable \eqref{6.11*} and taking into account the relations
$
\frac{\pi}{2\omega_0}=\frac{p}{2q}>1,
  $
we deduce
\begin{equation*}
\mathcal{G}(i\mathfrak{s}^{\star}\rho)=-\sqrt{1+a_3^{2}}
\begin{cases}
\frac{\mathcal{S}^{+}(2\omega_0(\mathfrak{s}-\mathfrak{s}^{\star}\rho);\theta_1,\theta_2,\frac{p}{2q},q_2)}{\mathcal{S}^{+}(2\omega_0(\mathfrak{s}-\mathfrak{s}^{\star}\rho);0,0,\frac{p}{2q},q^{*}_2)},\quad \mathfrak{K}>1,\\
\frac{\mathcal{S}^{+}(2\omega_0(\mathfrak{s}-\mathfrak{s}^{\star}\rho);\theta_1,\theta_2,\frac{p}{2q},q_2)}
{\mathcal{S}^{-}(2\omega_0(\mathfrak{s}-\mathfrak{s}^{\star}\rho);0,0,\frac{p}{2q},q^{*}_2)},\quad \mathfrak{K}<1,
\end{cases}
\end{equation*}
This representation means that the function $\mathcal{G}(i\mathfrak{s}^{\star},\rho)$ is the function $\mathcal{S}$ in the \textbf{FTK} case (see Section \ref{s5}). Appealing to Corollary \ref{c5.1}, we denote by
\begin{align*}
\bar{z}_{i,n}^{+}&:=\bar{z}_{i,n}^{+}(\theta_1,\theta_2,q_2),\quad\bar{z}_{i,-n}^{+}:=\bar{z}_{i,n}^{+}(\theta_1,\theta_2,q_2),\quad\bar{\K}^{+}:=\bar{\K}^{+}(\theta_1,\theta_2,q_2),\\
z_{i,n}^{+}&:=z_{i,n}^{+}(0,0,q^{*}_2),\quad\K^{+}:=\K^{+}(0,0,q^{*}_2),\quad
z_{i,n}^{-}:=z_{i,n}^{-}(0,0,q^{*}_2),\quad\K^{-}:=\K^{-}(0,0,q^{*}_2),
\end{align*}
the zeros and their number in $[0,4\pi q)$ of the functions $\s^{+}(z;\theta_1,\theta_2,\frac{p}{2q},q_2),$ $\s^{+}(z;0,0,\frac{p}{2q},q^{*}_2)$ and  $\s^{-}(z;0,$ $0,\frac{p}{2q},q^{*}_2),$ respectively. In virtue of \eqref{6.12*}, we can employ Proposition \ref{p5.1} to conclude that
\begin{align*}
&\mathcal{G}(i\mathfrak{s}^{\star}\rho)=\frac{\prod\limits_{i=0}^{\bar{\K}^{+}-1}[\bar{z}_{i}^{+}+2\omega_0(\mathfrak{s}^{\star}\rho-\mathfrak{s})]\prod\limits_{i=1}^{\K^{+}-1}(z_{i}^{+})^{2}}
{\omega_0(1+q_2^{*}q_1)(\mathfrak{s}-\mathfrak{s}^{*}\rho)\prod\limits_{i=1}^{\K^{+}-1}[z_i^{+}+2\omega_0(\mathfrak{s}^{\star}\rho-\mathfrak{s})]
[z_i^{+}-2\omega_0(\mathfrak{s}^{\star}\rho-\mathfrak{s})]\prod\limits_{i=0}^{\bar{\K}^{+}-1}\bar{z}_{i}^{+}
}\\
&
\times
\prod\limits_{n=1}^{+\infty}
\frac{\prod\limits_{i=0}^{\bar{\K}^{+}-1}\Big[\bar{z}_{i,n}^{+}+2\omega_0(\mathfrak{s}^{\star}\rho-\mathfrak{s})\Big]
\Big[-\bar{z}_{i,-n}^{+}+2\omega_0(\mathfrak{s}-\mathfrak{s}^{\star}\rho)\Big]}
{\prod\limits_{i=0}^{\K^{+}-1}\Big[z_{i,n}^{+}+2\omega_0(\mathfrak{s}^{\star}\rho-\mathfrak{s})\Big]
\Big[z_{i,n}^{+}+2\omega_0(\mathfrak{s}-\mathfrak{s}^{\star}\rho)\Big]}
\frac{\prod\limits_{i=0}^{\K^{+}-1}(z_{i,n}^{+})^{2}}
{\prod\limits_{i=0}^{\bar{\K}^{+}-1}\bar{z}_{i,n}^{+}(-\bar{z}_{i,n}^{+})},\,\text{if}\, \mathfrak{K}>1,
\end{align*}
and
\begin{align*}
&\mathcal{G}(i\mathfrak{s}^{\star}\rho)=\frac{\prod\limits_{i=0}^{\bar{\K}^{+}-1}[\bar{z}_{i}^{+}+2\omega_0(\mathfrak{s}^{\star}\rho-\mathfrak{s})]\prod\limits_{i=1}^{\K^{-}-1}(z_{i}^{-})^{2}}
{\omega_0(1-q_2^{*}q_1)(\mathfrak{s}-\mathfrak{s}^{*}\rho)\prod\limits_{i=1}^{\K^{-}-1}[z_i^{-}+2\omega_0(\mathfrak{s}^{\star}\rho-\mathfrak{s})]
[z_i^{-}-2\omega_0(\mathfrak{s}^{\star}\rho-\mathfrak{s})]\prod\limits_{i=0}^{\bar{\K}^{+}-1}\bar{z}_{i}^{+}
}\\
&
\times
\prod\limits_{n=1}^{+\infty}
\frac{\prod\limits_{i=0}^{\bar{\K}^{+}-1}\Big[\bar{z}_{i,n}^{+}+2\omega_0(\mathfrak{s}^{\star}\rho-\mathfrak{s})\Big]
\Big[-\bar{z}_{i,-n}^{+}+2\omega_0(\mathfrak{s}-\mathfrak{s}^{\star}\rho)\Big]}
{\prod\limits_{i=0}^{\K^{-}-1}\Big[z_{i,n}^{-}+2\omega_0(\mathfrak{s}^{\star}\rho-\mathfrak{s})\Big]
\Big[z_{i,n}^{-}+2\omega_0(\mathfrak{s}-\mathfrak{s}^{\star}\rho)\Big]}
\times
\frac{\prod\limits_{i=0}^{\K^{-}-1}(z_{i,n}^{-})^{2}}
{\prod\limits_{i=0}^{\bar{\K}^{+}-1}\bar{z}_{i,n}^{+}(-\bar{z}_{i,n}^{+})},\,\text{if}\, \mathfrak{K}<1.
\end{align*}
At this point, we state the first assumption on $\mathfrak{s}$.

\noindent\textbf{H15:} We require that $\mathfrak{s}$ satisfies inequalities: $\, \sin2\omega_0\mathfrak{s}+\text{sgn} (\mathfrak{K}-1)q^{*}_{2}\sin\pi\mathfrak{s}\neq 0,$ and
\[
|\mathfrak{s}|<4\pi q,\quad
2\omega_0\mathfrak{s}\neq -\bar{z}^{+}_{i,n},\, 2\omega_0\mathfrak{s}\neq \bar{z}^{+}_{i,-n},\quad i\in\{0,1,...,\bar{\K}^{+}-1\},\, \]
and
\[
2\omega_0\mathfrak{s}\neq \pm z^{-}_{i,n},\,\quad i\in\{0,1,...,\K^{-}-1\},\, n\in \N\cup\{0\},\quad\text{if}\, \mathfrak{K}<1,
\]
while
\[
2\omega_0\mathfrak{s}\neq \pm z^{+}_{i,n},\,  i\in\{0,1,...,\K^{+}-1\},\, n\in \N\cup\{0\},\quad\text{if}\, \mathfrak{K}>1.
\]
Then, setting for $n\in \N\cup\{0\}$
\[
\bar{\K}:=
\begin{cases}
\K^{+},\quad\text{if}\quad\mathfrak{K}>1,\\
\K^{-},\quad\text{if}\quad\mathfrak{K}<1,
\end{cases}
\]
\[
\bar{Z}^{+}_{i,n}:=\frac{\bar{z}_{i,n}^{+}-2\omega_0\mathfrak{s}}{|\mathfrak{s}^{\star}|2\omega_0},\quad
\bar{Z}^{+}_{i,-n}:=\frac{-\bar{z}_{i,-n}^{+}+2\omega_0\mathfrak{s}}{|\mathfrak{s}^{\star}|2\omega_0},\quad i\in\{0,1,...,\bar{\K}^{+}-1\},\]
\[
Z_{i,n}=\begin{cases}
\frac{z_{i,n}^{+}-2\omega_0\mathfrak{s}}{|\mathfrak{s}^{\star}|2\omega_0},\quad \text{if}\quad \mathfrak{K}>1,\\
\frac{z_{i,n}^{-}-2\omega_0\mathfrak{s}}{|\mathfrak{s}^{\star}|2\omega_0},\quad \text{if}\quad \mathfrak{K}<1,
\end{cases}\quad
Z_{i,-n}=\begin{cases}
\frac{z_{i,n}^{+}+2\omega_0\mathfrak{s}}{|\mathfrak{s}^{\star}|2\omega_0},\quad \text{if}\quad \mathfrak{K}>1,\\
\frac{z_{i,n}^{-}+2\omega_0\mathfrak{s}}{|\mathfrak{s}^{\star}|2\omega_0},\quad \text{if}\quad \mathfrak{K}<1,
\end{cases}
\, i\in\{0,1,..., \bar{\K}\},
\]
and taking into account \textbf{H15} and  the easy verified equality
\[
\mathcal{G}(0)=-\sqrt{1+a_3^{2}}\frac{\sin(2\omega_0\mathfrak{s}-\theta_1)+q_2\sin(\pi\mathfrak{s}-\theta_2)}
{\sin2\omega_0\mathfrak{s}+\text{sgn}(\mathfrak{K}-1)q_2^{*}\sin\pi\mathfrak{s}},
\]
we have
 \begin{align}\label{6.12}\notag
\mathcal{G}(i\mathfrak{s}^{\star}\rho)&=\frac{\mathcal{G}(0)\mathfrak{s}}{(\mathfrak{s}-\mathfrak{s}^{\star}\rho)}
\frac{\prod\limits_{i=0}^{\bar{\K}^{+}-1}[\bar{Z}^{+}_{i,0}+\text{sgn}(\mathfrak{s}^{\star})\rho]}
{\prod\limits_{i=1}^{\bar{\K}-1}[Z_{i,0}+\text{sgn}(\mathfrak{s}^{\star})\rho][Z_{i,-0}-\text{sgn}(\mathfrak{s}^{\star})\rho]}
\frac{\prod\limits_{i=1}^{\bar{\K}-1}Z_{i,0}Z_{i,-0}}{\prod\limits_{i=0}^{\bar{\K}^{+}-1}\bar{Z}_{i,0}^{+}}\\
&
\times
\prod\limits_{n=1}^{+\infty}\frac{\prod\limits_{i=0}^{\bar{\K}^{+}-1}[\bar{Z}^{+}_{i,n}+\text{sgn}(\mathfrak{s}^{\star})\rho]
[\bar{Z}^{+}_{i,-n}-\text{sgn}(\mathfrak{s}^{\star})\rho]}
{\prod\limits_{i=0}^{\bar{\K}-1}[Z_{i,n}+\text{sgn}(\mathfrak{s}^{\star})\rho][Z_{i,-n}-\text{sgn}(\mathfrak{s}^{\star})\rho]}
\frac{\prod\limits_{i=0}^{\bar{\K}-1}Z_{i,n}Z_{i,-n}}{\prod\limits_{i=0}^{\bar{\K}^{+}-1}\bar{Z}_{i,n}^{+}\bar{Z}_{i,-n}^{+}}.
\end{align}
In virtue of assumptions \textbf{H15}  and Corollary \ref{c5.1}, we conclude that the sequences: $\{Z_{i,n}\}_{n=0}^{\infty}$, $\{Z_{i,-n}\}_{n=0}^{\infty}$, $\{\bar{Z}^{+}_{i,n}\}_{n=0}^{\infty}$, $\{\bar{Z}^{+}_{i,-n}\}_{n=0}^{\infty}$, meet the requirements of \textbf{H2} and \textbf{H3} for each $i$, besides the terms of these sequences are positive for $n\geq 1$ and don't vanish if $n=0$.

Now, coming to equation \eqref{6.11} and keeping in mind factorization \eqref{6.12}, we employ Theorems \ref{t3.1} and \ref{t3.2} to build  the explicit solutions homogeneous and inhomogeneous equation \eqref{6.11}.

At this point, we first find a solution of homogeneous (6.11) ($\mathcal{F}\equiv 0$).  To this end, putting
\[
\d_{0}(\mathfrak{K}):=\frac{\mathcal{G}(0)\mathfrak{s}\prod\limits_{i=1}^{\bar{\K}-1}Z_{i,0}Z_{i,-0}}
{\prod\limits_{i=0}^{\bar{\K}^{+}-1}\bar{Z}_{i,0}^{+}},\,\,
\mathbb{L}_{0}(\rho,\mathfrak{K},\mathfrak{s}^{\star})=
\begin{cases}
\prod\limits_{i=0}^{\bar{\K}^{+}-1}\Gamma(\bar{Z}^{+}_{i,0}+\rho)\prod\limits_{i=1}^{\bar{K}-1}
\frac{\Gamma(1+Z_{i,-0}-\rho)}{\Gamma(\rho+Z_{i,0})},\quad\text{if}\quad \mathfrak{s}^{\star}>0,\\
\frac{1}{\prod\limits_{i=0}^{\bar{\K}^{+}-1}\Gamma(1+\bar{Z}^{+}_{i,0}-\rho)}\prod\limits_{i=1}^{\bar{K}-1}\frac{\Gamma(1+Z_{i,0}-\rho)}{\Gamma(\rho+Z_{i,-0})},\quad\,\text{if}\quad \mathfrak{s}^{\star}<0,
\end{cases}
\]
\[
\mathcal{R}(n,\mathfrak{K},\mathfrak{s}^{\star})=
\begin{cases}
\sum\limits_{i=0}^{\bar{\K}^{+}-1}[\bar{Z}_{i,-n}^{+}(\ln\bar{Z}_{i,-n}^{+}-1)-\bar{Z}_{i,n}^{+}(\ln\bar{Z}_{i,n}^{+}-1)
]\\
-\sum\limits_{i=0}^{\bar{\K}-1}[Z_{i,-n}(\ln Z_{i,-n}-1)-Z_{i,n}(\ln Z_{i,n}-1)
],\quad \text{ if }\quad\mathfrak{s}^{\star}>0,
\\
\sum\limits_{i=0}^{\bar{\K}^{+}-1}[\bar{Z}_{i,-n}^{+}(\ln\bar{Z}_{i,-n}^{+}-1)-\bar{Z}_{i,n}^{+}(\ln\bar{Z}_{i,n}^{+}-1)
]\\
+\sum\limits_{i=0}^{\bar{\K}-1}[Z_{i,-n}(\ln Z_{i,-n}-1)-Z_{i,n}(\ln Z_{i,n}-1)
],\quad \text{ if }\quad\mathfrak{s}^{\star}<0,
\end{cases}
\]
and
\[
\mathbb{L}_{1}(\rho,\mathfrak{K},\mathfrak{s}^{\star})=
\begin{cases}
\prod\limits_{n=1}^{+\infty}\prod\limits_{i=0}^{\bar{\K}^{+}-1}\frac{\Gamma(\bar{Z}^{+}_{i,n}+\rho)}{\Gamma(\bar{Z}^{+}_{i,-n}+1-\rho)}
\prod\limits_{i=0}^{\bar{\K}-1}\frac{\Gamma(Z_{i,-n}-\rho+1)}{\Gamma(Z_{i,n}+\rho)}
\bigg(\frac{\prod\limits_{i=0}^{\bar{\K}-1}Z_{i,n}Z_{i,-n}}{\prod\limits_{i=0}^{\bar{\K}^{+}-1}\bar{Z}^{+}_{i,n}\bar{Z}^{+}_{i,-n}}\bigg)^{\rho-\frac{1}{2}}e^{\mathcal{R}(\rho,\mathfrak{K},\,\mathfrak{s}^{\star})}\quad\text{if}\quad \mathfrak{s}^{\star}>0,\\
\prod\limits_{n=1}^{+\infty}\prod\limits_{i=0}^{\bar{\K}^{+}-1}\frac{\Gamma(\bar{Z}^{+}_{i,-n}+\rho)}{\Gamma(\bar{Z}^{+}_{i,n}+1-\rho)}
\prod\limits_{i=0}^{\bar{\K}-1}\frac{\Gamma(Z_{i,n}-\rho+1)}{\Gamma(Z_{i,-n}+\rho)}
\bigg(\frac{\prod\limits_{i=0}^{\bar{\K}-1}Z_{i,n}Z_{i,-n}}{\prod\limits_{i=0}^{\bar{\K}^{+}-1}\bar{Z}^{+}_{i,n}\bar{Z}^{+}_{i,-n}}\bigg)^{\rho-\frac{1}{2}}e^{\mathcal{R}(\rho,\mathfrak{K},\,\mathfrak{s}^{\star})}\quad\text{if}\quad \mathfrak{s}^{\star}<0,
\end{cases}
\]
and exploiting  Theorem \ref{t3.1}, we end up with the explicit form of the solution to homogeneous equation \eqref{6.11}:
\begin{equation}\label{6.13}
V_{h}(\rho,\sigma;\P)=(\d_0(\mathfrak{K}))^{\rho-1/2}\P(\rho)(a_1\sigma+a_2\sigma^{\nu})^{\frac{1}{2}-\rho}\mathbb{L}_{0}(\rho,\mathfrak{K},\mathfrak{s}^{\star})\mathbb{L}_{0}(\rho,\mathfrak{K},\mathfrak{s}^{\star})
\end{equation}
where, for $m\in\N\cup\{0\}$ and $n\in\N,$ $Re\, \rho$ satisfies the following inequalities (similar to \eqref{2.3} and \eqref{2.4}) :
\begin{equation}\label{6.14}
Re\rho\neq -m-\bar{Z}_{i,n}^{+},\, i\in\{0,...,\bar{\K}^{+}-1\},\,
Re\rho\neq m+1+Z_{j,1-n},\, j\in\{0,...,\bar{\K}-1\},\,
Re\rho\neq 1+m+Z_{0,-n}
\end{equation}
if $\mathfrak{s}^{\star}>0$, while
\begin{equation}\label{6.15}
Re\rho\neq -m-\bar{Z}_{i,-n}^{+},\, i\in\{0,...,\bar{\K}^{+}-1\},\,
Re\rho\neq m+1+Z_{j,n-1},\, j\in\{0,...,\bar{\K}-1\},\,
Re\rho\neq 1+m+Z_{0,n}
\end{equation}
if $\mathfrak{s}^{\star}<0$.

Besides, the following statement is simple consequence of Theorem \ref{t3.1}.
\begin{proposition}\label{p6.1}
Let $Re\, \rho$ meet requirement \eqref{6.14} in the case of $\mathfrak{s}^{\star}>0$ and \eqref{6.15} if $\mathfrak{s}^{\star}<0$. Then the function $\frac{V_{h}(\rho,\sigma;\P)}{\P(\rho)}$ has no poles and the infinite product $\mathbb{L}_1(\rho,\mathfrak{K},\mathfrak{s}^{\star})$ converges for each complex $\rho$. Besides, the relations hold:
\begin{align*}
&(\d_0(\mathfrak{K}))^{\rho-1/2}\mathbb{L}_{0}(\rho,\mathfrak{K},\mathfrak{s}^{\star})\mathbb{L}_{0}(\rho,\mathfrak{K},\mathfrak{s}^{\star})
=[(\mathfrak{s}-\mathfrak{s}^{\star}\rho)\mathcal{G}(i\mathfrak{s}^{\star}\rho)]^{\rho-1/2}\exp\{C_1\ln\rho+C_2\rho+O(1)\},\\
&|V_h(\rho,\sigma;\P)\P^{-1}(\rho)|\leq C|\rho|^{C_1+Re\rho-1/2}\exp\{\arg(a_1\sigma+a_2\sigma^{\nu})Im\rho+[-\frac{\pi}{2}\text{sgn}\,(Re\rho)+\theta_2(\text{sgn}\,\mathfrak{s}^{\star})]|Im\rho|\},
\end{align*}
as $|Im\rho|\to+\infty$ and $Re\rho$ remains bounded.
\end{proposition}

Since we look for a bounded solution at the infinity, Remark \ref{r3.1} and Theorem \ref{t3.2} tell us that we are left to construct the particular solution in form \eqref{2.11}. To this end, we first put the function $\P(\rho):=\P_1(\rho)$ where
\begin{equation}\label{6.16}
\P_1(\rho)=
\begin{cases}
\prod\limits_{i=1}^{5}\sin\pi(1+Z_{i,0}-\rho)\prod\limits_{i=0}^{2}\sin^{-1}\pi(1+\bar{Z}_{i,0}^{+}-\rho),\quad\text{if}\, \mathfrak{s}^{\star}<0,\\
\prod\limits_{i=0}^{4}\sin\pi(\bar{Z}_{i,0}^{+}+\rho)\prod\limits_{i=1}^{3}\sin^{-1}\pi(Z_{i,0}+\rho),\qquad\qquad\,\text{ if }\mathfrak{s}^{\star}>0.
\end{cases}
\end{equation}
It is apparent that the introduced function $\P_1$ obeys the properties:
\begin{equation}\label{6.17*}
\P_1(\rho)\approx Ce^{2\pi|Im\rho|},\quad\text{as}\quad |Im\rho|\to+\infty,
\end{equation}
besides the functions:
$
\frac{\P_1(\rho) \prod\limits_{i=1}^{5}\Gamma(1+Z_{i,0}-\rho)}{\prod\limits_{i=0}^{2}\Gamma(1+\bar{Z}_{i,0}^{+}-\rho)}\,$  if  $\mathfrak{s}^{\star}<0\,$  and $ \frac{\P_1(\rho)\prod\limits_{i=0}^{4}\Gamma(\bar{Z}_{i,0}^{+}+\rho)}{\prod\limits_{i=1}^{3}\Gamma(Z_{i,0}+\rho)}\,$ if $\mathfrak{s}^{\star}>0$
do not have any poles if for $l_1,l_2\in\N$, the inequalities hold
\begin{equation}\label{6.17}
Re\rho\neq l_2-Z_{i,0},\, i\in\{1,2,3\}, \text{ if } \mathfrak{s}^{\star}>0,\, \text{and }
Re\rho\neq 1+\bar{Z}^{+}_{j,0}-l_1,\, j\in\{0,1,2\}, \text{ if } \mathfrak{s}^{\star}<0,
\end{equation}
and do not have any nulls if
\begin{equation}\label{6.18}
Re\rho\neq l_2-\bar{Z}^{+}_{i,0},\, i\in\{0,1,2,3,4,\} \text{ if } \mathfrak{s}^{\star}>0,\, \text{and }
Re\rho\neq 1+Z_{j,0}-l_1,\, j\in\{1,2,3,4,5\} \text{ if } \mathfrak{s}^{\star}<0.
\end{equation}
Collecting Proposition \ref{p6.1} with \eqref{6.14}-\eqref{6.18} and setting
$$
Z^{*}=\begin{cases}
1+Z_{6,-0},\quad\text{if}\quad \bar{\K}>6,\\
4\pi q,\quad\qquad\text{ if}\quad \bar{\K}\leq 6,
\end{cases}
$$
 we conclude.
\begin{proposition}\label{p6.2}
The function $V_h(\rho,\sigma):=V_{h}(\rho,\sigma;\P_1)$ given with \eqref{6.13} and \eqref{6.16} has the following properties
\begin{description}
    \item[i] $V_h(\rho,\sigma)$ does not have any poles if in the case of $\mathfrak{s}^{\star}<0$
 the inequalities hold
\begin{equation}\label{6.19*}
-\bar{Z}_{0,-1}^{+}<Re\rho<Z^{*}\,\text{and }
Re\rho\neq 1-\bar{m}(i)+\bar{Z}^{+}_{i,0},\quad i\in\{0,1,2\},
\end{equation}
where the integer $\bar{m}(i)$ satisfies inequalities
$$\max\{1;\bar{Z}_{i,0}^{+}-Z^{*}+\}<\bar{m}(i)<1+\bar{Z}_{i,0}^{+}+\bar{Z}_{0,-1}^{+},$$
while in the case of $\mathfrak{s}^{\star}>0$ there hold
\begin{equation}\label{6.19}
-\bar{Z}_{5,0}^{+}<Re\rho<1+Z_{1,-0},\,
Re\rho\neq \bar{m}(i)-Z_{i,0},\quad i\in\{1,2,3\},
\end{equation}
where the integer $\bar{m}(i)$ satisfies relations
$$\max\{1;Z_{i,0}-\bar{Z}^{+}_{5,0}\}<\bar{m}(i)<1+Z_{i,0}+Z_{1,-0};$$
\item[ii] $V_h(\rho,\sigma)$ does not vanish if in the case of $\mathfrak{s}^{\star}<0$ the inequalities hold:
\begin{equation}\label{6.20}
-{Z}_{1,-0}<Re\rho<1+\bar{Z}^{+}_{3,0},\text{ and }\,
Re\rho\neq 1-\underline{m}(i)+Z_{i,0},\quad i\in\{1,2,3,4,5\},
\end{equation}
with the integer $\underline{m}(i)$ satisfying bounds:
$
\max\{1;Z_{i,0}-\bar{Z}^{+}_{3,0}\}<\underline{m}(i)<1+Z_{i,0}+Z_{1,-0},
$
while in the case of $\mathfrak{s}^{\star}>0$ there are inequalities
\begin{equation}\label{6.21}
-Z_{4,0}<Re\rho<1+\bar{Z}^{+}_{0,-1},\text{ and }\,
Re\rho\neq \underline{m}(i)-\bar{Z}^{+}_{i,0},\quad i\in\{0,1,2,3,4\},
\end{equation}
with the integer $\underline{m}(i)$ satisfying bounds: $\max\{1;\bar{Z}^{+}_{i,0}-Z_{4,0}\}<\underline{m}(i)<1+\bar{Z}^{+}_{i,0}
+\bar{Z}^{+}_{0,-1}.
$
\item[iii] If $Re\rho$ meets requirements  \eqref{6.20} and \eqref{6.21}, then  the estimate holds
\[
|V_h(\rho,\sigma)|^{-1}\leq C\exp\{-\Big[\frac{\pi}{2}-\varepsilon+\theta_2(sgn\mathfrak{s}^{\star})\Big]|Im\rho|\},\quad\text{as}\quad |Im\rho|\to+\infty,
\]
where $\varepsilon$ is a sufficiently small positive number, $0<\varepsilon<\theta/10$.
\end{description}
\end{proposition}
Taking into account Propositions \ref{p6.1} and \ref{p6.2} and Theorem \ref{t3.2}, and choosing $\mathcal{K}_1(\xi)=\frac{\sin^{2}\pi d}{\sin^{2}\pi(\xi+d)},$ $0<d-d_0<1$ (see Table \ref{tab:table5} in Example \ref{e5.4}), we construct the particular solution of equation \eqref{6.11}.
\begin{proposition}\label{p6.3}
Let assumptions \textbf{H14} and \textbf{H15} hold. Moreover, we assume that
\[
-Z_{1,-0}<Re\rho<\bar{Z}_{3,0}^{+},\text{ and }\,
Re\rho\neq d_0+Z_{i,-0}-\underline{m}(i),\quad i\in\{1,2,3,4,5\},
\]
if $\mathfrak{s}^{\star}<0$, and in the case of $\mathfrak{s}^{\star}>0$, there are
\[
-Z_{4,0}<Re\rho<\bar{Z}_{0,-1}^{+},\text{ and }\,
Re\rho\neq d_0-Z_{i,0}-1+\underline{m}(i),\quad i\in\{0,1,2,3,4\},
\]
with $\underline{m}(i)$ meets the requirements of \textbf{(ii)} in Proposition \ref{p6.2}.
Then the function
$$
V(\rho,\sigma)=\frac{1}{2 i}\int\limits_{\ell_{d_0}}\frac{V_{h}(\rho,\sigma;\P_1)f^{*}(i\mathfrak{s}^{\star}\rho+i\mathfrak{s}^{\star}\xi+i(1-\mathfrak{s}),\sigma)}{V_{h}(\rho+1+\xi,\sigma;\P_1)(a_1\sigma+a_2\sigma^{\nu})}[\cot\pi\xi+i]\frac{\sin^{2}\pi d}{\sin^{2}\pi(\xi+d)}d\xi
$$
solves equation \eqref{6.11}.
Besides, the function $\frac{V_h(\rho,\sigma;\P_1)}{V_h(\rho+1+\xi,\sigma;\P_1)}$ is analytic in $\rho$ if $\xi=-d_0+iy,$ $y\in\R$, $d_0\in[0,1],$ and $Re\rho$ meets the requirements
\[
-Z_{1,-0}<Re\rho<\bar{Z}^{+}_{3,0},\, \,
Re\rho\neq d_0+Z_{i,-0}-\underline{m}(i),\,\, i\in\{1,2,3,4,5\},\,\,
Re\rho\neq 1-\bar{m}(i)+\bar{Z}_{j,0}^{+},\,\,d j\in\{0,1,2\},
\]
in the case of $\mathfrak{s}^{\star}<0,$ and
\[
-Z_{4,0}<Re\rho<1+Z_{1,-0},\,\,
Re\rho\neq d_0-Z_{i,0}-1+\underline{m}(i),\,\, i\in\{0,1,2,3,4\},\,\,
Re\rho\neq \bar{m}(i)-Z_{j,0},\,\, j\in\{1,2,3\},
\]
if $\mathfrak{s}^{\star}>0.$
Integer $\underline{m}(i)$ and $\bar{m}(i)$ meet requirement of Proposition \ref{p6.2}.
\end{proposition}

\noindent\textit{Third Stage.} Denoting
\[
\mathcal{N}_2(\mathfrak{x})=[\cot\mathfrak{x}(\omega_0-\pi/2)+a_4][\cot\mathfrak{x}(\omega_0+\pi/2)+\mathfrak{K}a_4]^{-1},
\]
and collecting Proposition \ref{p6.3} with relations \eqref{6.9*} and \eqref{6.10}, \eqref{6.11*}, we end up with  the solution
\begin{align*}
U_1^{*}(\lambda,x_2,\sigma)&=\frac{V_h(-i\lambda/\mathfrak{s}^{\star},\sigma;\P_1)}{2i}\frac{\mathfrak{K}\mathcal{N}_2(i\lambda+\mathfrak{s})\sin(i\lambda+\mathfrak{s})(x_2+\pi/2)}
{[\mathfrak{K}\mathcal{N}_2(i\lambda+\mathfrak{s})+1]\sin(i\lambda+\mathfrak{s})(\omega_0+\pi/2)}\\ \notag
&\times
\int\limits_{\ell_{d_0}}
\frac{f^{*}(\lambda+i\mathfrak{s}^{\star}\xi+i(1-\mathfrak{s}),\sigma)[\cot\pi\xi+i]}
{(a_1\sigma+a_2\sigma^{\nu})V_h(1+\xi-i\lambda/\mathfrak{s}^{\star},\sigma;\P_1)}
\frac{\sin^{2}\pi d}{\sin^{2}\pi(\xi+d)}d\xi,
\end{align*}
\begin{align}\label{6.20*}
U_2^{*}(\lambda,x_2,\sigma)&=\frac{V_h(-i\lambda/\mathfrak{s}^{\star},\sigma;\P_1)}{2i}\frac{\sin(i\lambda+\mathfrak{s})(x_2-\pi/2)}
{[\mathfrak{K}\mathcal{N}_2(i\lambda+\mathfrak{s})+1]\sin(i\lambda+\mathfrak{s})(\omega_0-\pi/2)}\\ \notag
&\times
\int\limits_{\ell_{d_0}}
\frac{f^{*}(\lambda+i\mathfrak{s}^{\star}\xi+i(1-\mathfrak{s}),\sigma)[\cot\pi\xi+i]}
{(a_1\sigma+a_2\sigma^{\nu})V_h(1+\xi-i\lambda/\mathfrak{s}^{\star},\sigma;\P_1)}
\frac{\sin^{2}\pi d}{\sin^{2}\pi(\xi+d)}d\xi.
\end{align}
Finally, we are left to compute inverse Laplace and Fourier transformations to obtain the integral representation of the solution \eqref{6.1}-\eqref{6.5}. To this end, we appeal to (5.1.33) in \cite{GKMR} and deduce
\[
\mathcal{L}(t,y):=\frac{1}{2\pi i}\int\limits_{-i\infty}^{+i \infty}\frac{e^{\sigma t}d\sigma}{(a_1\sigma+a_2\sigma^{\nu})^{d_0-iy}}
=\begin{cases}
\frac{t^{-1+\frac{a_1+a_2\nu}{a_1+a_2}[d_0-iy]}}{\Gamma\big(\frac{a_1+a_2\nu}{a_1+a_2}[d_0-iy]\big)(a_2+a_1)^{d_0-iy}},\quad\quad \text{if}\, \text{ either } a_1=0 \text{ or } a_2=0,\\
\,\\
a_1^{-d_0+iy}t^{-1+d_0-iy}E_{1-\nu,d_0-iy}^{d_0-iy}(-\frac{a_2}{a_1}t^{1-\nu}),\quad\text{if}\quad a_1,a_2>0,
\end{cases}
\]
where $E_{\alpha,\beta}^{\gamma}(\cdot)$ is the three-parametric Mittag-Leffler function (see, e.g. (5.1.4) in \cite{GKMR}).

\noindent Then, collecting these relations with  \eqref{6.20*},  we arrive at the explicit solution to \eqref{6.1}-\eqref{6.5}.
\begin{proposition}\label{p6.4}
Let assumptions of Propositions \ref{p6.1}-\ref{p6.3} and restriction \eqref{6.0*}  hold, then the following functions solves problem \eqref{6.1}-\eqref{6.5}:
\begin{align*}
U_1(x_1,x_2,t)&=\int\limits_{0}^{t}d\tau\int\limits_{-\infty}^{+\infty}d\varrho f(x_1-\varrho,t-\tau)e^{[-\mathfrak{s}^{\star}d_0+1-\mathfrak{s}](x_1-\varrho)}\int\limits_{-\infty}^{+\infty}dy\frac{\mathcal{L}(t,y)e^{iy\mathfrak{s}^{\star}(x_1-\varrho)}\sin^{2}\pi d}
{\sin\pi(d_0-iy)\sin^{2}\pi(d-d_0+iy)}\\
&
\times \int\limits_{-\infty}^{+\infty}\frac{e^{i\lambda\varrho}V_h(-i\lambda/\mathfrak{s}^{\star},\sigma;\P_1)\mathfrak{K}\mathcal{N}_2(i\lambda+\mathfrak{s})\sin(i\lambda+\mathfrak{s})(x_2+\pi/2)}
{V_h(1-d_0+iy-i\lambda/\mathfrak{s}^{\star},\sigma;\P_1)[1+\mathfrak{K}\mathcal{N}_2(i\lambda+\mathfrak{s})]\sin(i\lambda+\mathfrak{s})(\omega_0+\pi/2)}d\lambda;
\end{align*}
\begin{align*}
U_2(x_1,x_2,t)&=\int\limits_{0}^{t}d\tau\int\limits_{-\infty}^{+\infty}d\varrho f(x_1-\varrho,t-\tau)e^{[-\mathfrak{s}^{\star}d_0+1-\mathfrak{s}](x_1-\varrho)}\int\limits_{-\infty}^{+\infty}dy\frac{\mathcal{L}(t,y)e^{iy\mathfrak{s}^{\star}(x_1-\varrho)}\sin^{2}\pi d}
{\sin\pi(d_0-iy)\sin^{2}\pi(d-d_0+iy)}\\
&
\times \int\limits_{-\infty}^{+\infty}\frac{e^{i\lambda\varrho}V_h(-i\lambda/\mathfrak{s}^{\star},\sigma;\P_1)\sin(i\lambda+\mathfrak{s})(x_2-\pi/2)}
{V_h(1-d_0+iy-i\lambda/\mathfrak{s}^{\star},\sigma;\P_1)[1+\mathfrak{K}\mathcal{N}_2(i\lambda+\mathfrak{s})]\sin(i\lambda+\mathfrak{s})(\omega_0-\pi/2)}d\lambda,
\end{align*}
where $Im \lambda$ meets requirements:
\[
-Z_{1,-0}<\frac{Im\lambda}{-\mathfrak{s}^{\star}}<\bar{Z}^{+}_{3,0},\,\,
\frac{Im\lambda}{\mathfrak{s}^{\star}}\neq d_0+Z_{i,-0}-\underline{m}(i),\, i\in\{1,2,3,4,5\},\,\,
\frac{Im\lambda}{\mathfrak{s}^{\star}}\neq 1-\bar{m}(j)+\bar{Z}_{j,0}^{+},\, j\in\{0,1,2\},
\]
in the case of $\mathfrak{s}^{\star}<0$, while if $\mathfrak{s}^{\star}>0$ then the inequalities hold
\[
-Z_{4,0}<\frac{Im\lambda}{\mathfrak{s}^{\star}}<1+Z_{1,-0},\,\,
\frac{Im\lambda}{\mathfrak{s}^{\star}}\neq d_0-1-Z_{i,0}+\underline{m}(i),\, i\in\{0,1,2,3,4\},\,\,
\frac{Im\lambda}{\mathfrak{s}^{\star}}\neq \bar{m}(j)-Z_{j,0},\, j\in\{1,2,3\}
\]
with integer $\underline{m}(i)$ and $\bar{m}(i)$ satisfying relations in \textbf{(i),(ii)} of Proposition \ref{p6.2}.
\end{proposition}
\noindent
\begin{remark}\label{r6.2}
It is worth noting that  relations on $Im\lambda$ in Proposition \ref{p6.4} provide additional assumptions on the weight $\mathfrak{s}$. Indeed,  putting $Im\lambda=0$ in the relations above, we  arrive at the conditions:

\noindent$\bullet$ if $\mathfrak{s}^{\star}<0$, then
$\mathfrak{s}\neq \frac{\bar{z}_{i}^{+}}{2\omega_0}-|\mathfrak{s}^{\star}|[\bar{m}(i)-1],\, i\in\{0,1,2\},$ and for $j\in\{0,1,...,5\}$
\[
-\frac{z_{1}^{+}}{2\omega_0}<\mathfrak{s}<\frac{\bar{z}_{3}^{+}}{2\omega_0},\quad
\mathfrak{s}\neq |\mathfrak{s}^{\star}|[\underline{m}(j)-d_0] -\frac{z_{j}^{+}}{2\omega_0},\, \,\text{ if }\, \mathfrak{K}>1,
\]
\[
-\frac{z_{1}^{-}}{2\omega_0}<\mathfrak{s}<\frac{\bar{z}_{3}^{+}}{2\omega_0},\quad
\mathfrak{s}\neq |\mathfrak{s}^{\star}|[\underline{m}(j)-d_0] -\frac{z_{j}^{-}}{2\omega_0},\, \,\text{ if }\, \mathfrak{K}<1;
\]
\noindent$\bullet$ if $\mathfrak{s}^{\star}>0$, then
\[
-\frac{z_1^{+}}{2\omega_0}-1<\mathfrak{s}<\frac{z_4^{+}}{2\omega_0},\mathfrak{s}\neq \frac{z_i^{+}}{2\omega_0}+|\mathfrak{s}^{\star}|[1-\underline{m}(i)-d_0],\, i\in\{0,...,4\}, \mathfrak{s}\neq -|\mathfrak{s}^{\star}|\bar{m}(j)+\frac{z_j^{+}}{2\omega_0}, j\in\{1,2,3\},\text{if }\mathfrak{K}>1,
\]
while if $\mathfrak{K}<1$ then
\[
-\frac{z_1^{-}}{2\omega_0}-1<\mathfrak{s}<\frac{z_4^{-}}{2\omega_0},\mathfrak{s}\neq \frac{z_i^{-}}{2\omega_0}+|\mathfrak{s}^{\star}|[1-\underline{m}(i)-d_0],\, i\in\{0,...,4\}, \mathfrak{s}\neq -|\mathfrak{s}^{\star}|\bar{m}(j)+\frac{z_j^{-}}{2\omega_0}, j\in\{1,2,3\}.
\]
It is apparent that the selected weight $\mathfrak{s}$ in accordance with  these assumptions will satisfy  \textbf{H15}.
\end{remark}
We conclude this section with some comments related to  other boundary problems with dynamic boundary conditions in plane corners.

\begin{remark}\label{r6.4}
Actually, with nonessential modifications in the arguments in Section \ref{s6}, the very same results hold for transmission problems \eqref{6.1}-\eqref{6.4} supplemented with Neumann boundary conditions on $\partial G_{i}\backslash g,$ $i=1,2.$ Besides, the proposed  approach  in Section \ref{s6} can be incorporated to find explicit solutions of the non-classical Dirichlet or Neumann  boundary value problem with  fractional dynamic boundary conditions in plane corners.
\end{remark}


\section*{Appendix: Proof of Proposition \ref{p3.1}}
\label{s7}

\theoremstyle{definition}
\newtheorem{reAPP}{Remark}[section]
\renewcommand{\thereAPP}{A.\arabic{reAPP}}
\setcounter{equation}{0}
\setcounter{subsection}{0}
\renewcommand{\theequation}{A.\arabic{equation}}

\noindent We start with validation of the first estimate in Proposition \ref{p3.1}. To this end, it is enough to evaluate the first term in the left-hand side of the inequality (i). The second series will be examined with the same arguments.
The simple technical calculations and properties of the sequence of $\b(n)$ provide the relations
\begin{align}\label{A.1}\notag
&\sum_{n=\mathfrak{N}\,+1}^{+\infty}\frac{z}{\b(n)(\b(n)+z)}=\sum_{n=\mathfrak{N}\,+1}^{+\infty}\frac{z_1(\b(n)+z_1)}{\b(n)[(\b(n)+z_1)^{2}+z_2^{2}]}+
i\sum_{n=\mathfrak{N}\,+1}^{+\infty}\frac{z_2}{[(\b(n)+z_1)^{2}+z_2^{2}]}
\\
&+
\sum_{n=\mathfrak{N}\,+1}^{+\infty}\frac{z_2^{2}}{\b(n)[(\b(n)+z_1)^{2}+z_2^{2}]}
=C+O(1/|z_{2}|)+\sum_{n=\mathfrak{N}+1}^{+\infty}\frac{z_2^{2}}{\b(n)[(\b(n)+z_1)^{2}+z_2^{2}]}.
\end{align}
To handle the last term in the right-hand side of this equality,  we take advantage of the easy verified inequalities
\[
\Big|\sum_{n=\mathfrak{N}\,+1}^{+\infty}\frac{z_2^{2}}{\b(n)[(\b(n)+z_1)^{2}+z_2^{2}]}\Big|\leq \int_{\mathfrak{N}\,+1}^{+\infty}\frac{C z_{2}^{2} dx}{\b(x)[z_2^{2}+\b^{2}(x)]}\leq\frac{z_2^{2}}{\delta_{0}}\int_{\b(\mathfrak{N}\,+1)}^{+\infty}\frac{Cdy}{y[z_{2}^{2}+y^{2}]}
\]
with $\delta_0>0$. Here we used that the sequence $\{\b(n)\}_{n=1}^{+\infty}$ is the strictly increasing, i.e.
\begin{equation}\label{A.2}
\b'(x)>\delta_0>0,\quad \forall x\geq \mathfrak{N}+1.
\end{equation}
Performing the change of variable $(z^{2}_{2}y^{-2})=\mathfrak{p}$, we obtain
\[
\Big|\sum_{n=\mathfrak{N}\, +1}^{+\infty}\frac{z_2^{2}}{\b(n)[(\b(n)+z_1)^{2}+z_2^{2}]}\Big|\leq
\frac{C}{\delta_{0}}\int^{z_{2}^{2}\b^{-2}(\mathfrak{N}\,+1)}_{0}\frac{d\mathfrak{p}}{1+\mathfrak{p}^{2}}=C\ln[1+z_{2}^{2}\b^{-2}(\mathfrak{N}\, +1)].
\]
Collecting this inequality with \eqref{A.1}, we arrive at the first estimate in Proposition \ref{p3.1}.

\noindent$\bullet$ Concerning the second estimate in Proposition \ref{p3.1}, we can conclude that  second and third terms in the left-hand side of this inequality are evaluated exactly like the first one. Hence, we focus here only on the proof of the equality
\[
\sum_{n=\mathfrak{N}\, +1}^{+\infty}\Big|\frac{z}{(\b(n)+z)(\b(n)+C^{\star}+z)}\Big|=C+O(1/|z|),
\]
if $|z_{2}|\to+\infty$ and $z_1$ meets the requirements of Proposition \ref{p3.1}.

\noindent
Performing the simple calculations, we reach to the representation
\begin{equation}\label{A.3}
\sum_{n=\mathfrak{N}\,+1}^{+\infty}\frac{z}{(\b(n)+z)(\b(n)+C^{\star}+z)}=z\sum_{j=1}^{4}\mathfrak{B}_{j}(z),
\end{equation}
where we put
\begin{align*}
\mathfrak{B}_{1}(z)&=\sum_{n=\mathfrak{N}\,+1}\frac{C^{\star}(\b(n)+z_1)}{[(\b(n)+z_1)^{2}+z_2^{2}][(\b(n)+C^{\star}+z_1)^{2}+z_2^{2}]};\\
\mathfrak{B}_{2}(z)&=\sum_{n=\mathfrak{N}\,+1}\frac{(\b(n)+z_1)^{2}}{[(\b(n)+z_1)^{2}+z_2^{2}][(\b(n)+C^{\star}+z_1)^{2}+z_2^{2}]};\\
\mathfrak{B}_{3}(z)&=-\sum_{n=\mathfrak{N}\,+1}\frac{z_{2}^{2}}{[(\b(n)+z_1)^{2}+z_2^{2}][(\b(n)+C^{\star}+z_1)^{2}+z_2^{2}]};\\
\mathfrak{B}_{4}(z)&=-i\sum_{n=\mathfrak{N}\,+1}\frac{z_{2}(2\b(n)+2z_1+C^{\star})}{[(\b(n)+z_1)^{2}+z_2^{2}][(\b(n)+C^{\star}+z_1)^{2}+z_2^{2}]}.
\end{align*}
Next, we exploit the inequality \eqref{A.2} to achieve  the  relations
\[
\sum_{n=\mathfrak{N}\,+1}^{+\infty}\frac{1}{\b^{2}(n)+z_2^{2}}\leq \int_{\mathfrak{N}\,+1}^{+\infty}\frac{dx}{\b^{2}(x)+z_{2}^{2}}\leq \frac{1}{\delta_{0}}
\int_{\b(\mathfrak{N}\,+1)}^{+\infty}\frac{dy}{y^{2}+z_{2}^{2}}= \frac{C}{|z_{2}|},
\]
which in turn arrive at the bound
\[
|z_2||\mathfrak{B}_1(z)|+|\mathfrak{B}_{2}(z)|+|\mathfrak{B}_{3}(z)|+|\mathfrak{B}_4(z)|
\leq\sum_{n=\mathfrak{N}\,+1}^{+\infty}\frac{C}{(\b(n)+C^{\star}+z_1)^{2}+z_2^{2}}\leq\frac{C}{z_{2}^{2}}.
\]
Collecting these estimates with representation \eqref{A.3}, we end up with the desired bound.

\noindent$\bullet$ As for the third estimate in Proposition \ref{p3.1}, to verify this statement  is enough to obtain the inequality
\begin{equation}\label{A.4}
\sum_{n=\mathfrak{N}+1}\Big|\b(n)\frac{z^{3}}{(\b(n)+z)^{3}}\sum_{j=1}^{+\infty}\frac{1}{(j+3)}\frac{z^{j}}{(\b(n)+z)^{j}}\Big|
\leq C|z^{2}|[\ln|z|+1]+O(1),\quad\text{as } |z_{2}|\to+\infty
\end{equation}
and $z_1$ satisfying the requirements of Proposition \ref{p3.1}.

It is apparent that assumptions on the sequence $\b(n)$ and $z_1$ provide the bound
\[
\Big|\frac{z}{\b(n)+z}\Big|<1\quad \text{for each } n\geq \mathfrak{N}+1\quad\text{and }|z_2|\to+\infty.
\]
In light of the last inequality, we have
\begin{align*}
&\frac{1}{|\b(n)+z|}\sum_{j=1}^{+\infty}\frac{1}{(j+3)}\Big|\frac{z}{(\b(n)+z)}\Big|^{j}\leq
\frac{1}{|\b(n)+z|}\sum_{j=1}^{+\infty}\Big|\frac{z}{(\b(n)+z)}\Big|^{j}=
\frac{1}{|\b(n)+z|-|z|}\\
&=\frac{\sqrt{(\b(n)+z_1)^{2}+z_{2}^{2}}+\sqrt{z_{1}^{2}+z_{2}^{2}}}
{\b(n)(\b(n)+2z_{1})}
\leq C\sum_{k=1}^{3}\mathfrak{A}_{k}(z),
\end{align*}
where we set
\[
\mathfrak{A}_{1}(z)=\frac{|z_1|}{\b(n)|\b(n)+2z_1|},\,
\mathfrak{A}_{2}(z)=\frac{|z_2|}{\b(n)|\b(n)+2z_1|},\,
\mathfrak{A}_{3}(z)=\frac{|\b(n)+z_1|}{\b(n)|\b(n)+2z_1|}.
\]
Summarizing, we are left to produce the suitable bound of
$\sum_{n=\mathfrak{N}+1}^{+\infty}|z|\b(n)\Big|\frac{z}{\b(n)+z}\Big|^{2}\sum_{k=1}^{3}\mathfrak{A}_{k}(z).$
 To this end, we appeal to estimates in points (i) and (ii) in Proposition \ref{p3.1} and arrive at the relations
\begin{align*}
&\sum_{n=\mathfrak{N}\,+1}^{+\infty}\b(n)\Big|\frac{z}{\b(n)+z}\Big|^{2}[|z|\mathfrak{A}_{1}(z)+\mathfrak{A}_{2}(z)]\leq C|z|\ln|z|+O(1),\\
&\sum_{n=\mathfrak{N}\,+1}^{+\infty}|z|\b(n)\Big|\frac{z}{\b(n)+z}\Big|^{2}\mathfrak{A}_{3}(z)\leq C|z|[|z|+1].
\end{align*}
in the case of  $z_{2}\to\pm\infty$.
Collecting these inequalities, we deduce  bound \eqref{A.4} and, accordingly, the desired estimate in (iii).

\noindent$\bullet$ At this point, we will examine the first equality in (iv). Note that the second one is verified with the same arguments.
First, we use the easily verified relations:
\begin{align*}
\sum_{j=1}^{+\infty}\frac{1}{(j+3)}\Big(\frac{z}{\b(n)+z}\Big)^{j}&=\sum_{j=1}^{+\infty}\frac{-3}{j(j+3)}\Big(\frac{z}{\b(n)+z}\Big)^{j}+
\sum_{j=1}^{+\infty}\frac{1}{j}\Big(\frac{z}{\b(n)+z}\Big)^{j},\\
\sum_{j=1}^{+\infty}\frac{1}{(j+3)}\Big(\frac{z}{\b(n)+C^{\star}+z}\Big)^{j}&=\sum_{j=1}^{+\infty}\frac{-3}{j(j+3)}\Big(\frac{z}{\b(n)+C^{\star}+z}\Big)^{j}+
\sum_{j=1}^{+\infty}\frac{1}{j}\Big(\frac{z}{\b(n)+C^{\star}+z}\Big)^{j}.
\end{align*}
In particular, taking into account of the assumptions on $\b(n)$ and $z_1$, we deduce the equalities:
\begin{align*}
\sum_{j=1}^{+\infty}\frac{1}{(j+3)}\Big(\frac{z}{\b(n)+z}\Big)^{j}&=\sum_{j=1}^{+\infty}\frac{-3}{j(j+3)}\Big(\frac{z}{\b(n)+z}\Big)^{j}-\ln\Big(1-\frac{z}{\b(n)+z}\Big),\\
\sum_{j=1}^{+\infty}\frac{1}{(j+3)}\Big(\frac{z}{\b(n)+C^{\star}+z}\Big)^{j}&=\sum_{j=1}^{+\infty}\frac{1}{j(j+3)}\Big(\frac{z}{\b(n)+C^{\star}+z}\Big)^{j}-\ln\Big(1-\frac{z}{\b(n)+z+C^{\star}}\Big).
\end{align*}
Therefore, we end up with the equality
\begin{align*}
&\sum_{n=\mathfrak{N}\,+1}^{\infty}\Big[\b(n)\Big(\frac{z}{\b(n)+z}\Big)^{3}\sum_{j=1}^{+\infty}\frac{1}{j+3}\Big(\frac{z}{\b(n)+z}\Big)^{j}\\&
-
(\b(n)+C^{\star})\Big(\frac{z}{\b(n)+C^{\star}+z}\Big)^{3}\sum_{j=1}^{+\infty}\frac{1}{j+3}\Big(\frac{z}{\b(n)+C^{\star}+z}\Big)^{j}
\Big]\equiv\sum_{k=1}\mathfrak{D}_{k}(z),
\end{align*}
where
\begin{align*}
&\mathfrak{D}_1(z)=C^{\star}\sum_{n=\mathfrak{N}\,+1}^{\infty}\Big(\frac{z}{\b(n)+C^{\star}+z}\Big)^{3}\Big[\sum_{j=1}^{+\infty}\frac{3}{j(j+3)}\Big(\frac{z}{\b(n)+C^{\star}+z}\Big)^{j}+\ln\Big(1-\frac{z}{\b(n)+C^{\star}+z}\Big)\Big],\\
&\mathfrak{D}_2(z)=\sum_{n=\mathfrak{N}\,+1}^{\infty}\Big[\frac{\b(n)z^{3}}{(\b(n)+C^{\star}+z)^{3}}-\frac{\b(n)z^{3}}{(\b(n)+z)^{3}}\Big]
\Big[\sum_{j=1}^{+\infty}\frac{3}{j(j+3)}\Big(\frac{z}{\b(n)+z}\Big)^{j}+\ln\Big(1-\frac{z}{\b(n)+z}\Big)\Big],\\
&\mathfrak{D}_{3}(z)=\sum_{n=\mathfrak{N}\,+1}^{+\infty}\b(n)\Big(\frac{z}{\b(n)+C^{\star}+z}\Big)^{3}
\sum_{j=1}^{+\infty}\frac{3z^{j}}{j(j+3)}\Big(\frac{1}{(\b(n)+C^{\star}+z)^{j}}-\frac{1}{(\b(n)+z)^{j}}\Big),\\
&\mathfrak{D}_4(z)=\sum_{n=\mathfrak{N}\,+1}\b(n)\Big(\frac{z}{\b(n)+C^{\star}+z}\Big)^{3}\ln\frac{(\b(n)+C^{\star})(\b(n)+z)}{\b(n)(\b(n)+C^{\star}+z)}.
\end{align*}
At this point, we treat each term $\mathfrak{D}_{k}(z),$ separately.

\noindent$\bullet$ Collecting statement (ii) of  Proposition \ref{p3.1} with easily verified relations:
\begin{align}\label{A.5}\notag
&\Big|\frac{z}{\b(n)+C^{\star}+z}\Big|<1,\quad
\sum_{j=1}^{+\infty}\frac{3}{j(j+3)}\Big|\frac{z}{\b(n)+C^{\star}+z}\Big|^{j}<\sum_{j=1}^{+\infty}\frac{C}{j^{2}}=C,\\
&\Big|\frac{z}{\b(n)+C^{\star}+z}\ln\Big(1-\frac{z}{\b(n)+z+C^{\star}}\Big)\Big|<C,
\end{align}
with the positive constant $C$ being independent of $n$ and $z$, we end up with
\[
|\mathfrak{D}_1(z)|\leq C|z|+O(1).
\]
$\bullet$ Performing simple technical  calculations, we conclude
\[
\mathfrak{D}_{2}(z)=-\sum_{n=\mathfrak{N}\,+1}^{+\infty}\frac{C^{\star}\b(n)z^{3}[(\b(n)+z)^{2}+(\b(n)+z)(\b(n)+C^{\star}+z)+(\b(n)+C^{\star}+z)^{2}]}
{(\b(n)+z)^{3}(\b(n)+C^{\star}+z)^{3}}\]
\[
\times\Big[\sum_{j=1}^{+\infty}\frac{3}{j(j+3)}\Big(\frac{z}{\b(n)+z}\Big)^{j}+\ln\Big(1-\frac{z}{\b(n)+z}\Big)\Big].
\]
After that, inequalities \eqref{A.5} and statement (ii) in Proposition \ref{p3.1} arrive at the inequality
\[
|\mathfrak{D}_{2}(z)|\leq C|z| +O(1).
\]
$\bullet$ Since $\Big|\frac{C^{\star}}{\b(n)+z}\Big|<<1$, we have
\begin{align*}
\frac{1}{(\b(n)+C^{\star}+z)^{j}}-\frac{1}{(\b(n)+z)^{j}}&=\frac{1}{(\b(n)+C^{\star}+z)^{j}}\Big[1-\Big(1+\frac{C^{\star}}{\b(n)+z}\Big)^{j}\Big]\\ &
=\frac{-1}{(\b(n)+C^{\star}+z)^{j}}\Big[\frac{C^{\star}}{\b(n)+z}+(j-1)O\Big(\frac{1}{(\b(n)+z)^{2}}\Big)\Big]
\end{align*}
and, accordingly,
\begin{align*}
&\sum_{j=1}^{\infty}\frac{z^{j}}{j(j+3)}\Big[\frac{1}{(\b(n)+C^{\star}+z)^{j}}-\frac{1}{(\b(n)+z)^{j}}\Big]\\&=
\sum_{j=1}^{\infty}\frac{z^{j}}{j(j+3)(\b(n)+C^{\star}+z)^{j}}\Big[\frac{C^{\star}}{\b(n)+z}+(j-1)O\Big(\frac{1}{(\b(n)+z)^{2}}\Big)\Big]
\\
&
=
\frac{C^{\star}}{\b(n)+z}\sum_{j=1}^{\infty}\frac{z^{j}}{j(j+3)(\b(n)+C^{\star}+z)^{j}}\Big[1-O\Big(\frac{1}{(\b(n)+z)}\Big)\Big]
\\&+\sum_{j=1}^{\infty}\frac{z^{j}}{(j+3)(\b(n)+C^{\star}+z)^{j}}O\Big(\frac{1}{(\b(n)+z)^{2}}\Big)\\&
=
\frac{C^{\star}}{\b(n)+z}\sum_{j=1}^{\infty}\frac{z^{j}}{j(j+3)(\b(n)+C^{\star}+z)^{j}}\Big[1-O\Big(\frac{1}{(\b(n)+z)^{2}}\Big)\Big]
\\&-O\Big(\frac{1}{(\b(n)+z)^{2}}\Big)
\sum_{j=1}^{\infty}\Big[\frac{3z^{j}}{j(j+3)(\b(n)+C^{\star}+z)^{j}}-
\ln\Big(1-\frac{z}{\b(n)+C^{\star}+z}\Big)\Big].
\end{align*}
Finally, coming to $\mathfrak{D}_3(z)$ and taking into account \eqref{A.5}, assumptions in Proposition \ref{p3.1} and aforementioned relations, we deduce
\[
|\mathfrak{D}_{3}(z)|\leq C|z|+O(1).
\]

\noindent $\bullet$ In order to manage $\mathfrak{D}_{4}(z)$, we take advantage of the identities
\[
\ln\frac{(\b(n)+z)(\b(n)+C^{\star})}{(\b(n)+z+C^{\star})\b(n)}=
\ln\frac{\bigg(1+\frac{C^{\star}}{\b(n)}\bigg)}{\bigg(1+\frac{C^{\star}}{\b(n)+z}\bigg)}
=
\frac{C^{\star}}{\b(n)}+O\Big(\frac{1}{\b^{2}(n)}\Big)-\frac{C^{\star}}{\b(n)+z}+O\Big(\frac{1}{(\b(n)+z)^{2}}\Big)
\]
which in turn provide
\[
\mathfrak{D}_4(z)=\sum_{n=\mathfrak{N}\,+1}C^{\star}\Big[\b(n)\{O((\b(n)+z)^{-2})+O(\b^{-2}(n))\}-\frac{z}{\b(n)+z}\Big]\Big(\frac{z}{\b(n)+C^{\star}+z}\Big)^{3}.
\]
At last, applying statements \textbf{(i)} and  \textbf{(ii)} of Proposition \ref{p3.1} to the right-hand side of the last equality, we obtain the desired bound
\[
|\mathfrak{D}_{4}(z)|\leq C|z|+O(1).
\]
Collecting representations of all $\mathfrak{D}_{k}$, $k=1,2,3,4,$ we arrive at the inequality in (iv) of this proposition. Thus, the proof is completed. \qed


\section*{Data Availability Statements.}
 Data sharing not applicable to this article as no datasets were generated or analyzed during the current study.



\end{document}